\documentclass[article,11pt,final]{amsart}
\usepackage[top=25mm,bottom=25mm,left=25mm,right=25mm]{geometry}
\usepackage{amsmath,amsthm,amsfonts,amssymb}
\usepackage{mathrsfs,dsfont}
\usepackage{graphicx}
\usepackage{float}
\usepackage{enumitem}
\usepackage{xcolor}
\usepackage{marginnote}

\newtheorem{lemma}{Lemma}[section]

\newtheorem{proposition}{Proposition}[section]
\newtheorem{theorem}{Theorem}[section]
\newtheorem{corollary}{Corollary}[section]
\newtheorem{definition}{Definition}[section]

\newtheorem{example}{Example}[section]

\makeatletter
\def\section{\@startsection{section}{1}%
\z@{1\linespacing\@plus\linespacing}{1\linespacing}%
{\bf\centering}}
\def\subsection{\@startsection{subsection}{0}%
\z@{\linespacing\@plus\linespacing}{\linespacing}%
{\bf}}
\makeatother

\makeatletter
\@addtoreset{equation}{section}
\makeatother

\DeclareMathOperator{\spec}{spec}

\DeclareMathOperator{\Id}{Id}

\newcommand{\cD}{\mathcal{D}}
\newcommand{\cA}{\mathcal{A}}
\newcommand{\cL}{\mathcal{L}}
\newcommand{\cK}{\mathcal{K}}

\newcommand{\cB}{\mathcal{B}}
\newcommand{\cC}{\mathcal{C}}
\newcommand{\cX}{\mathcal{X}}

\newcommand{\cV}{\mathcal{V}}

\newcommand{\cM}{\mathcal{M}}

\newcommand{\R}{\mathbb{R}}
\newcommand{\Z}{\mathbb{Z_+}}

\newcommand{\1}{\mathbf{1}}

\begin{document}
\title[Density of states on nested fractals]
{Density of states for the Anderson model on nested fractals}

\author{Hubert Balsam, Kamil Kaleta, Mariusz Olszewski and Katarzyna Pietruska-Pa{\l}uba}
\address{H. Balsam and K. Pietruska-Pa{\l}uba \\ Institute of Mathematics \\ University of Warsaw
\\ ul. Banacha 2, 02-097 Warszawa, Poland}
\email{hubert.balsam@gmail.com, kpp@mimuw.edu.pl}

\address{K. Kaleta and M. Olszewski \\ Faculty of Pure and Applied Mathematics, Wroc{\l}aw University of Science and Technology, Wyb. Wyspia\'nskiego 27, 50-370 Wroc{\l}aw, Poland}
\email{kamil.kaleta@pwr.edu.pl, mariusz.olszewski@pwr.edu.pl}

\begin{abstract}
{We prove the existence and establish the Lifschitz singularity of the integrated density of states for certain random Hamiltonians $H^\omega=H_0+V^\omega$ on fractal spaces of infinite diameter. The kinetic term  $H_0$ is given by $\phi(-\mathcal L),$ where $\mathcal L$ is the Laplacian on the fractal and $\phi$ is a completely monotone function satisfying some  mild regularity conditions.  The random potential $V^\omega$ is of alloy-type.  }


\bigskip
\noindent
2010 {\it MS Classification}: Primary: 82B44, 28A80, 60K37; Secondary: 47D08, 60J45, 60J57.
\end{abstract}

\footnotetext{Research was supported in parts by the National Science Centre, Poland, grants no.\ 2015/17/B/ST1/01233 and 2019/35/B/ST1/02421.}

\maketitle

\baselineskip 0.5 cm

\bigskip\bigskip

\section{Introduction}

Let $\cK^{\langle \infty \rangle}$ be a planar unbounded simple nested fractal (USNF in short) with the Good Labeling Property (GLP in short) and let $\cL$ be the associated Laplacian, i.e.\ the generator of the heat semigroup on $L^2(\cK^{\langle \infty \rangle}, \mu)$ (or, equivalently, of the Brownian motion with values in $\cK^{\langle \infty \rangle}$). Here $\mu$ denotes the normalized Hausdorff measure on the fractal. We propose and study the random Anderson-type Schr\"odinger operator
\[
H^{\omega} = H_0 + V^{\omega}, \quad \text{acting in} \quad L^2(\cK^{\langle \infty \rangle}, \mu),
\]
where the kinetic term $H_0$  takes the form
\[
H_0 = \phi(-\cL),
\]
for a sufficiently regular operator monotone function $\phi$, and $V^{\omega}$ is the operator of multiplication by a function that we call a fractal alloy-type potential; it is built as follows:
\[
V^{\omega}(x) := \sum_{v \in \cV_0^{\langle \infty \rangle}} \xi_v(\omega) \cdot W(x,v), \quad x \in \cK^{\langle \infty \rangle},
\]
where $\cV_0^{\langle \infty \rangle}$ denotes the set of vertices of the self-similar fractal lattice spanning $\cK^{\langle \infty \rangle}$,
$W$ is a non-negative single site (two-argument) potential on $\cK^{\langle \infty \rangle} \times \cV_0^{\langle \infty \rangle}$, and $\{\xi_v: v\in \cV_0^{\langle \infty \rangle}\}$ are nonnegative and nondegenerate i.i.d.\ random variables over
a given probability space $\left(\Omega, \cM, \mathbb{Q}\right)$.

The main goal of this paper is to establish  the existence and then to study asymptotic properties of the \emph{integrated density of states} (IDS) of the operator $H^{\omega}$. In the setting of planar USNF's with the GLP, we give general sufficient conditions on $\phi$, $W$, and $\xi_v$'s under which the IDS exists and exhibits the Lifschitz-type behaviour at the bottom of the spectrum of $H^{\omega}$.

The random Schr\"odinger operator $H^{\omega}$ can be interpreted as the Hamiltonian in the mathematical model of  certain quantum system -- namely the motion of a single particle (an electron) in a perfectly ordered material (say crystal) with impurities and defects (the so-called disordered medium). This material is modelled by the nested fractal -- the  ions are located at the vertices of the fractal lattice $\cV_0^{\langle \infty \rangle}$, and the disorder is modelled by the alloy-type potential $V^{\omega}$. The model we propose is a fractal counterpart of the classical \emph{Anderson model} in which the ions were placed at the integer lattice $\mathbb{Z}^d$ (see \cite{bib:Stoll} and references therein). To the best of our knowledge models based on such fractal alloy-type potentials have not been studied yet. We remark that the operator $\phi(-\cL)$ describes the kinetic energy of the free particle. In this paper we consider the functions $\phi(\lambda)$ that are comparable to $\lambda^{\alpha/d_w}$ at zero (here $\alpha \in (0,d_w]$, where $d_w$ is the walk dimension of $\cK^{\langle \infty \rangle}$), and grow sufficiently fast at infinity. For precise statement see assumption {\bf (B)} in Section \ref{sec:subord}. Our framework covers the \emph{non-relativistic} models with $\phi(\lambda)=\lambda$ (i.e. Hamiltonians  based on the Laplacian) as well as the \emph{relativistic-type} models with $\phi(\lambda) = (\lambda+m^{2/\alpha})^{\alpha/2}-m$, $\alpha \in (0,2)$, $m \geq 0$ (i.e.\ based on non-local kinetic term operators).

We now summarize  main contributions of the paper.

\smallskip
\noindent
(1) \emph{Existence of the IDS.} We give fairly general sufficient conditions on random variables $\xi_v$ and the single site potential $W$
(for detailed statements see assumptions {\bf (Q1)} and {\bf (W1)}-{\bf (W3)} in Section \ref{sec:Ratp_and_FKS}) under which the IDS of the operator $H^{\omega}$ (denoted by $\mathcal N$) exists. It is obtained as the vague limit of the normalized measures counting the eigenvalues of the corresponding Dirichlet and Neumann random Schr\"odinger operators acting on $L^2(\mathcal{K}^{\left\langle M\right\rangle},\mu)$, where $\mathcal{K}^{\left\langle M\right\rangle}$ are bounded fractals which approximate $\cK^{\langle \infty \rangle}$ as $M \to \infty$. This result is stated as Theorem \ref{thm:IDS}.

\smallskip
\noindent
(2) \emph{Lifschitz tail of the IDS.} Under  additional assumptions on the distribution function $F_{\xi}$ of the random variables $\xi_v$ and on the single site potential $W$ (see {\bf (Q2)}, {\bf (W4)} and {\bf (W5)} in Section \ref{sec:LT}; in particular, we assume that $W$ is of finite range), we describe the asymptotic behaviour of $\mathcal N(\lambda):=\mathcal N[0,\lambda]$ as $\lambda \to 0^+$. More precisely, our main result, Theorem \ref{th:Lifshitz-IDS}, states that there are positive constants $K,\widetilde K, R, \widetilde R$
such that
\[
-K\leq \liminf_{\lambda\searrow 0}\frac{\lambda^{ \frac{d}{\alpha}}\log \mathcal N(\lambda)}{g(R/\lambda)} \quad \text{and} \quad
\limsup_{\lambda\searrow 0}\frac{\lambda^{ \frac{d}{\alpha}}\log \mathcal N(\lambda)}{g(\widetilde R/\lambda)}\;\leq \; -\widetilde K,
\]
where $g(r) := -\log F_{\xi}(D_0/r)$, for a certain constant $D_0>0$,  $F_\xi$ being the cumulative distribution functions of the $\xi_v$'s.  Clearly, when $\xi_v$'s have an atom at zero, then $\lim_{\lambda\searrow 0} g(D_0/\lambda)$ exists and is finite. Consequently, the rates above simplify  in that case. 

\smallskip

The Lifschitz singularity of the IDS of a random Schr\"{o}dinger is a strong indication that the system exhibits the so-called spectral localisation (see e.g.\ Combes and Hislop \cite{bib:Com-His}, Bourgain and Kenig \cite{bib:Bou-Ken}, Germinet, Hislop and Klein \cite{bib:Ger-His-Kle}, and their references). Rigorous proofs of the localisation property often rely on the approximation of the IDS via the Lifschitz tail (see e.g. papers by Klopp \cite{bib:Klopp1,bib:Klopp2} and Kirsch-Veseli\'{c} \cite{bib:KV}). In this context, it is worth mentioning that the Lifschitz tail in the classical continuous setting (on $\mathbb R^d)$ has been established for various types of random potentials (cf.\
Benderskii and Pastur \cite{bib:BP}, Friedberg and Luttinger \cite{bib:FL}, Luttinger \cite{bib:Lut}, Nakao \cite{bib:Nak}, Pastur \cite{bib:Pas}, Kirsch and Martinelli \cite{bib:KM1, bib:KM2}, Mezincescu \cite{bib:Mez}, Kirsch and Simon \cite{bib:KS}, Kirsch and Veseli\'c \cite{bib:KV}) and in the  discrete setting (on $\mathbb Z^d$) e.g.\ by Fukushima \cite{bib:F}, Fukushima, Nagai and Nakao \cite{bib:FNN}, Nagai \cite{bib:Nag}, Romerio and Wreszinski \cite{bib:RW}, Simon \cite{bib:Sim}. The kinetic term $H_0$ in those papers is the Laplace operator - continuous or discrete.
In this regard, nonlocal kinetic terms on $\R^d$ and $\mathbb{Z}^d$ were considered by Okura \cite{bib:Okura}, Kaleta and Pietruska-Pa{\l}uba \cite{bib:KaPP4,bib:KaPP3}, and Gebert and Rojas-Molina \cite{bib:GRM}.


\smallskip

On fractals, the problem in question has been studied  in a very restricted scope.  Pietruska-Pa{\l}uba \cite{bib:KPP-PTRF} established the existence and the Lifschitz tail of the IDS for the Brownian motion evolving among Poissonian killing obstacles (which corresponds to the Laplacian $\cL$ with Dirichlet conditions on the set of obstacles) on the planar Sierpi\'nski gasket.  Shima \cite{bib:Sh} extended this result to general nested fractals and to random Schr\"odinger operators based on $\cL$ with Poissonian random fields whose two-argument shape functions is of finite range. More recently,  Kaleta and Pietruska-Pa{\l}uba \cite{bib:KaPP2} have  studied the existence of the IDS for random Schr\"odinger operators $\phi(-\cL)+V^{\omega}$, for a large class of operator monotone functions $\phi$ and Poissonian potentials with fairly general shape functions, possibly of infinite range, on the unbounded planar Sierpi\'nski gasket. That model corresponds to subordinate Brownian motions evolving in presence of the Poissonian potential; at the level of operators it allows for both local and non-local models. The Lifschitz singularity of the IDS was obtained in \cite{bib:KaPP}, but the approach of that paper required some restrictions on the behaviour of $\phi$ at zero:\ it excluded non-local models with $\phi(\lambda) \approx \lambda$ as $\lambda \searrow 0$, e.g.\ the relativistic ones with the mass $m>0$.

Our present paper extends the above results  in at least two different directions, leading to a  substantial   progress in the field. First of all, all the above investigations were restricted to  Poissonian random media, while now we treat the existence and the asymptotics of the IDS for the fractal-alloy type potentials on fairly general planar nested fractals. Such a model is new even in the case of the Sierpi\'nski gasket. Moreover, we cover an essentially wider class of operator monotone functions of the Laplacian $\phi(-\cL)$, including relativistic kinetic terms. We mention that the results in Theorem \ref{th:Lifshitz-IDS} constitute  a fractal counterpart of \cite[Theorem 1.1]{bib:KaPP3}.

Analysis on fractals  is still a vivid branch of mathematics,  see e.g.\ the paper of Kigami \cite{bib:Kig4}, and the series of papers of Alonso-Ruiz et al.\ \cite{bib:ARBCRST1,bib:ARBCRST2,bib:ARBCRST3,bib:ARBCRST4}), concerned in particular with function spaces on fractals.  Also, fractal-type spaces and sets with fractal boundaries became more and more popular in modeling  -- e.g.\ in the theory of PDE's (Chen et al.\ in \cite{bib:CHT} and Dekkers et al. \cite{bib:DRT}), information theory  (Akkermans et al.\ \cite{bib:ACDRT}), magnetostatics problems (Hinz and Teplyaev \cite{bib:HT}), acoustics theory (Hinz et al.\ \cite{bib:HRT}), theory of random walks on these sets (e.g.\  Kumagai and  Nakamura \cite{bib:KN}). See also the paper of Kumagai \cite{bib:Kum7} for an overview on diffusions on disordered  media.
These are just some examples of possible applications of theory of fractals, chosen from the very rich  recent literature of the subject.

Let us now say a few words about the proofs of our main results. In the case of homogeneous spaces, which are translation invariant, say $\R^d$, and alloy-type potentials which are based on the integer lattice $\mathbb{Z}^d$, the existence and the asymptotics of the IDS can be studied by using the classical notions of stationarity and ergodicity of such random fields (see e.g.\ \cite[Theorem VI.1.1 and Remark VI.1.2]{bib:Car-Lac}). For less regular  spaces such as fractals, this problem becomes more difficult and requires a specialized approach. Our proof of Theorem \ref{thm:IDS} follows a general approach from \cite{bib:KaPP2} which is a modification of the argument that was originally performed for the Brownian motion evolving among Poisson obstacles on hyperbolic spaces \cite{bib:Szn2,bib:Szn3} and on the Sierpi\'nski triangle \cite{bib:KPP-PTRF}. Roughly speaking, such an approach is based on an approximation of the operator $H^{\omega}$ acting in $L^2(\cK^{\langle \infty\rangle}, \mu)$ by the operators defined for $L^2(\cK^{\langle M \rangle}, \mu)$ as $M \nearrow \infty$. The key step is concerned with the existence of a system of consistent labels of the vertices, which means that the fractal possesses a sufficiently regular periodic structure. The method of \cite{bib:KaPP2} works well for a wide class of subordinate Brownian motions evolving in presence of a fairly general Poisson random field, but it was implemented for the Sierpi\'nski triangle only. That fractal is one of the most regular,  planar USNF's; at that point, an extension to fractals with richer geometry was an open problem. The proof of Theorem \ref{th:Lifshitz-IDS} follows the ideas of \cite{bib:KaPP3} -- it adapts the technique from $\R^d$ to the nested fractal setting. The method has a functional analytic nature; one of key steps is based on an application of the Temple inequality -- the idea to use it in the context of the alloy-type random media comes from the classical papers of Simon \cite{bib:Sim}, and Kirsch and Simon \cite{bib:KS}. In order to apply this tool, one needs to reduce the problem to study the Schr\"odinger operators with periodic potentials \cite{bib:KS} or, equivalently, to project the initial operators on tori \cite{bib:KaPP3}. What is really crucial there from our point of view is that the Euclidean spaces have a natural periodic structure which allows one to make such a reduction.

Therefore, we see that the possibility of implementation of the ideas described above for more general nested fractals  relies   on their specific in-built regularity,  which can be described in short as the existence of a regular enough \emph{periodic structure}. For planar USNF's this problem has been solved just recently by Kaleta, Olszewski and Pietruska-Pa\l uba \cite{bib:KOP} who introduced the notion of the \emph{Good Labeling Property} and described the class of fractals having such a regularity (see also Nieradko and Olszewski \cite{bib:NO}). That allowed the authors to define the folding projection $\pi_M: \cK^{\langle \infty\rangle} \to \cK^{\langle M \rangle}$, $M \in \Z$, and construct the sequence of conservative diffusions on bounded fractals $\cK^{\langle M \rangle}$ with precisely established relations between the processes on consecutive levels. We want to emphasize that these tools are critical for the proof of Theorems \ref{thm:IDS} and \ref{th:Lifshitz-IDS}.

The organization of the material in the paper is as follows.
In Section 2 we have collected the preliminaries needed in subsequent sections -- essentials on  fractals, stochastic processes on fractal sets and their generators, random Schr\"{o}dinger operators. In particular in Proposition \ref{prop:kato} we prove that the random potentials we consider belong to Kato classes of respective processes, which allow to correctly express the random semigroups by the Feynman--Kac formula. The proof of the existence of the IDS is given in Section 3. Some of the technical lemmas were moved to the Appendix. Section 4 contains the proof of the Lifschitz singularity. We work with the Laplace transform of the IDS; we prove separately the upper (Theorem \ref{th:main-alloy-type-fractals}) and lower bounds (Theorem \ref{th:main-2}), which are then transformed into estimates for the IDS at zero (Theorem \ref{th:Lifshitz-IDS}). The tools that are crucial for development in Section~4 (Temple inequality, Bernstein-type estimates, Tauberian theorems) are formulated as needed.

\section{Preliminaries}

\subsection{Unbounded simple nested fractals with Good Labeling Property} \label{sec:usnf}

In this section we recall some basic definitions concerning nested fractals from \cite{bib:Lin,bib:Kum,bib:kpp-sausage,bib:kpp-sto}, and labeling properties from \cite{bib:KOP}. Consider a collection of similitudes $\Psi_i : \mathbb{R}^2 \to \mathbb{R}^2$ with a common scaling factor $L>1,$  and a common isometry part $U,$ i.e. \ $\Psi_i(x) = (1/L) U(x) + \nu_i,$  where  $\nu_i \in \mathbb{R}^2$, $i \in \{1, ..., N\}.$ We shall assume $\nu_1 = 0$.
There exists a unique nonempty compact set $\mathcal{K}^{\left\langle 0\right\rangle}$ (called {\em the  fractal generated by the system} $(\Psi_i)_{i=1}^N$) such that $\mathcal{K}^{\left\langle 0\right\rangle} = \bigcup_{i=1}^{N} \Psi_i\left(\mathcal{K}^{\left\langle 0\right\rangle}\right)$.  As $L>1$, each similitude has exactly one fixed point so there are  $N$ fixed points of the transformations $\Psi_1, ..., \Psi_N$ (not necessarily distinct).  Let $F$ be the collection of those fixed points.

A fixed point $x \in F$ is an essential fixed point if there exists another fixed point $y \in F$ and two different similitudes $\Psi_i$, $\Psi_j$ such that $\Psi_i(x)=\Psi_j(y)$.
The set of all essential fixed points for transformations $\Psi_1, ..., \Psi_N$ is denoted by $\cV_{0}^{\left\langle 0\right\rangle}$. The set $\mathcal{K}^{\left\langle 0\right\rangle}$ is called a \emph{nested fractal} if the following conditions are met.
\begin{enumerate}
\item $k:=\# \cV_{0}^{\left\langle 0\right\rangle} \geq 2.$
\item (Open Set Condition) There exists an open set $U \subset \mathbb{R}^2$ such that for $i\neq j$ one has\linebreak $\Psi_i (U) \cap \Psi_j (U)= \emptyset$ and $\bigcup_{i=1}^N \Psi_i (U) \subseteq U$.
\item (Nesting) $\Psi_i\left(\mathcal{K}^{\left\langle 0 \right\rangle}\right) \cap \Psi_j \left(\mathcal{K}^{\left\langle 0 \right\rangle}\right) = \Psi_i \left(\cV_{0}^{\left\langle 0\right\rangle}\right) \cap \Psi_j \left(\cV_{0}^{\left\langle 0\right\rangle}\right)$ for $i \neq j$.
\item (Symmetry) For $x,y \in \cV_{0}^{\left\langle 0\right\rangle},$ let $S_{x,y}$ denote the symmetry with respect to the line bisecting the segment $\left[x,y\right]$. Then
\begin{equation*}
\forall i \in \{1,...,M\} \ \forall x,y \in \cV_{0}^{\left\langle 0\right\rangle} \ \exists j \in \{1,...,M\} \ S_{x,y} \left( \Psi_i \left(\cV_{0}^{\left\langle 0\right\rangle} \right) \right) = \Psi_j \left(\cV_{0}^{\left\langle 0\right\rangle} \right).
\end{equation*}
\item (Connectivity) On the set $\cV_{-1}^{\left\langle 0\right\rangle}:= \bigcup_i \Psi_i \left(\cV_{0}^{\left\langle 0\right\rangle}\right)$ we define graph structure $E_{-1}$ as follows:\\
$(x,y) \in E_{-1}$ if and only if $x, y \in \Psi_i\left(\mathcal{K}^{\left\langle 0 \right\rangle}\right)$ for some $i$.
Then the graph $(\cV_{-1}^{\left\langle 0\right\rangle},E_{-1} )$ is required to be connected.
\end{enumerate}

As
for $k = 2$ the set $\mathcal{K}^{\left\langle 0\right\rangle}$ is just a line segment between two essential fixed points, we
require $k\geq 3$. In this case the points from $\cV_{0}^{\left\langle 0\right\rangle}$ are the vertices of a regular polygon \cite[Proposition 2.1]{bib:KOP}.

The remaining notions are collected in a single definition.
\begin{definition} Let $M\in\mathbb{Z}.$
\begin{itemize}
\item[(1)]
\begin{equation} \label{eq:Kn}
\mathcal{K}^{\left\langle M\right\rangle} = L^M \mathcal{K}^{\left\langle 0\right\rangle}
  = \Psi^{-M}_1(\mathcal K^{\langle 0\rangle}),
\end{equation}
 is the fractal of size $L^M$
(the $M$-complex  attached to $(0,0)$, see \eqref{eq:Mcompl} below).

\item[(2)]
\begin{equation} \label{eq:Kinfty}
\mathcal{K}^{\left\langle \infty \right\rangle} = \bigcup_{M=0}^{\infty} \mathcal{K}^{\left\langle M\right\rangle}.
\end{equation}
is the unbounded simple nested fractal (USNF).
\item[(3)] $M$-complex: \label{def:Mcomplex}
every set $\Delta_M \subset \mathcal{K}^{\left\langle \infty \right\rangle}$ of the form
\begin{equation} \label{eq:Mcompl}
\Delta_M  = \mathcal{K}^{\left\langle M \right\rangle} + \nu_{\Delta_M},
\end{equation}
where $\nu_{\Delta_M}=\sum_{j=M+1}^{J} L^{j} \nu_{i_j},$ for some $J \geq M+1$, $\nu_{i_j} \in \left\{\nu_1,...,\nu_N\right\}$, is called an \emph{$M$-complex}.
\item[(4)] Vertices of the $M$-complex \eqref{eq:Mcompl}: the set $\cV\left(\Delta_M\right) =L^M\cV_0^{\langle 0 \rangle}+\nu_{\Delta_M}= L^{M} \mathcal V^{\left\langle 0 \right\rangle}_0 + \sum_{j=M+1}^{J} L^{j} \nu_{i_j}$.
\item[(5)] Vertices of $\mathcal{K}^{\left\langle M \right\rangle}$:
$$
\mathcal V^{\left\langle M\right\rangle}_{M} = \cV\left(\mathcal{K}^{\left\langle M \right\rangle}\right) = L^M \mathcal V^{\left\langle 0\right\rangle}_{0}.
$$
\item[(6)] Vertices of all $M$-complexes inside an $(M+m)$-complex (defined recursively for $m>0$):
$$
\cV_M^{\langle M+m\rangle}= \bigcup_{i=1}^{N} \cV_M^{\langle M+m-1\rangle} + L^{M+m} \nu_i.
$$
\item[(7)] Vertices of all 0-complexes inside the unbounded nested fractal:
$$
\mathcal V^{\left\langle \infty \right\rangle}_{0} = \bigcup_{M=0}^{\infty} \mathcal V^{\left\langle M\right\rangle}_{0}.
$$
\item[(8)] Vertices of $M$-complexes from the unbounded fractal:
$$
\mathcal V^{\left\langle \infty \right\rangle}_{M} = L^{M} \mathcal V^{\left\langle \infty \right\rangle}_{0}.
$$
\item[(9)]$\cC_M(x)$:  the union of all $M-$complexes containing $x$:\\
-- the unique $M-$complex when $x\notin \cV_M^{\langle \infty \rangle},$\\
-- the union of all $M-$complexes intersecting at $x$ when $x\in \cV_M^{\langle \infty\rangle}.$ Their number is denoted by ${\rm rank}\,(x)$.
\item[(10)] $M$-complexes in $\cK^{\langle M+1 \rangle}$:
$$
\Delta_{M,i} = \cK^{\langle M \rangle} + L^M \nu_i, \quad i= 1,...,N.
$$
\end{itemize}
\end{definition}

The fractal (Hausdorff) dimension of $\mathcal{K}^{\left\langle \infty \right\rangle}$ is equal to  $d=\frac{\log N}{\log L}$. The Hausdorff  measure in  dimension $d$ with support $\mathcal K^{\langle \infty\rangle}$ will be denoted by $\mu$. It will be normalized so as to have $\mu\left(\mathcal{K}^{\left\langle 0\right\rangle}\right)=1$. It serves as a `uniform' measure on $\mathcal{K}^{\left\langle \infty \right\rangle}.$

\begin{definition}\label{def:graph-distance}
For $M \in \mathbb Z$ and $x,y \in \mathcal{K}^{\left\langle \infty \right\rangle}$ let
\begin{equation}
\label{eq:graphmetric}
d_M (x,y):= \left\{ \begin{array}{ll}
0, & \textrm{if } x=y ;\\[1mm]
1, & \textrm{if there exists } \Delta_M \in \mathcal{T}_M \textrm{ such that } x,y \in \Delta_M \textrm{ and } x \neq y ;\\[1mm]
n>1, & \textrm{if there does not exist } \Delta_M \in \mathcal{T}_M \textrm{ such that } x,y \in \Delta_M \textrm{ and } n \textrm{ is the smallest}\\
&  \textrm{number for which there exists a chain } \Delta_M^{(1)}, \Delta_M^{(2)}, ..., \Delta_M^{(n)} \in \mathcal{T}_M \textrm{ such that } x \in \Delta_M^{(1)},\\
&  y \in \Delta_M^{(n)} \textrm{ and } \Delta_M^{(i)} \cap \Delta_M^{(i+1)} \neq \emptyset \textrm{ for  }1 \leq i \leq n-1.
\end{array} \right.
\end{equation}
\end{definition}
The set $\mathcal C_M(x)$ defined above is just the ball in $d_M$ metric with center $x$ and radius $1$.

Let
\begin{equation}\label{eq:def-r0}
r_0:=\max\,\mbox{rank}(v),\qquad v\in \cV_0^{\langle\infty\rangle}.
\end{equation}
For $k\geq 4$ there is $r_0 =2$. For $k=3$ we have $r_0 \in \{2,3\}$. In case of Sierpi\'nski gasket $r_0 =2$, but in Example 3.2 in \cite{bib:KOP} there is $r_0=3$.

We also need the estimate on the number of points from the $0-$grid inside $\mathcal K^{\langle M\rangle},$ for $M\in\mathbb Z_+,$ i.e. the cardinality of $\cV_0^{\langle M\rangle}.$
Denote this number by $k_0^{\langle M\rangle}.$
We readily see that there exists a constant $C_0>1$ such that
\begin{equation}\label{eq:number-points}
L^{ Md}\leq \# \mathcal V^{\langle M\rangle}_0 = k_0^{\langle M\rangle} \leq C_0 L^{ Md}.
\end{equation}

We now introduce the concept of good labeling of vertices.

Consider the alphabet of $k$ symbols $\cA:=\left\{a_1, a_2,a_3,...,a_k\right\}$,  where  $ k=\# \mathcal V^{\left\langle 0\right\rangle}_{0}\geq 3.$ The elements of $\cA$ will be called labels.  A \emph{labelling function of order $M \in \mathbb{Z}$} is any map $l_M: \mathcal V^{\left\langle \infty \right\rangle}_{M} \to \cA.$ Recall that the vertices of every $M$-complex $\Delta_M$ are the vertices of a regular polygon with $k$ vertices. In particular, there exist exactly $k$ different rotations around the barycenter of $\mathcal K^{\langle M \rangle}$,  mapping $\mathcal V^{\left\langle M \right\rangle}_{M}$ onto $\mathcal V^{\left\langle M \right\rangle}_{M}.$ They will be denoted by  $\{R_1, ..., R_k\}=: \mathcal{R}_M $ (the rotations are ordered in such a way that for $i=1,2,...,k,$ the rotation $R_i$ rotates by angle $\frac{2\pi i}{k})$.

\begin{definition}[\textbf{Good labelling function of order $M$}]
\label{def:glp} Let $M \in \mathbb{Z}$.  A function $\ell_M: \mathcal V^{\left\langle \infty \right\rangle}_{M} \to \cA$  is called a \emph{good labelling function of order $M$} if the following conditions are met.
\begin{itemize}
\item[(1)] The restriction of $\ell_M$ to $\mathcal V^{\left\langle M \right\rangle}_{M}$ is a bijection onto $\cA$.
\item[(2)] For every $M$-complex $\Delta_M$ represented as
$$
\Delta_M  = \mathcal{K}^{\left\langle M \right\rangle} +\nu_{\Delta_M},
$$
where $\nu_{\Delta_M}=  \sum_{j=M+1}^{J} L^{j} \nu_{i_j},$  with some $J \geq M+1$ and $\nu_{i_j} \in \left\{\nu_1,...,\nu_N\right\}$ (cf. Def. \ref{def:Mcomplex}), there exists a rotation $R_{\Delta_M} \in \mathcal{R}_M$ such that
\begin{align}\label{eq:rot}
\ell_M(v)=\ell_M\left(R_{\Delta_M}\left(v -\nu_{\Delta_M}\right)\right) , \quad v \in \cV\left(\Delta_M\right).
\end{align}
\end{itemize}
\end{definition}
 An USNF $\mathcal{K}^{\left\langle \infty \right\rangle}$ is said to have the  \emph{good labelling property of order $M$} if a good labelling function of order $M$ exists.

Thanks to the selfsimilar structure of $\mathcal{K}^{\left\langle \infty \right\rangle},$ the \emph{good labelling property of order $M$} for some $M \in \mathbb{Z}$ is equivalent to this property of any other order $M'  \in \mathbb{Z}$. This gives rise to the following general definition.

\begin{definition}[\textbf{Good labelling property}] \label{def:glp_gen} An USNF $\mathcal{K}^{\left\langle \infty \right\rangle}$ is said to have the \emph{good labelling property (GLP in short)} if it has the good labelling property of order $M$ for some $M \in \mathbb{Z}$.
\end{definition}

In other words, the fractal $\mathcal{K}^{\left\langle \infty \right\rangle}$ has the GLP iff the vertices from $\mathcal V^{\left\langle \infty \right\rangle}_{M}$ can be labeled in such way that each $M$-complex has the complete set of labels attached to its vertices and the order of label is preserved between $M$-complexes.

It is known that in the case of $k$  being  prime or $k=2^n$ the fractal has the GLP. For other values of $k$ one can construct examples having GLP and different examples not having that property -- it depends on the structure of $0$-complexes inside $1$-complexes.  For a detailed analysis of GLP, including the geometrical characterization of that property, see \cite{bib:KOP, bib:NO}.

\medskip

For an unbounded fractal $\mathcal{K}^{\left\langle \infty \right\rangle}$ having the GLP, we define a projection map $\pi_{M}$ from $\mathcal{K}^{\left\langle \infty \right\rangle}$  onto the primary $M$-complex $\mathcal{K}^{\left\langle M \right\rangle}$ by the formula
\begin{equation}
\label{eq:piem}
\pi_M(x) := R_{\Delta_M}\left(x -\nu_{\Delta_M}\right),\quad x\in \mathcal K^{\langle \infty\rangle},
\end{equation}
where $\Delta_M = \mathcal{K}^{\left\langle M \right\rangle} + \nu_{\Delta_M} = \mathcal{K}^{\left\langle M \right\rangle} + \sum_{j=M+1}^{J} L^{j} \nu_{i_j}$ is an $M$-complex containing $x$ and $R_{\Delta_M}\in\mathcal{R}_M$ is the unique rotation determined by \eqref{eq:rot}. More precisely,
\begin{itemize}
\item[(1)] if $x \notin \cV_M^{\langle \infty\rangle}$, then we take $\Delta_M = \mathcal{C}_M(x)$ (i.e.\ $\Delta_M$ is the unique $M$-complex containing $x$),
\item[(2)] if $x \in \cV_M^{\langle \infty\rangle}$, then $\Delta_M$ can be chosen as any of the $M$-complexes meeting at $x$  (thanks to the GLP the image does not depend on a particular choice of an $M$-complex containing $x$).
\end{itemize}
\noindent

This projection restricted to any $M$-complex $\Delta_M$ is a bijection, therefore the inverse of this restriction, $\left(\pi_M\mid_{\Delta_M}\right)^{-1}=:\widetilde{\pi}_{\Delta_M},$  is well defined and given by the formula
\begin{equation*}
\widetilde{\pi}_{\Delta_M}(x) = R_{\Delta_M}^{-1}(x) + \nu_{\Delta_M},\quad x\in \mathcal K^{\langle M\rangle}
\end{equation*}
where $\Delta_M = \mathcal{K}^{\left\langle M \right\rangle} + \nu_{\Delta_M} = \mathcal{K}^{\left\langle M \right\rangle} + \sum_{j=M+1}^{J} L^{j} \nu_{i_j}$.

We can also project onto any other arbitrarily chosen $M$-complex $\Delta_M$. We define \linebreak$\pi_{\Delta_M}:\mathcal{K}^{\left\langle \infty \right\rangle}\to\Delta_M$ by setting
\begin{equation}
\label{eq:uprojection}
\pi_{\Delta_M}(x) = \widetilde{\pi}_{\Delta_M}\left(\pi_{M}(x)\right).
\end{equation}
Clearly, $\pi_{\mathcal{K}^{\left\langle M \right\rangle}} = \pi_{M}$, because $\widetilde{\pi}_{\mathcal{K}^{\left\langle M \right\rangle}} = \Id$.

\subsection{Brownian motion, subordinate Brownian motions and corresponding reflected processes} \label{sec:subordinate}

\subsubsection{Brownian motion on USNF}

Let $Z=(Z_t, \mathbf{P}^{x})_{t \geq 0, \, x \in \mathcal{K}^{\left\langle \infty \right\rangle}}$ be \emph{the Brownian motion} on the USNF $\mathcal{K}^{\left\langle \infty \right\rangle}$ \cite{bib:Lin, bib:Kus2}. Such a process has been constructed by means of Dirichlet forms \cite{bib:Fuk2, bib:Kum}. It is a Feller process with continuous trajectories, whose distributions are invariant under local isometries of $\mathcal{K}^{\left\langle \infty \right\rangle}$. The $L^2$-generator of  its transition
semigroup  is the Laplacian on the fractal,  denoted by $\mathcal L$ (without any subscript). The process $Z$ has transition probability densities $g(t,x,y)$ with respect to the  $d$-dimensional  Hausdorff measure $\mu$ on $\mathcal{K}^{\left\langle \infty \right\rangle}$. More precisely, one has
$$
\mathbf{P}^{x}(Z_t \in A) = \int_A g(t,x,y) \mu({\rm d}y), \quad t > 0, \ \ x \in \mathcal{K}^{\left\langle \infty \right\rangle}, \ \ A \in \cB(\mathcal{K}^{\left\langle \infty \right\rangle}).
$$
Densities $g(t,x,y)$ are jointly continuous on $(0,\infty) \times \mathcal{K}^{\left\langle \infty \right\rangle} \times \mathcal{K}^{\left\langle \infty \right\rangle}$ and satisfy the scaling property
$$
g(t,x,y) = L^{d} g(L^{d_w} t, L x, L y), \quad t>0, \ \ x, y \in \mathcal{K}^{\left\langle \infty \right\rangle}.
$$
Moreover, they enjoy the following \emph{sub-Gaussian estimates}:\ there exist absolute constants \linebreak  $C_1^g,...,C_4^g >0$  such that \cite[Theorems 5.2, 5.5]{bib:Kum}
\begin{multline}
\label{eq:kum}
C_{1}^g t^{-d_s/2} \exp \left(-C_{2}^g \left(\frac{\left|x-y \right|^{d_w}}{t} \right)^{\frac{1}{d_J -1}} \right) \leq g(t,x,y) \\
\leq C_{3}^g t^{-d_s/2} \exp \left(-C_{4}^g \left(\frac{\left|x-y \right|^{d_w}}{t} \right)^{\frac{1}{d_J -1}} \right), \quad t>0, \ \ x,y \in \mathcal{K}^{\left\langle \infty \right\rangle},
\end{multline}
where $d_w$  is the walk dimension of $\mathcal{K}^{\left\langle \infty \right\rangle},$  $d_s=2d/d_w$ is its  spectral dimension, and $d_J > 1$ is the so-called \emph{chemical exponent} of $\mathcal{K}^{\left\langle \infty \right\rangle}$. The regularity properties of the densities $g$ and the bounds \eqref{eq:kum} has been established by T. Kumagai in \cite{bib:Kum} for general nested fractals under some assumption which in our setting is always satisfied (see the comments following \cite[Lemma A.2]{bib:KOP}).

The constant $d_J$ has been introduced in the cited paper as a parameter describing the \emph{shortest path scaling} on a given nested fractal  \cite[Section 3]{bib:Kum}.
Typically $d_w \neq d_J$, but it is known that in the case of Sierpi\'nski gasket one has $d_w = d_J$ .

\subsubsection{Construction of the reflected Brownian motion} \label{sec:constr_rbm}

Suppose now that the unbounded fractal $\mathcal K^{\langle\infty\rangle}$ has the GLP. For an arbitrary $M\in\mathbb Z$ the reflected Brownian motion on $\mathcal{K}^{\left\langle M\right\rangle}$ is defined canonically by
\begin{equation}\label{eq:process-def}
Z_t^M = \pi_M(Z_t),
\end{equation}
where $\pi_M: \mathcal{K}^{\left\langle \infty  \right\rangle} \to \mathcal{K}^{\left\langle M\right\rangle}$ is the projection from \eqref{eq:piem}.

Formally,  this is  the stochastic process $(Z_t^M, \mathbf{P}^{x}_{M})_{t \geq 0, \, x \in \mathcal{K}^{\left\langle M\right\rangle}}$, where the measures $\mathbf{P}^{x}_{M}$, $x \in \mathcal{K}^{\left\langle M\right\rangle}$  (on $C(\mathbb R_+,\mathcal K^{\langle M\rangle})$), are defined as projections of the measures $\mathbf{P}^{x}$, $x \in \mathcal{K}^{\left\langle M\right\rangle}$ (on $C(\mathbb R_+,\mathcal K^{\langle \infty\rangle})$), determining the distribution of the free Brownian motion. The finite dimensional distributions
of $Z^M$ are given by
\begin{align} \label{eq:fdd_reflected}
\mathbf{P}^{x}_{M}(Z^M_{t_1} \in A_1, ..., Z^M_{t_n} \in A_n)= \mathbf{P}^{x}(Z_{t_1} \in \pi_M^{-1}(A_1), ..., Z_{t_n} \in \pi_M^{-1}(A_n)),
\end{align}
for every $n=1,2,...,$ $0 \leq t_1 < t_2 <...< t_n$, $x \in \mathcal{K}^{\left\langle M\right\rangle}$ and $A_1,...,A_n \in \cB(\mathcal{K}^{\left\langle M\right\rangle})$. Note that in fact the projections of the measures $\mathbf{P}^{x}$ (denoted by $\pi_M(\mathbf{P}^{x})$) are well defined for every $x \in \mathcal{K}^{\left\langle \infty \right\rangle}$ and the right hand side of \eqref{eq:fdd_reflected} defines the finite dimensional distributions for such measures in general.

From the definition of the measures $\mathbf{P}^{x}_{M}$, it is clear that  one-dimensional distributions of the process $Z^M$ are absolutely continuous with respect to the Hausdorff measure $\mu$ restricted to the complex $\mathcal{K}^{\left\langle M\right\rangle}$.

The transition probability density function $g_M(t,x,y): (0,\infty) \times \mathcal{K}^{\left\langle \infty \right\rangle} \times \mathcal{K}^{\left\langle M\right\rangle} \to (0,\infty)$ is given by
\begin{equation}
\label{eq:refldens}
g_{M}\left(t,x,y\right)= \left\{ \begin{array}{ll}
\displaystyle\sum_{y'\in \pi_{M}^{-1} (y)}{ g(t,x,y')} & \textrm{if } y \in \mathcal{K}^{\left\langle M\right\rangle} \backslash \cV_{M}^{\left\langle M\right\rangle}, \\
\displaystyle\sum_{y'\in \pi_{M}^{-1} (y)}{g(t,x,y')} \cdot \textrm{rank}(y') & \textrm{if } y \in \cV_{M}^{\left\langle M\right\rangle}, \\
\end{array}\right.
\end{equation}
where $\textrm{rank}(y')$ is the number of $M$-complexes meeting at the point $y'{\in \cV_M^{\langle \infty\rangle}}$.

\subsubsection{Subordinate processes} \label{sec:subord}

Let $S=(S_t)_{t \geq 0}$ be a \emph{subordinator} defined on a probability space $(\widetilde{\Omega}, \mathcal{F}, \mathcal{P})$, i.e.\ a nondecreasing L\'evy process with values in $\R_+$. The laws of $S_t$, given by $\eta_t({\rm d}u):=\mathcal{P}(S_t \in {\rm d}u),$ $t\geq 0,$ form a convolution semigroup of probability measures on $[0,\infty)$ which is uniquely determined by the Laplace transform
\begin{align}\label{eq:laplace}
\int_{[0,\infty)} {\rm e}^{-\lambda u} \eta_t({\rm d}u)  = {\rm e}^{-t \phi(\lambda)}, \quad \lambda>0,
\end{align}
where the Laplace exponent $\phi$ is a \emph{Bernstein function} with $\phi(0+) = 0$. Our standard reference to Bernstein functions and corresponding stochastic processes is the monograph \cite{bib:SSV}.

It is known that every Bernstein function $\phi$ with $\phi(0+) = 0$ admits the representation
\begin{align} \label{eq:def_Phi}
\phi(\lambda) = b \lambda + \int_{(0,\infty)}(1-{\rm e}^{-\lambda u}) \rho({\rm d}u),
\end{align}
where $b \geq 0,$ and $\rho$ is a nonnegative Radon measure on $(0,\infty)$ obeying $\int_{(0,\infty)} (u \wedge 1)\rho({\rm d}u) < \infty$. The number $b$ and the measure $\rho$ are called the drift term and the L\'evy measure of the subordinator $S,$ respectively.

The following regularity conditions on te Bernstein function $\phi$ will be our standing assumptions for the entire paper.

\bigskip

\begin{itemize}	
	\item[\bf (B)] There exist $ \alpha\in (0,d_w]$ and $C_1,C_2,\lambda_0>0$ such that
		\begin{equation}\label{eq:phi_at_zero}
	C_1 \lambda^{  \alpha/d_w}\leq 	\phi(\lambda)\leq C_2\lambda^{ \alpha/d_w},\quad \lambda \in (0,\lambda_0].
	\end{equation}
	Moreover,
	\begin{align}\label{eq:phi_at_infinity}
	\lim_{\lambda \to \infty} \frac{\phi(\lambda)}{\log \lambda} = \infty.
	\end{align}
\end{itemize}

\bigskip

 Our framework in this paper covers a wide class of Bernstein functions and the corresponding subordinators, including $\alpha/d_w$-stable and relativistic $\alpha/d_w$-stable ones; here $\phi(\lambda)=\lambda^{\alpha/d_w}$, $\alpha \in (0,d_w]$, and $\phi(\lambda)=(\lambda+m^{d_w/\vartheta})^{\vartheta/d_w}-m$, $\vartheta \in (0,d_w)$, $m>0$, respectively. In the latter case we have $\phi(\lambda) \approx \lambda$ for $\lambda \to 0^{+}$, and $\phi(\lambda) \approx \lambda^{\vartheta/d_w}$ for $\lambda \to \infty$. Further examples can be found e.g.\ in the monograph \cite{bib:SSV}.

Clearly, under \eqref{eq:phi_at_infinity} we have $\lim_{\lambda \to \infty} \phi(\lambda) = \infty$, and therefore either $b>0$ or $\int_{(0,\infty)} \rho({\rm d}u) = \infty$. We also easily see from \eqref{eq:laplace} that in this case $\eta_t(\left\{0\right\})=0$, for every $t>0$. It was proven in \cite[Lemma 2.1]{bib:KaPP3} that under the complete assumption {\bf (B)}, for every $\gamma, t_0 >0$ there exists a constant $c=c(\gamma,t_0)$ such that
\begin{align}\label{eq:sub_est}
\int_0^{\infty} u^{-\gamma} \eta_t({\rm d}u) \leq c t^{ -\gamma d_w/\alpha}, \quad t \geq t_0.
\end{align}
which in particular gives that for every $t_0>0$,
$$
\sup_{t \geq t_0} \int_0^{\infty} u^{-\gamma} \eta_t({\rm d}u) < \infty.
$$
On the other hand, it follows from \eqref{eq:phi_at_zero} (see \cite[Lemma 2.2]{bib:KaPP2}) that
\begin{equation}
\label{eq:sub2}
\int_1^{\infty} \eta_t (u,\infty) \frac{{\rm d}u}{u} < \infty, \quad t>0.
\end{equation}
One can check that this is equivalent to the condition that
\begin{equation} \label{eq:sub}
\text{for every \ $a>1$ \ and \ $t>0$ \ we have}  \ \ \sum_{M=1}^{\infty} \eta_t\left(a^M, \infty\right) < \infty.
\end{equation}

Given a subordinator $S$ with Laplace exponent $\phi$ satisfying the assumption {\bf (B)}, we define the \emph{subordinate Brownian motion} $X = (X_t)_{t \geq 0}$ and the the \emph{subordinate reflected Brownian motion} $X^M = (X^M_t)_{t \geq 0}$ by
$$
X_t := Z_{S_t}, \quad t \geq 0,
$$
and
$$
X^M_t := Z^M_{S_t}, \quad t \geq 0,
$$
respectively. Formally, these processes are given on respective product probability spaces, but for simplicity their probability measures will be denoted by the same symbols $\mathbf{P}^x$, $x \in \mathcal{K}^{\left\langle \infty \right\rangle}$, and $\mathbf{P}^x_M$, $x \in \mathcal{K}^{\left\langle M\right\rangle}$ as for the diffusions $Z$ and $Z^M$ (here $x$ denotes the starting position of the process). By  general theory of subordination (see e.g.\ \cite[Chapters 5 and 13]{bib:SSV}) the subordinate processes $X$  and $X^M$ are Feller  with c\`adl\`ag paths.

As already mentioned above, under \eqref{eq:phi_at_infinity} we have $\eta_t(\left\{0\right\})=0$, $t>0$, which ensures the existence of transition probability densities of the processes $X$ and $X^M$. They are given by:
$$
p(t,x,y) = \int_0^{\infty} g(u,x,y) \eta_t({\rm d}u), \quad  t>0, \ x,y, \in \mathcal{K}^{\left\langle \infty \right\rangle},
$$
and
$$
p_M(t,x,y) = \int_0^{\infty} g_M(u,x,y) \eta_t({\rm d}u), \quad  t>0, \ x,y, \in \mathcal{K}^{\left\langle M \right\rangle},
$$
respectively. Due to Tonelli's theorem, one has
\begin{equation}
\label{eq:subrefldens}
p_{M}\left(t,x,y\right)= \left\{ \begin{array}{ll}
\displaystyle\sum_{y'\in \pi_{M}^{-1} (y)}{ p(t,x,y')} & \textrm{if } y \in \mathcal{K}^{\left\langle M\right\rangle} \backslash \cV_{M}^{\left\langle M\right\rangle}, \\
\displaystyle\sum_{y'\in \pi_{M}^{-1} (y)}{p(t,x,y')} \cdot \textrm{rank}(y') & \textrm{if } y \in \cV_{M}^{\left\langle M\right\rangle} \\
\end{array}\right.
\end{equation}
(recall that $\textrm{rank}(y')$ denotes the number of $M$-complexes meeting at the point $y'{\in \cV_M^{\langle \infty\rangle}}$).
Both densities $p(t,x,y)$ and $p_M(t,x,y)$ are symmetric in space variables and bounded for every fixed $t>0$. They also inherit the continuity properties from the densities $g(t,x,y)$ and $g_M(t,x,y)$, which is proven in Lemma \ref{lem:regularity_of_p} below. In consequence, both processes $X$ and $X^M$ are strong Feller.

\begin{lemma} \label{lem:regularity_of_p} The following hold.
\begin{itemize}
\item[(a)] For every fixed $t_0 >0$ there exists a constant $c_1=c_1(t_0)$ such that for every $t \geq t_0$ we have
$$
\sup_{(x,y) \in \mathcal{K}^{\left\langle \infty \right\rangle} \times \mathcal{K}^{\left\langle \infty \right\rangle}} p(t,x,y) \leq c_1 t^{-d/\alpha}.
$$
Moreover, the function $(t,x,y) \mapsto p(t,x,y)$ is continuous on $(0,\infty) \times \mathcal{K}^{\left\langle \infty \right\rangle} \times \mathcal{K}^{\left\langle \infty \right\rangle}$.

\item[(b)] For every fixed $t_0 >0$ there is a constant $c_2=c_2(t_0) >0$ such that for every $t \geq t_0$ and $M \in \Z$ we have
$$
\sup_{(x,y) \in \mathcal{K}^{\left\langle M \right\rangle} \times \mathcal{K}^{\left\langle M \right\rangle}} p_M(t,x,y) \leq c_2 \left(t^{ -d/\alpha} \vee L^{ -d M}\right).
$$
Moreover, the function $(t,x,y) \mapsto p_M(t,x,y)$ is continuous on $(0,\infty) \times \mathcal{K}^{\left\langle M \right\rangle} \times \mathcal{K}^{\left\langle M \right\rangle}$.
\end{itemize}
\end{lemma}

\begin{proof}
We first show (a). The boundedness property follows directly from the estimate
$$
\sup_{(x,y) \in \mathcal{K}^{\left\langle \infty \right\rangle} \times \mathcal{K}^{\left\langle \infty \right\rangle}}  g(u,x,y) \leq C_3^g u^{-d_s/2}, \quad u >0,
$$
(constant $C_3^g$ comes from the upper on-diagonal estimate of the density $g(u,x,y)$) and \eqref{eq:sub_est} applied to $\gamma = d_s/2$.

We now prove the continuity. Fix $(t,x,y) \in (0,\infty) \times \mathcal{K}^{\left\langle \infty \right\rangle} \times \mathcal{K}^{\left\langle \infty \right\rangle}$ and $\varepsilon >0$. By \eqref{eq:sub_est} there exists $r>0$ such that
\begin{align} \label{eg:pt_cts_1}
 C_3^g\mathbf{E}_{\mathcal{P}} \left[(S_{t/2})^{-d_s/2} \1_{S_{t/2} \in (0,r]}\right] = C_3^g \int_0^r u^{-d_s/2} \eta_{t/2}(du) < \frac{\varepsilon}{8}.
\end{align}
Let $R> r \vee 1$ be such that
\begin{align} \label{eg:pt_cts_2}
C_3^g (R-1)^{-d_s/2} < \frac{\varepsilon}{8}.
\end{align}
Moreover, by continuity of the function $(u,x,y) \mapsto g(u,x,y)$ on $(0,\infty) \times \mathcal{K}^{\left\langle \infty \right\rangle} \times \mathcal{K}^{\left\langle \infty \right\rangle}$ (which implies the uniform continuity on compact subsets of $(0,\infty) \times \mathcal{K}^{\left\langle \infty \right\rangle} \times \mathcal{K}^{\left\langle \infty \right\rangle}$) there exists $\delta \in (0,1)$ such that for all $s, v \in [r,R]$ with $|s-v|< \delta$ and for all $z,w \in \mathcal{K}^{\left\langle \infty \right\rangle}$ with  $|x-z| < \delta, |y-w| < \delta$ we have
\begin{align} \label{eg:pt_cts_3}
|g(s,x,y) - g(v,z,w)| < \frac{\varepsilon}{4}.
\end{align}
On the other hand, it follows from the stochastic continuity of the subordinator $S$ that there exists $\widetilde \delta \in (0,\delta \wedge t/2)$ such that
$$
\mathcal{P}(|S_t-S_u|\geq \delta) < \left(\frac{\varepsilon}{4\sqrt{4 (C_3^g)^2\mathbf{E}_{\mathcal{P}} \left[(S_{t/2})^{-d_s}\right]}}\right)^2, \quad \text{for every \ $u>0$ \ that satisfies \ $|t-u|<\widetilde \delta$}
$$
(note that the finiteness of the expectation in the denominator follows again from \eqref{eq:sub_est}).

We may now write, for every $u >0$ and $z,w \in \mathcal{K}^{\left\langle \infty \right\rangle}$ such that $|t-u|+|x-z|+|y-w| < \widetilde \delta$,
\begin{align*}
&|p(t,x,y)  - p(u,z,w)| \leq \mathbf{E}_{\mathcal{P}} \left[ |g(S_t,x,y)-g(S_u,z,w)|\right] \\
                      & = \mathbf{E}_{\mathcal{P}} \left[ |\ldots|\1_{\left\{|S_t-S_u| \geq \delta\right\}}\right]+\mathbf{E}_{\mathcal{P}} \left[ |\ldots|\1_{\left\{|S_t-S_u| < \delta\right\}} \right] \\
											& = \mathbf{E}_{\mathcal{P}} \left[ |\ldots|\1_{\left\{|S_t-S_u| \geq \delta\right\}} \right] + \mathbf{E}_{\mathcal{P}} \left[ |\ldots|\1_{\left\{|S_t-S_u| < \delta\right\}}\1_{\left\{S_{t/2}<r\right\}}\right] \\
											& \ \ \ \ + \mathbf{E}_{\mathcal{P}} \left[ |\ldots|\1_{\left\{|S_t-S_u| < \delta\right\}}\1_{\left\{S_{t/2} \geq r\right\}}\1_{\left\{S_{u}, S_t \leq R\right\}}\right]
											+ \mathbf{E}_{\mathcal{P}} \left[ |\ldots|\1_{\left\{|S_t-S_u| < \delta\right\}}\1_{\left\{S_{t/2} \geq r\right\}}\1_{\left\{S_{u} > R \, \vee S_t > R\right\}}\right] \\
											& =: I_1+I_2+I_3+I_4.
\end{align*}
We first estimate $I_1$. By the Cauchy--Schwarz inequality, the elementary estimate $(a-b)^2 \leq a^2+b^2$, $a,b \geq 0$, and the fact that $S_t, S_u \geq S_{t/2}$, $\mathcal{P}$-a.s., we get
\begin{align*}
I_1 & \leq \sqrt{\mathbf{E}_{\mathcal{P}} \left[ |g(S_t,x,y)-g(S_u,z,w)|^2\right]} \sqrt{\mathcal{P}(|S_t-S_u| \geq \delta)}  \\
    & < \sqrt{\mathbf{E}_{\mathcal{P}} \left[ (2 C_3^g (S_{t/2})^{-d_s/2})^2 \right]} \cdot \frac{\varepsilon}{4\sqrt{4(C_3^g)^2 \mathbf{E}_{\mathcal{P}} \left[(S_{t/2})^{-d_s}\right]}} = \frac{\varepsilon}{4}.
\end{align*}
By monotonicity and \eqref{eg:pt_cts_1}, we also have
$$
I_2 \leq 2C_3^g \mathbf{E}_{\mathcal{P}} \left[(S_{t/2})^{-d_s/2} \1_{S_{t/2}<r}\right] < \frac{\varepsilon}{4},
$$
and, by \eqref{eg:pt_cts_3},
$$
I_3 < \frac{\varepsilon}{4}.
$$
Finally, by \eqref{eg:pt_cts_2}, we get
$$
I_4 \leq \mathbf{E}_{\mathcal{P}} \left[ C_3^g\left((S_t)^{-d_s/2}+(S_u)^{-d_s/2}\right) \1_{\left\{|S_t-S_u| < \delta\right\}}\1_{\left\{S_{u} > R \, \vee S_t > R\right\}}\right]
    \leq 2C_3^g (R-1)^{-d_s/2} \leq \frac{\varepsilon}{4},
$$
which completes the proof of (a).

We now show the boundedness in part (b). It follows from \cite[Corollary 3.1]{bib:MO} that there is a constant $c_1>0$ (uniform in $M$) such that
\begin{align*}
\sup_{(x,y) \in \mathcal{K}^{\left\langle M \right\rangle} \times \mathcal{K}^{\left\langle M \right\rangle}}  g_M(u,x,y) & \leq c_1\left( u^{-d_s/2} \vee L^{ -d M}\right) \\
 & = c_1\left( u^{-d_s/2}\1_{\left\{u \in (0,L^{d_wM})\right\}} + L^{ -d M}\1_{\left\{u \geq L^{d_wM}\right\}}\right) , \quad u >0.
\end{align*}
Therefore, by applying \eqref{eq:sub_est} with $\gamma = d_s/2$, we get the claimed bound. The proof of the continuity of $p_M(t,x,y)$ runs in the same way as that in part (a). It is omitted.
\end{proof}

Some of our results in the next chapters will be based on properties of the bridge measures for subordinate processes.
We shall denote by $\mathbf P^{x,y}_{t}$, $t>0$, $x,y \in \cK^{\langle \infty \rangle}$, a version of the conditional law, under $\mathbf P^x$, of $(X_s)_{s \in [0,t]}$ given $X_t=y$. This is a measure on $D([0,t], \cK^{\langle \infty \rangle})$ that satisfies the following disintegration formula:
for any $0 <s<t$ and  $A\in \sigma(X_u: u \leq s),$
\begin{equation}\label{eq:bridge}
p(t,x,y)\mathbf{P}^{x,y}_{t}(A) = \mathbf{E}^x[\mathbf{1}_A p(t-s,X_s,y)].
\end{equation}
For precise definition and more information on the properties of Markovian bridges constructed for general Feller processes we refer to \cite{bib:CU}.
We also consider  bridge measures for the reflected processes, i.e.\ the measures $\mathbf{P}^{x,y}_{M,t}$, $M \in \mathbb{Z}$, $t>0$, $x,y \in \cK^{\langle M \rangle}$ on $D([0,t], \cK^{\langle M \rangle})$ such that for every $0\leq s < t$ and $A \in \sigma\left(X_u^M: u \leq s \right)$ we have
\begin{equation}
\label{eq:Mbridge}
p_M(t,x,y) \mathbf{P}^{x,y}_{M,t}(A)= \mathbf{E}^x \left[\mathbf{1}_{A} p_M(t-s,X^M_s,y) \right].
\end{equation}
The expected values with respect to the measures $\mathbf{P}^{x,y}_{t}$ and $\mathbf{P}^{x,y}_{M,t}$ will be denoted by $\mathbf{E}^{x,y}_{t}$ and $\mathbf{E}^{x,y}_{M,t}$, respectively.

The following identity, connecting the bridge measures for the reflected and free subordinate Brownian motion, will be needed below. Its proof is a consequence of  in \cite[Theorem 4.3]{bib:KOP}; it follows analogously to the proof of \cite[Lemma 2.6]{bib:KaPP2}.

Let us recall that $\Delta_{M,i}$ for $1 \leq i \leq N$ denote the $M$-complexes in $\cK^{\langle M+1 \rangle}$ (see Definition \ref{def:Mcomplex} (10)) and $\pi_{\Delta_{M,i}}: \cK^{\langle \infty \rangle} \to \Delta_{M,i}$ is a projection onto $\Delta_{M,i}$ (see \ref{eq:uprojection}).

\begin{lemma}
\label{lem:bridges} Let $\cK^{\langle \infty \rangle}$ be an USNF with the GLP and let the assumption \textup{\textbf{(B)}} hold.  We have the following.
\begin{itemize}
\item[(a)] For every $t>0$, $x,y \in \cK^{\langle \infty \rangle} \backslash \cV_M^{\langle \infty \rangle}$, $M \in \mathbb{Z}_{+}$ and the set $A \in \cB(D[0,t], \cK^{\langle M \rangle})$ we have
\begin{equation*}
p_M(t,\pi_M(x), \pi_M(y)) \mathbf{P}_{M,t}^{\pi_M(x),\pi_M(y)} (A) = \sum_{y' \in \pi_M^{-1} (\pi_M (y))} p(t,x,y') \mathbf{P}_{t}^{x,y'} \left(\pi_M^{-1}(A)\right).
\end{equation*}
\item[(b)] Consequently, for any $i=1,2,...,N$ and $x \in \cK^{\langle M \rangle} \backslash \cV_M^{\langle \infty \rangle}$,
\begin{equation*}
\sum_{x' \in \pi_M^{-1}(x)} p(t,\pi_{\Delta_{M,i}}(x), x') \mathbf{P}_{t}^{\pi_{\Delta_{M,i}}(x),x'} \left(\pi_M^{-1}(A) \right) = \sum_{x' \in \pi_M^{-1}(x)} p(t,x, x') \mathbf{P}_{t}^{x,x'} \left(\pi_M^{-1}(A) \right).
\end{equation*}
\end{itemize}
\end{lemma}

\subsubsection{Transition semigroups and generators of the processes} \label{sec:refl_ev}
Throughout, we will use the notation
$(T_t)_{t\geq 0},$ $(T_t^M)_{t\geq 0}$ for the $L^2$-transition semigroups of the processes $X, X^M.$ We have
\[T_t f(x)=\int_{\mathcal K^{\langle\infty\rangle}} f(y)p(t,x,y)\,{\rm d}\mu(y),\quad f\in L^2(\mathcal K^{\langle \infty\rangle}, \mu),\ x\in\mathcal K^{\langle\infty\rangle}, \  t >0,\]
and
\[T_t^M f(x)=\int_{ \mathcal K^{\langle M \rangle}} f(y) p_M(t,x,y) \,{\rm d}\mu(y),\quad f\in L^2(\mathcal K^{\langle M\rangle}, \mu),\ x\in\mathcal K^{\langle M\rangle}, \ t>0.\] Generators of these semigroups can be identified as the operators $-\phi(-\mathcal L)$ and $-\phi(-\mathcal L_M)$, where $\mathcal L$ and $\mathcal L_M$  are generators of the diffusion processes $Z$ and $Z^M$, respectively.

By Lemma \ref{lem:regularity_of_p} (b), for every $t > 0$ the kernel $p_M(t,x,y)$ is a bounded function, and since $\mu(\mathcal K^{\langle M\rangle})<\infty$, the operators $T_t^M$, $t>0$, are Hilbert--Schmidt on $L^2(\mathcal K^{\langle M\rangle}, \mu)$. The operators $\phi(-\mathcal L_M)$ have purely discrete spectra:
 their eigenvalues, all of finite multiplicity, satisfy $0=\lambda_1^M < \lambda_2^M \leq \lambda_3^M \leq \ldots \rightarrow \infty.$
Denote the eigenvalues of the operator $-\mathcal L_M$ (which is a special case of $\phi(-\mathcal L_M)$ for $\phi(\lambda)=\lambda$) by
$\mu_1^M, \mu_2^M, \mu_3^M, \ldots$. Since they scale as
$\mu_{k}^{M} = L^{-Md_w}\cdot \mu_k^{1}$, $k = 1,2,\dots$, we have
\begin{equation}\label{eq:2.8}
\lambda_k^{M}=\phi(\mu_{k}^{M})=\phi(L^{-Md_w}\cdot \mu_k^{1}) \qquad k = 1,2,\ldots.
\end{equation}
Moreover, there is a complete system of eigenfunctions $(\psi_k^M)_{k\geq 1}$ (with $\psi_1\equiv L^{-\frac{Md}{2}}$) such that
\begin{equation}\label{eq:2.9}
\phi(-\mathcal L_M)\psi_k^M = \lambda_k^{M}\psi_k^M \qquad {\rm and} \qquad   T_t^M\psi_k^M = {\rm e}^{-t\lambda_k^{M}}\psi_k^M .
\end{equation}

\subsection{Random alloy-type potentials and Feynman-Kac semigroups} \label{sec:Ratp_and_FKS}

Throughout this section we assume that $\cK^{\langle \infty \rangle}$ is an USNF with the GLP. Let us recall that $\cV_0^{\langle M \rangle}$ denotes the set of vertices of all $0$-complexes inside $\cK^{\langle M \rangle}$ and analogously $\cV_0^{\langle \infty \rangle}$ is the set of vertices of all $0$-complexes inside $\cK^{\langle \infty \rangle}$.

We consider the following class of random potentials.

\begin{definition} \textbf{(Fractal alloy-type potential)} \label{def:alloy} Let $W: \cK^{\langle \infty \rangle} \times \cV_0^{\langle \infty \rangle} \to [0,\infty)$ be a  Borel-measurable function, and let $\{\xi_v: v\in \cV_0^{\langle \infty \rangle}\}$ be a family of i.i.d.\  nonnegative and nondegenerate random variables over
a probability space $\left(\Omega, \cM, \mathbb{Q}\right)$. Then the random field
\begin{equation}\label{eq:jeden}
V^{\omega}(x) := \sum_{v\in \cV_0^{\langle \infty \rangle}} \xi_v(\omega) \cdot W(x,v)
\end{equation}
will be called an \emph{alloy-type potential} based on the fractal lattice $\cV_0^{\langle \infty \rangle}$. The two-argument profile function $W$ will be referred to as the \emph{single-site potential}.
\end{definition}
 Later on, we will need to replace this alloy-type potential $V^{\omega}$ with the one based on the same single-site potential and the configuration of the lattice random variables which is periodized with respect to the mapping $\pi_M$. More precisely, we define
\begin{equation} \label{eq:periodized_pot}
V_M^{\omega}(x) = \sum_{v\in \mathcal V^{\langle M\rangle}_{0}} \xi_{v}(\omega) \sum_{v' \in \pi_M^{-1}(v)} W(x,v'), \quad x \in \cK^{\langle \infty\rangle}, \ M \in \mathbb{Z}_{+}.
\end{equation}
This definition relies on the Good Labeling Property of $\cK^{\langle \infty\rangle}$, which permitted to define  continuous projections $\pi_M : \cK^{\langle \infty\rangle} \to \cK^{\langle M\rangle}$. Note that in general it is not true that  the realizations of $V_M^{\omega}$ are periodic with respect to $\pi_M,$ i.e.\ there exists single-site potentials $W$  such that $V^{\omega}_M(\pi_M(x)) \neq V^{\omega}_M(x)$, for some $x \in \cK^{\langle \infty\rangle}$, see Example \ref{ex2}.

Without further assumptions, these potentials may not be regular enough to proceed. In particular, we need that $V^\omega$ and $V^\omega_M$ belong to the local Kato class of both $X$ and $X^M.$  Let us recall the definition.
\begin{definition}\label{def:kato} We say that a measurable function $f:\mathcal K^{\langle \infty \rangle}\to\mathbb R$ belongs to the Kato class of the process $X,$ denoted $\mathbb K^X,$ if
\begin{equation}\label{eq:Kato}
\lim_{t\to 0} \sup_{x\in\mathcal K^{\langle\infty\rangle}} \int_0^t  \left|T_t f(x)\right|  \,{\rm d}s = 0.
\end{equation}
We say that $f$  belongs to the local Kato class of the process $X,$ denoted $\mathbb K^X_{loc},$ if for any  $M$-complex $\Delta$ the function $f\cdot\mathbf 1_{\Delta}$ belongs to $\mathcal K^X.$
\end{definition}
\noindent To get the definition of the Kato class  and the local Kato class of $X^M$ (denotes as $\mathbb K^{X^M}$ and $\mathbb K^{X^M}_{loc}$), obvious changes have to be implemented. It is then direct to observe that $\mathbb K^{X^M}_{loc} = \mathbb K^{X^M}$.

\

To meet these requirements, we introduce the following assumptions:
\begin{itemize}
\item[\bf{(Q1)}] random variables $\xi_{ v},$ $ v\in \cV_0^{\langle\infty\rangle}$,  belong to $L^1(\mathbb Q)$;
\item[\bf {(W1)}] $W(\cdot,v)\in \mathbb K_{loc}^X$  for every $v\in \cV_0^{\langle\infty\rangle},$  and there exists a collection of numbers $\{a_{ v}: v\in \cV_0^{\langle\infty\rangle}\}$ with  $\sum_{v}a_{v}<\infty$ such that for all $x,v\in \mathcal{K}^{\langle \infty \rangle}$ and $M \in\mathbb Z$,
\begin{equation*}
\left(x \in \mathcal{K}^{\langle M\rangle} \wedge v\notin \mathcal{K}^{\langle M+1\rangle} \right) \Rightarrow  W(x,v)\leq a_{v};
\end{equation*}
\item[{ \bf (W2)}]  for sufficiently large $M \in \mathbb{Z}_{+}$,
\begin{equation}
\label{eq:a2cond}
\sum_{v' \in \pi_M^{-1}(\pi_M(v))} W\left(\pi_M(x),v'\right) \leq \sum_{v' \in \pi_M^{-1}(\pi_M(v))} W\left(\pi_{M+1}(x),v'\right), \quad x \in \cK^{\langle \infty \rangle}, \ v\in \cV_0^{\langle \infty \rangle}.
\end{equation}
\end{itemize}

As the profile function $W(x,v)$ measures the strength of influence of the single site located at $v$ onto a particle at the point $x$,
 condition {\bf (W2)} means that for every fixed
$x \in \cK^{\langle \infty \rangle}$ and $v \in \cV_0^{\langle \infty \rangle}$ the influence of the sites from $\pi_M^{-1}(\pi_M(v))$ onto the particle located at $\pi_M(x) \in \cK^{\langle M \rangle}$ is on average  smaller  than the influence onto the position $\pi_{M+1}(x)$.

Before we pass to the Kato-class issues,  observe that
under these assumptions  the random potentials \eqref{eq:jeden} and \eqref{eq:periodized_pot} are
$\mathbb Q-$a.s.\ finite. Indeed, the series $\sum_{v\in \cV_0^{\langle \infty\rangle}} a_v \xi_v$ is convergent  (to a finite limit) $\mathbb Q-$almost surely:\ we have
 \[\mathbb {E^Q}\sum_{v\in\ \cV_0^{\langle\infty\rangle}} a_v\xi_v  = \sum_v a_v\mathbb {E^Q}\xi_v <\infty,\]
 so the series is convergent in $L^1(\mathbb Q),$ thus also $\mathbb Q-$amost surely.
  For $x\in\mathcal K^{\langle \infty\rangle}$ let $\mathbb Z\ni M_0=M_0(x)$ be the smallest number for which $x\in\mathcal K^{\langle M_0\rangle}.$ We then have
\[V^\omega(x) = \sum_{v\in \mathcal K^{\langle M_0+1\rangle}} \xi_v(\omega) W(x,v) + \sum_{ v\notin \mathcal K^{\langle M_0+1\rangle}} \xi_{v}(\omega) W(x,v). \]
The first sum has a finite number of terms,  and the coefficients $W(x,v)$ under the second sum can be estimated by $a_v.$ The convergence follows.

For the periodized potential the argument is similar: for given $x$ and $M_0$ as above, if $M\geq M_0$, then 
\[ V^\omega_M(x) \leq  \sum_{v\in \cV_0^{\langle M\rangle}} \xi_v \left( W(x,v) +\sum_{v'\notin \cV_0^{\langle M_0+1\rangle}} a_{v'}\right)\leq
\sum_{v\in \cV_0^{\langle M\rangle}} \xi_v \left( W(x,v) +\|(a_v)\|_1\right) <\infty,\]
since the sum has a finite number of terms. When $M<M_0,$ then we split the  inner sum into two parts:\  the first one over $v'\in \cV_0^{\langle M_0\rangle},$ the other over $v'\notin \cV_0^{\langle M_0\rangle}.$
Again, the first one has finite number of terms and for the other we can use  {\bf (W1)}.

\begin{proposition}\label{prop:kato}
 Let {\bf (Q1)}, {\bf (W1)}  and {\bf (W2)}  hold.
Then:
\begin{itemize}
\item[(i)] $\mathbb Q-$almost surely, $V^\omega\in \mathbb K_{loc}^X(\mathcal K^{\langle\infty\rangle});$
\item[(ii)] For any $M\in\mathbb Z_+,$ $\mathbb Q-$almost surely, $V^\omega_M\in \mathbb K_{loc}^X(\mathcal K^{\langle \infty\rangle});$
\item[(iii)] For  sufficiently large $M\in\mathbb Z_+,$ $\mathbb Q-$almost surely $V_M^\omega\in \mathbb K^{X^M}(\mathcal K^{\langle M\rangle}).$
    \end{itemize}
\end{proposition}
\begin{proof}

(i) It is enough to show that there exists a measurable set $\Omega_0\subset \Omega$ of full measure such that for any $\omega\in\Omega_0$ and $M\in\mathbb Z_+$   condition \eqref{eq:Kato} holds for $ f=V^\omega\mathbf 1_{\mathcal K^{\langle M\rangle}}.$ We have:
\begin{eqnarray*}
V^\omega(x)\mathbf 1_{\mathcal K^{\langle M\rangle}}(x) &=&
\mathbf 1_{\mathcal K^{\langle M\rangle}}(x)\left(\sum_{v\in \cV_0^{\langle M+1 \rangle}} \xi_v(\omega) W(x,v)
+\sum_{v \in \cV_0^{\langle \infty \rangle} \backslash \cV_0^{\langle M+1 \rangle}} \xi_{v}(\omega) W(x,v)
\right).
\end{eqnarray*}
The first sum consists of finite number of terms. In the other sum, we use the the bound from   {\bf (W1)}:\ $W(x, v)\leq a_{v}.$ Consequently,
\[
V^\omega(x)\mathbf 1_{\mathcal K^{\langle M\rangle}}(x) \leq
\mathbf 1_{\mathcal K^{\langle M\rangle}}(x)\left(\sum_{ v \in
 \cV_0^{\langle M+1\rangle} } \xi_v(\omega) W(x,v)
+\sum_{v\in \cV_0^{\langle\infty\rangle}} \xi_v(\omega)a_{v}\right),
\]
and the series $\sum_{v\in \cV_0^{\langle\infty\rangle}} \xi_v(\omega)a_{v}$ is convergent to a finite limit, $\mathbb Q-$almost surely.
It follows:
\begin{eqnarray*}
\int_0^{t} T_s  (V^\omega\mathbf 1_{\mathcal K^{\langle M\rangle}})(x)\,{\rm d}s
&\leq & \sum_{v\in
 \cV_0^{\langle M+1\rangle} } \xi_v(\omega) \int_0^{t}  T_s (W(\cdot,v)\mathbf 1_{\mathcal K^{\langle M\rangle}}(\cdot))(x)\,{\rm d}s + t \sum_{v\in \cV_0^{\langle \infty\rangle}} \xi_v(\omega)a_{v}\\
 &\leq& \sum_{v\in
 \cV_0^{\langle M+1\rangle} } \xi_v(\omega) \sup_{x\in\mathcal K^{\langle\infty\rangle}} \int_0^{t}  T_s (W(\cdot,v)\mathbf 1_{\mathcal K^{\langle M\rangle}}(\cdot))(x)\,{\rm d}s+ t \sum_{v\in \cV_0^{\langle \infty\rangle}} \xi_v(\omega)a_{v}.
\end{eqnarray*}
As $W(\cdot, v)$ was assumed to be in the local Kato class of $X$ for any $v\in \cV_0^{\langle\infty\rangle}$, the supremum in the first term tends to 0 when $t\to 0.$ As there is a finite number of $v'$s involved, the entire  first term tends to 0. The series in second term is $\mathbb{Q}$-a.s. bounded, hence for $t \to 0$ the whole expression converges to zero.

\smallskip

(ii) The proof is identical with that of (i) - we only need to replace
$\xi_v$ with $\xi_{\pi_M(v)}.$

\smallskip

(iii)
  Let now $M\in\mathbb Z_+$ be large enough so that {\bf (W2)} holds. We have, for $x\in\mathcal K^{\langle M\rangle}$:
\begin{eqnarray*}
V_M^\omega(x) &=& \sum_{v\in \cV_0^{\langle M\rangle}}\xi_{v }(\omega)\sum_{v'\in\pi_M^{-1}(v)}W(x,v')\\
&=& \sum_{v\in \cV_0^{\langle M\rangle}}\xi_{v}(\omega)\left(
\sum_{v' \in\pi_M^{-1}(v)\cap\mathcal K^{\langle M+1\rangle}}W(x,v') +
\sum_{v'\in \pi_M^{-1}( v)\setminus\mathcal K^{\langle M+1\rangle}}W(x,v')
\right).
\end{eqnarray*}

By the definition of the operators $T_t^M$ we have

\begin{eqnarray*}
 T_t^M V_M^{\omega}(x) & = &\int_{\mathcal{K}^{\langle M \rangle}}  p_M (s,x,y)  V_M^{\omega}(y) \mu({\rm d}y)\\
& = &\int_{\mathcal{K}^{\langle M \rangle}} \left( \sum_{y' \in \pi_M^{-1}(y)}  p (s,x,y')  \right) V_M^{\omega}(y) \mu({\rm d}y)\\
&=&\int_{\mathcal{K}^{\langle M \rangle}} \left( \sum_{y' \in \pi_M^{-1}(y) \cap \mathcal{K}^{\langle M+1 \rangle}}  p (s,x,y')  + \sum_{y' \in \pi_M^{-1}(y) \backslash \mathcal{K}^{\langle M+1 \rangle}}  p (s,x,y') \right) V_M^{\omega}(y) \mu({\rm d}y)\\
&=& \int_{\mathcal K^{\langle M+1\rangle}} p(s,x,y) V_M(\pi_M(y))\mu({\rm d} y)  + \int_{\mathcal{K}^{\langle M \rangle}} \sum_{y' \in \pi_M^{-1}(y) \backslash \mathcal{K}^{\langle M+1 \rangle}}  p (s,x,y')  V_M^{\omega}(y) \mu({\rm d}y)
\\
&\leq&  T_s\left(\mathbf{1}_{\mathcal{K}^{\langle M+1\rangle}} V_M \right)(x)  + \int_{\mathcal{K}^{\langle M \rangle}} \sum_{y' \in \pi_M^{-1}(y) \backslash \mathcal{K}^{\langle M+1 \rangle}}  p (s,x,y') V_M^{\omega}(y) \mu({\rm d}y).
\end{eqnarray*}
The inequality between the two last lines comes from {\bf (W2)}.
Using Tonelli's theorem and estimates \cite[Lemma 3.4]{bib:MO} we have
\begin{equation*}
\sum_{y' \in \pi_M^{-1}(y) \backslash \mathcal{K}^{\langle M+1 \rangle}}  p (s,x,y')  = \int_0^\infty \sum_{y' \in \pi_M^{-1}(y) \backslash \mathcal{K}^{\langle M+1 \rangle}}  g (u,x,y')  \eta_s({\rm d}u) \leq c L^{ -d	 M},
\end{equation*}
what gives us
\begin{align*}
&\sup_{x \in \mathcal{K}^{\langle \infty \rangle}} \int_0^t  T_s^M V_M^{\omega}(x)  {\rm d}s\\
 &\leq \sup_{x \in \mathcal{K}^{\langle \infty \rangle}} \int_0^t  T_s\left(\mathbf{1}_{\mathcal{K}^{\langle M+1\rangle}} V_M^{\omega} \right)(x) {\rm d}s
+ \sup_{x \in \mathcal{K}^{\langle \infty \rangle}} \int_0^t \int_{\mathcal{K}^{\langle M \rangle}} \sum_{y' \in \pi_M^{-1}(y) \backslash \mathcal{K}^{\langle M+1 \rangle}}  p (s,x,y') V_M^{\omega}(y) \mu({\rm d}y) {\rm d}s\\
&\leq \sup_{x \in \mathcal{K}^{\langle \infty \rangle}} \int_0^t  T_s\left(\mathbf{1}_{\mathcal{K}^{\langle M+1\rangle}} V_M^{\omega} \right)(x) {\rm d}s
+ ctL^{ -d M} \int_{\mathcal{K}^{\langle M \rangle}} V_M^{\omega}(y) \mu({\rm d}y).
\end{align*}
The uniform convergence of the first term to zero follows from (ii). Moreover, (ii) implies also that  $V_M^{\omega} \in L^1_{loc}(\mathcal K^{\langle \infty\rangle},\mu)$. Hence the integral in the second term is bounded and the whole expression tends to zero as $t \to 0$.

\end{proof}

To get the existence of the Integrated Density of States, we need
yet another technical assumption  on the single-site potential $W$.
We will be assuming that

\medskip

\begin{itemize}
\item[{\bf (W3)}] \hskip 1cm $\displaystyle\sum_{M=1}^{\infty} \sup_{x \in \cK^{\langle \infty \rangle}}  \sum_{v \in \cV_0^{\langle \infty \rangle} \backslash \mathcal{C}_{\lfloor M/4 \rfloor}(x)} W(x,v) < \infty. $
\end{itemize}

\medskip
\noindent

We now give two examples of single-site potentials that satisfy {\bf (W1)-(W3)}.
\begin{example}
\label{ex1}
For $x \in \cK^{\langle \infty \rangle}$, $y \in \cV_0^{\langle \infty \rangle}$ let $f(x,y) = \min \{M \in \mathbb{N}: \exists \Delta_M \ x,y \in \Delta_M \}$.
We set
\begin{equation*}
W(x,y) = \left\{
\begin{array}{ll}
\varphi\left(f(x,y) \right), & x \notin \cV_0^{\langle \infty \rangle}\\
0, & x \in \cV_0^{\langle \infty \rangle}
\end{array}
\right.
\end{equation*}
where $\varphi: \mathbb{N} \to [0,\infty)$ is a nonincreasing function and there exists $c>1$ such that $\varphi(m) \leq N^{-cm}$ for $m \in  \mathbb{N}$.

We first show that { \bf (W2)} holds.
Let $M \geq 0$. If $\pi_M(x) \neq \pi_{M+1}(x)$, then $\pi_{M+1}(x)$ is in different $M$-complex than $\cK^{\langle M \rangle}$ and for each $m \leq M$ and $v \in \cV_{0}^{\langle \infty \rangle}$ there is $\# \{v' \in \pi_M^{-1}(\pi_M(v)) : f(\pi_{M+1}(x),v') = m\} = \# \{v' \in \pi_M^{-1}(\pi_M(v)) : f(\pi_{M}(x),v') = m\}$, because the internal structure of each $M$-complex is the same and $\pi_{M}(\pi_{M+1}(x)) = \pi_{M}(x)$ (see Proposition 3.3 in \cite{bib:KOP}).

Analogously, $\# \{v' \in \pi_M^{-1}(\pi_M(v)) : f(\pi_{M+1}(x),v') = M+1\} = \# \cV_{0}^{\langle M+1 \rangle} - \# \cV_{0}^{\langle M \rangle}  =\# \{v' \in \pi_M^{-1}(\pi_M(v)) : f(\pi_{M}(x),v') = M+1\}$, as in each set there are points $v' \in \cK^{\langle M+1 \rangle}$ which are not included in one $M$-complex (more precisely, not included in $\Delta_{M} (\pi_M(x))$ in the former case and $\Delta_{M} (\pi_{M+1}(x))$ in the latter).

For $m>M+1$ we have $\{v' \in \pi_M^{-1}(\pi_M(v)) : f(\pi_{M+1}(x),v') = m\} =  \{v' \in \pi_M^{-1}(\pi_M(v)) : f(\pi_{M+1}(x),v') = m\}$.

Summarizing, both sums in \eqref{eq:a2cond} have exactly the same number of terms $\varphi(m)$ for each $m$;  in particular, we have equality in \eqref{eq:a2cond} for all $M\geq 0$. Thus {\bf (W2)} holds.

To verify {\bf (W3)}, observe that if $x\in \cK^{\langle \infty \rangle} \backslash \cV_0^{\langle \infty \rangle}$, i.e.\ it is not a vertex of a $0$-complex, then it is an element of exactly one $0$-complex, denoted by $\Delta_0(x)$, which has $k$ vertices. There are $N-1$ $0$-complexes in $\Delta_1(x)$ different than $\Delta_0(x)$. In general, there are $N^{m-1}(N-1)$ $0$-complexes in $\Delta_m(x) \backslash \Delta_{m-1}(x)$, each with $k$ vertices. For $v \in \cV_0^{\langle \infty \rangle} \cap \left(\Delta_m(x) \backslash \Delta_{m-1}(x)\right)$ we have $f(x,v)=m$.

It means that for $m \geq 1$ there are no more than $N^{m-1}(N-1)k$ points $v \in \cV_0^{\langle \infty \rangle}$ such that $f(x,v) = m$, because some vertices may belong to more than one $0$-complex.

We can estimate
\begin{align*}
\sum_{v \in \cV_0^{\langle \infty \rangle} \backslash \mathcal{C}_{\lfloor M/4 \rfloor}(x)} W(x,v) & \leq \sum_{m=\lfloor M/4 \rfloor +1}^{\infty} k(N-1)N^{m-1} \varphi(m) \\
& \leq \sum_{m=\lfloor M/4 \rfloor +1}^{\infty} k(N-1)N^{m-1} N^{-cm} \\
& \leq \frac{k(N-1)}{N} \sum_{m=\lfloor M/4 \rfloor +1}^{\infty} N^{-(c-1)m} \\
& = \frac{k(N-1)}{N^c-N} N^{-(c-1)\lfloor M/4 \rfloor},
\end{align*}
which is a term of convergent series. Hence { \bf (W3)} is satisfied.

\textbf{(W1)} follows similarly. For a vertex $v \in \cV_0^{\langle \infty \rangle}$ we assign a value $a_v$ as follows
\begin{equation*}
a_v = \begin{cases} 1, & \text{ if } v \in \cK^{\langle 0 \rangle}\\
N^{-cm}, & \text{ if } v \in \cK^{\langle m \rangle} \backslash \cK^{\langle m-1 \rangle} \text{ for } m\geq 1\\
\end{cases}
\end{equation*}
As before, for $m\geq 1$ there are no more than $N^{m-1}(N-1) k$ points $v \in \cV_0^{\langle \infty \rangle}$ such that $a_v = N^{-cm}$. We can calculate
\begin{equation*}
\sum_v a_v \leq k\cdot 1+ \sum_{m=1}^{\infty} N^{m-1}(N-1) k \cdot N^{-cm} = k + \frac{k}{N} \sum_{m=1}^{\infty} N^{-m(c-1)}
\end{equation*}
what gives us a convergent geometric series since $c>1$.

Let now $M \in \mathbb{Z}, M \geq 0$ and $x \in \cK^{\langle M \rangle}$, $v \notin \cK^{\langle M+1 \rangle}$.

If $x \in \cV_0^{\langle \infty \rangle}$, then $W(x,v) = 0$. Otherwise, let $m\geq M+2$ be the smallest integer such that $v \in \cK^{\langle m \rangle}$. We have $f(x,v) =m$ and
\begin{equation*}
W(x,v) = \varphi (f(x,v)) = \varphi(m) \leq N^{-cm} = a_v,
\end{equation*}
therefore the condition in \textbf{(W1)} is satisfied.

\end{example}

The next example of a single site potential with non-zero values at vertices leads to the strict inequality in condition {\bf (A2)}, for some points.

\begin{example}
\label{ex2}
Let $\cK^{\langle \infty \rangle}$ be the unbounded Sierpi\'nski gasket.
For $x \in \cK^{\langle \infty \rangle}, v \in \cV_0^{\langle \infty \rangle}$ let $f(x,v) = \min \{M \in \mathbb{N}: \exists \Delta_M \ x,v \in \Delta_M \}$ and
\begin{equation*}
W(x,v) = 4^{-f(x,v)}.
\end{equation*}
Then { \bf (W1)-(W3)} hold. In particular, one has the strict inequality in \eqref{eq:a2cond} for $x \in \cV_{M}^{\langle \infty \rangle}$ such that $\pi_M(x) \neq \pi_{M+1}(x)$.
We leave the verification of { \bf (W1)-(W3)} to the reader - the calculations are elementary, but arduous.
\end{example}

\subsection{Random Schr\"{o}dinger semigroups}\label{sec:RSS}

Throughout  the rest of this paper   we assume that $X = (X_t)_{t \geq 0}$ is a subordinate Brownian motion introduced above, and  that the assumptions \textbf{(B)},  \textbf{(Q1)} and \textbf{(W1)}-\textbf{(W3)} hold. 

Given a nonnegative alloy-type potential $V^{\omega}$ from the respective Kato class, we define the random Feynman--Kac semigroup $\big\{T^{\omega}_t: t \geq 0\big\}$ of the subordinate Brownian motion $X$:
$$
 T^{V^\omega}_t f(x)  =  \mathbf{E}^x \left[{\rm e}^{-\int_0^t V^{\omega}(X_s) {\rm d}s}f(X_t) \right], \quad f \in L^2(\cK^{\langle \infty\rangle}, \mu), \ \ t>0.
$$
This semigroup is strongly continuous on $L^2(\cK^{\langle \infty\rangle}, \mu)$ and it consists of  self-adjoint  operators.
Its $L^2$-generator can be computed to be $-H^{\omega}$, where $H^{\omega} =  \phi(-\mathcal L)  + V^{\omega}$ is the random Schr\"odinger operator based on generator of the subordinate Brownian motion $X$ on $\cK^{\langle \infty\rangle}$. Recall that $\phi$ is the Laplace exponent of the subordinator $S$, introduced in Section \ref{sec:subordinate}, and $\mathcal L$ is the generator of the free Brownian motion on USNF (the fractal Laplacian).  Our standard reference to  Schr\"odinger operators based on the generators of Feller processes and the corresponding Feynman--Kac semigroups is the monograph \cite{bib:DC} by Demuth and van Casteren (see also \cite[Chapters 3.2-3.3]{bib:CZ}).

We  also  define random Feynman--Kac semigroups $\big\{T_t^{D,M, V^\omega}: t\geq 0\big\}$ and $\big\{T_t^{N,M, V^\omega}: t\geq 0\big\}$ for the process $X$ \emph{killed upon exiting the complex} $\cK^{\langle M\rangle}$ and for the \emph{reflected} processes $X^M$ evolving in $\cK^{\langle M\rangle}$, respectively. More precisely, we let
\begin{equation*}
T_t^{D,M, V^\omega} f(x) = \mathbf{E}^x \left[{\rm e}^{-\int_0^t V^{\omega}(X_s) {\rm d}s}f(X_t); t<\tau_{\cK^{\langle M\rangle}} \right], \quad f \in L^2(\cK^{\langle M\rangle}, \mu), \ M \in \mathbb{Z}_+, \ t>0;
\end{equation*}
\begin{equation*}
T_t^{N,M, V^\omega} f(x) = \mathbf{E}_M^x \left[{\rm e}^{-\int_0^t V^{\omega}(X^M_s) {\rm d}s}f(X^M_t) \right], \quad f \in L^2(\cK^{\langle M\rangle}, \mu), \ M \in \mathbb{Z}_+, \ t>0.
\end{equation*}
Here $\tau_{\cK^{\langle M\rangle}} = \inf \{t: X_t \notin \cK^{\langle M\rangle} \}$ denotes the first exit time of the process from the set $\cK^{\langle M\rangle}$.
Also, denote by $A^{D,M, V^\omega}$ and $A^{N,M, V^\omega}$ the $L^2$-generators of these semigroups, respectively. As mentioned above, killing and reflecting the process correspond to imposing the \emph{Dirichlet} and \emph{Neumann conditions} on the generator of this process, respectively. Therefore, the (positive definite) `finite volume' operators
$$H_{M}^{D,  V^\omega}  := - A^{D,M, V^\omega} \qquad \text{and} \qquad H_{M}^{N, V^\omega} := - A^{N,M, V^\omega}$$
can be seen as the \emph{generalized Schr\"odinger operators based on the Dirichlet and Neumann generators} of the initial subordinate Brownian motion $X$.

The operators $T_t^{D,M,V^\omega}$ and $T_t^{N,M, V^\omega}$, $t>0$, are of integral type: for every $t>0$ there exist symmetric and bounded kernels $u_{M,\omega}^D(t,x,y)$ and $u_{M,\omega}^N(t,x,y)$ such that
\begin{equation*}
T_t^{D,M, V^\omega} f(x) = \int_{\cK^{\langle M \rangle}} u_{M,\omega}^D(t,x,y) f(y) \mu({\rm d}y), \quad f \in L^2(\cK^{\langle M\rangle}, \mu)
\end{equation*}
and
\begin{equation*}
T_t^{N,M, V^\omega} f(x) = \int_{\cK^{\langle M \rangle}} u_{M,\omega}^N(t,x,y) f(y) \mu({\rm d}y), \quad f \in L^2(\cK^{\langle M\rangle}, \mu).
\end{equation*}
Due to \eqref{eq:bridge} and \eqref{eq:Mbridge}, these kernels have the following bridge representations
\begin{equation*}
u_{M,\omega}^D(t,x,y) = p(t,x,y) \, \mathbf{E}_{t}^{x,y} \left[{\rm e}^{-\int_0^t V^{\omega}(X_s) {\rm d}s}; t<\tau_{\cK^{\langle M\rangle}} \right]
\end{equation*}
and
\begin{equation*}
u_{M,\omega}^N(t,x,y) = p_M(t,x,y) \,  \mathbf{E}_{M,t}^{x,y} \left[{\rm e}^{-\int_0^t V^{\omega}(X^M_s) {\rm d}s}\right].
\end{equation*}
By the argument already used for the semigroups of free processes,  the operators $T_t^{D,M, V^\omega}$, $T_t^{N,M, V^\omega}$, $t>0$, are Hilbert--Schmidt. Therefore, the spectra of $H_{M}^{D,  V^\omega}$ and $H_{M}^{N,  V^\omega}$ consist of isolated eigenvalues of finite multiplicity such that
$$0 \leq\lambda_1^{D,M, V^\omega} < \lambda_2^{D,M, V^\omega} \leq  \lambda_3^{D,M, V^\omega} \leq \ldots \to \infty$$
and
$$0 \leq \lambda_1^{N,M, V^\omega} < \lambda_2^{N,M, V^\omega} \leq \lambda_2^{N,M, V^\omega} \leq \ldots \to \infty,$$
respectively.

We consider the random empirical measures on $[0,\infty)$ built on spectra of $H_{M}^{D,  V^\omega}$ and $H_{M}^{N,  V^\omega}$, normalized by $\mu(\cK^{\langle M\rangle})$:

\begin{equation}
\label{eq:lmd}
\mathcal N_{M}^{D,  V^\omega} := \frac{1}{\mu\left(\cK^{\langle M\rangle}\right)} \sum_{n=1}^{\infty} \delta_{\lambda_n^{D,M, V^\omega}}
\end{equation}
and
\begin{equation}
\label{eq:lmn}
\mathcal N_{M}^{N,  V^\omega} := \frac{1}{\mu\left(\cK^{\langle M\rangle}\right)} \sum_{n=1}^{\infty} \delta_{\lambda_n^{N,M, V^\omega}} \ .
\end{equation}

Our aim is to prove that the random measures $\mathcal N_{M}^{D,  V^\omega}$ and $\mathcal N_{M}^{N,  V^\omega}$ converge vaguely to a common limit $\mathcal N$, being a nonrandom measure on $[0,\infty)$. This measure will be called the integrated density of states.

\section{Existence of the density of states}\label{sec:existence}

\begin{theorem} \label{thm:IDS}
Let $\cK^{\langle \infty \rangle}$ be an USNF with the GLP and let the assumptions \textup{\textbf{(B)}},  \textup{\textbf{(Q1)}} and \textup{\textbf{(W1)}}-\textup{\textbf{(W3)}} hold. 
Then the random measures $\mathcal N_{M}^{D,  V^\omega}$ and $\mathcal N_{M}^{N,  V^\omega}$ are $\mathbb{Q}$-almost surely
vaguely convergent as $M \to \infty$ to a common nonrandom limit measure $\mathcal N$ on $[0,\infty)$.
\end{theorem}
\noindent
Here we describe the strategy of the proof, which will be given in more detail in the Appendix.

As explained in Introduction, our proof of Theorem \ref{thm:IDS} follows the general approach from \cite{bib:KaPP2}.
We shall deduce the convergence of the measures $\mathcal N_{M}^{D,  V^\omega}$ and $\mathcal N_{M}^{N,  V^\omega}$ from the convergence of their Laplace transforms
$$
\Lambda_{M}^{D,  V^\omega} (t) = \int_0^{\infty} {\rm e}^{-\lambda t} \mathcal N_{M}^{D,  V^\omega} ({\rm d}\lambda) \qquad \text{and} \qquad \Lambda_{M}^{N,  V^\omega} (t) = \int_0^{\infty} {\rm e}^{-\lambda t} \mathcal N_{M}^{N,  V^\omega} ({\rm d}\lambda), \qquad t >0.
$$
First observe that
\begin{align*}
\Lambda_{M}^{D,  V^\omega} (t) & = \frac{1}{\mu\left(\cK^{\langle M\rangle}\right)} \sum_{n=1}^{\infty} {\rm e}^{-\lambda_n^{D,M, V^\omega}t} =  \frac{1}{\mu\left(\cK^{\langle M\rangle}\right)}	 \textrm{Tr}\, T_t^{D,M, V^\omega} \\
& =  \frac{1}{\mu\left(\cK^{\langle M\rangle}\right)} \int_{\cK^{\langle M\rangle}} u_{M,\omega}^D(t,x,x) \, \mu({\rm d}x) \\
& = \frac{1}{\mu\left(\cK^{\langle M\rangle}\right)} \int_{\cK^{\langle M\rangle}} p(t,x,x) \mathbf{E}_{t}^{x,x} \left[{\rm e}^{-\int_0^t V^{\omega}(X_s){\rm d}s}; t<\tau_{\cK^{\langle M\rangle}} \right] \mu({\rm d}x)
\end{align*}
and
\begin{align*}
\Lambda_{M}^{N,  V^\omega} (t) & = \frac{1}{\mu\left(\cK^{\langle M\rangle}\right)} \sum_{n=1}^{\infty} {\rm e}^{-\lambda_n^{N,M, V^\omega}t} =  \frac{1}{\mu\left(\cK^{\langle M\rangle}\right)}	 \textrm{Tr}\, T_t^{N,M, V^\omega} \\
& =  \frac{1}{\mu\left(\cK^{\langle M\rangle}\right)} \int_{\cK^{\langle M\rangle}} u_{M,\omega}^N(t,x,x) \, \mu({\rm d}x) \\
 & =  \frac{1}{\mu\left(\cK^{\langle M\rangle}\right)} \int_{\cK^{\langle M\rangle}} p_M(t,x,x) \mathbf{E}_{M,t}^{x,x} \left[{\rm e}^{-\int_0^t V^{\omega}(X^M_s){\rm d}s}\right] \mu({\rm d}x) .
\end{align*}

Unfortunately, in general expectations of these Laplace transforms  need not be monotone in $M$. 
Therefore the convergence of $\mathcal N_{M}^{D,  V^\omega}$ and $\mathcal N_{M}^{N,  V^\omega}$ will not be analyzed directly -- instead, we replace the initial alloy-type potential $V^{\omega}$ with the periodized one $ V_M^\omega$ (given by \eqref{eq:periodized_pot}), and here we will be able to prove monotonicity, and consequently convergence.
Indeed, once the path of the initial subordinate Brownian motion $X$ on $\cK^{\langle \infty\rangle}$ and $t>0$ are fixed, we show that the expectations $\mathbb{E}_{\mathbb{Q}} {\rm e}^{-\int_0^t V^{\omega}_M(\pi_M(X_s)){\rm d}s}$ are monotone in $M$ (by $\mathbb{E}_{\mathbb{Q}}$ we denote the expected value with respect to the probability measure $\mathbb{Q}$, see Definition \ref{def:alloy}).

This monotonicity argument will be then used to prove that the expectations $ \mathbb{E}_{\mathbb{Q}} \Lambda_{M}^{N,  V_M^\omega}(t)$ converge as $M\to\infty$, for every fixed $t>0$, to a finite limit $\Lambda(t)$ and that the  convergence  is monotone (decreasing). This approach hinges on the following lemma.

\begin{lemma} 
\label{lem:monoto}
 Let $\cK^{\langle \infty \rangle}$ be an USNF with the GLP and let the assumptions \textup{\textbf{(Q1)}} and \textup{\textbf{(W1)}}-\textup{\textbf{(W2)}} hold.    For every $t>0$ and  sufficiently large $M \in \mathbb{Z}_{+}$ we have
\begin{equation}
\label{eq:monoto}
\mathbb{E}_{\mathbb{Q}} {\rm e}^{-\int_0^t V^{\omega}_{M+1}(\pi_{M+1}(X_s)){\rm d}s} \leq \mathbb{E}_{\mathbb{Q}} {\rm e}^{-\int_0^t V^{\omega}_M(\pi_M(X_s)){\rm d}s} \ .
\end{equation}
\end{lemma}

\begin{proof}
Every point from $\cV_0^{\langle M+1\rangle}$ is projected via $\pi_M$ onto a point in $\cV_0^{\langle M\rangle}$, so
\begin{equation*}
\cV_0^{\langle M+1\rangle} = \bigcup_{y \in \cV_0^{\langle M\rangle}} \pi_{M}^{-1}(y) \cap \cV_0^{\langle M+1\rangle}
\end{equation*}
and this sum is disjoint. Therefore, for every fixed $t>0$ and the path of the process $X$, we have
\begin{align*}
\mathbb{E}_{\mathbb{Q}} {\rm e}^{-\int_0^t V^{\omega}_{M+1}(\pi_{M+1}(X_s)){\rm d}s}
&= \mathbb{E}_{\mathbb{Q}} \exp \left( -\int_0^t \sum_{y \in \cV_{0}^{\langle M+1\rangle}} \xi_y(\omega)  \sum_{y' \in \pi_{M+1}^{-1}(y)} W\left(\pi_{M+1}(X_s), y' \right) \ {\rm d}s \right)\\
&= \mathbb{E}_{\mathbb{Q}} \exp \left( -\int_0^t \sum_{y \in \cV_{0}^{\langle M \rangle}} \sum_{\widetilde{y} \in \pi_{M}^{-1}(y) \cap \cV_{0}^{\langle M+1 \rangle}} \xi_{\widetilde{y}}(\omega)   \sum_{y' \in \pi_{M+1}^{-1}(\widetilde{y})} W\left(\pi_{M+1}(X_s), y' \right)  \ {\rm d}s \right)\\
&= \mathbb{E}_{\mathbb{Q}} \prod_{y \in \cV_{0}^{\langle M \rangle}} \exp \left( - \sum_{\widetilde{y} \in \pi_{M}^{-1}(y) \cap \cV_{0}^{\langle M+1 \rangle}}   \xi_{\widetilde{y}}(\omega)  \int_0^t \sum_{y' \in \pi_{M+1}^{-1}(\widetilde{y})} W\left(\pi_{M+1}(X_s), y' \right) {\rm d}s \right)\\
& = \prod_{y \in \cV_{0}^{\langle M \rangle}} \mathbb{E}_{\mathbb{Q}} \exp \left( - \sum_{\widetilde{y} \in \pi_{M}^{-1}(y) \cap
\cV_{0}^{\langle M+1 \rangle}}   \xi_{\widetilde{y}}(\omega)  a_{\widetilde{y}} \right) = (\ast)
\end{align*}
with
\begin{equation*}
a_{\widetilde{y}} = \int_0^t \sum_{y' \in \pi_{M+1}^{-1}(\widetilde{y})} W\left(\pi_{M+1}(X_s), y' \right) {\rm d}s,
\end{equation*}
because all $\xi_{\widetilde{y}}$'s are independent random variables. Note that the numbers $a_{\widetilde{y}}$ do not depend on $\omega$.

Fix now $y \in \cV_{0}^{\langle M \rangle}$ and denote $C_y=\left\{\widetilde{y} \in \pi_{M}^{-1}(y) \cap \cV_{0}^{\langle M+1 \rangle}: a_{\widetilde{y}} >0\right\}$.
Suppose first that $C_y \neq \emptyset$. Let us enumerate elements of the set $C_y$ as $\widetilde{y}_i$  for $1\leq i \leq n = \# C_y$.
Then, by using the generalized H\"older inequality
\begin{equation*}
\label{eq:hoelder}
\mathbb{E}_{\mathbb{Q}}\left(|X_1...X_n|\right) \leq \left(\mathbb{E}_{\mathbb{Q}}|X_1|^{p_1}\right)^{\frac{1}{p_1}} ... \left(\mathbb{E}_{\mathbb{Q}}|X_n|^{p_n}\right)^{\frac{1}{p_n}}
\end{equation*}
with $X_i = {\rm e}^{-\xi_{\widetilde{y}_i} a_{\widetilde{y}_i}}$ and $p_i = \frac{1}{a_{\widetilde{y}_i}} \sum_{j=1}^{n} a_{\widetilde{y}_j}$ (clearly, $\frac{1}{p_1} + ...+ \frac{1}{p_n} = 1$), we obtain
\begin{align*}
\mathbb{E}_{\mathbb{Q}} \exp \left( - \sum_{i=1}^{n} \xi_{\widetilde{y}_i}(\omega) a_{\widetilde{y}_i} \right) & = \mathbb{E}_{\mathbb{Q}} \prod_{i=1}^{n} {\rm e}^{- \xi_{\widetilde{y}_i}(\omega) a_{\widetilde{y}_i}} \\ & \leq \prod_{i=1}^{n} \left( \mathbb{E}_{\mathbb{Q}} {\rm e}^{-\xi_{\widetilde{y}_i}(\omega) \sum_{j=1}^{n} a_{\widetilde{y}_j}} \right)^{\frac{a_{\widetilde{y}_i}}{\sum_{j=1}^{n} a_{\widetilde{y}_j}}} = \mathbb{E}_{\mathbb{Q}} {\rm e}^{-\xi_{y}(\omega) \sum_{j=1}^{n} a_{\widetilde{y}_j}}.
\end{align*}
Again, the last equality is a consequence of the fact that all $\xi_{\widetilde{y}_i}$'s are i.i.d.\ random variables. This gives
$$
\mathbb{E}_{\mathbb{Q}} \exp \left( - \sum_{\widetilde{y} \in \pi_{M}^{-1}(y) \cap \cV_{0}^{\langle M+1 \rangle}}   \xi_{\widetilde{y}}(\omega)  a_{\widetilde{y}} \right)
\leq \mathbb{E}_{\mathbb{Q}} \exp \left( - \xi_{y}(\omega) \sum_{\widetilde y \in \pi_{M}^{-1}(y) \cap \cV_{0}^{\langle M+1 \rangle}} a_{\widetilde{y}} \right).
$$
If $C_y = \emptyset$, then the same bound holds trivially. In consequence,
\begin{equation*}
(\ast) \leq \prod_{y \in \cV_{0}^{\langle M \rangle}} \mathbb{E}_{\mathbb{Q}} \exp \left( - \xi_{y}(\omega) \sum_{\widetilde y \in \pi_{M}^{-1}(y) \cap \cV_{0}^{\langle M+1 \rangle}} \int_0^t \sum_{y' \in \pi_{M+1}^{-1}(\widetilde{y})} W\left(\pi_{M+1}(X_s), y' \right) {\rm d}s  \right).
\end{equation*}
Finally, since $\pi_{M} \circ \pi_{M+1} = \pi_M$ (\cite[Proposition 3.3]{bib:KOP}), we have that for every $y \in \cV_0^{\langle M\rangle}$
\begin{equation*}
\pi^{-1}_{M}(y) = \bigcup_{\widetilde{y} \in \pi^{-1}_{M}(y) \cap \cV_0^{\langle M+1 \rangle}} \pi^{-1}_{M+1}(\widetilde{y})
\end{equation*}
and this sum is disjoint. By rearranging the summation in the expression above, we can then get that it is equal to
\begin{align*}
\prod_{y \in \cV_{0}^{\langle M \rangle}}\mathbb{E}_{\mathbb{Q}}  \exp & \left( - \xi_{y}(\omega) \int_{0}^{t} \sum_{y' \in \pi_{M}^{-1}(y)} W\left(\pi_{M+1}(X_s), y' \right) {\rm d}s  \right) \\
&= \mathbb{E}_{\mathbb{Q}} \exp \left( - \sum_{y \in \cV_{0}^{\langle M \rangle}}  \xi_{y}(\omega) \int_{0}^{t} \sum_{y' \in \pi_{M}^{-1}(y)} W\left(\pi_{M+1}(X_s), y' \right) {\rm d}s  \right).
\end{align*}
Applying  \textbf{(W2)},  we conclude that for sufficiently large $M$ the last expectation is less than or equal to
$$
\mathbb{E}_{\mathbb{Q}} \exp \left( - \sum_{y \in \cV_{0}^{\langle M \rangle}}  \xi_{y}(\omega) \int_{0}^{t} \sum_{y' \in \pi_{M}^{-1}(y)} W\left(\pi_{M}(X_s), y' \right) {\rm d}s  \right)
 = \mathbb{E}_{\mathbb{Q}} {\rm e}^{-\int_0^t V^{\omega}_M(\pi_M(X_s)){\rm d}s},
$$
which completes the proof.
\end{proof}

Below we formulate the proposition concerning the monotone convergence of the expected values of $ \mathbb{E}_{\mathbb{Q}} \Lambda_{M}^{N,  V_M^\omega}(t)$. Its proof, together with some accompanying lemmas, can be found in the Appendix.
It is similar to the proof for Poissonian potentials on the Sierpi\'{n}ski gasket from \cite{bib:KaPP2}.

\begin{proposition} \label{thm:convergence_with_star}
Let $\cK^{\langle \infty \rangle}$ be an USNF with the GLP and let the assumptions \textup{\textbf{(B)}},  \textup{\textbf{(Q1)}} and \textup{\textbf{(W1)}}-\textup{\textbf{(W2)}} hold.  Then, for every $t>0$, $ \mathbb{E}_{\mathbb{Q}} \Lambda_{M}^{N,  V_M^\omega}(t)$ decreases to a finite limit $\Lambda(t)$ as $M \to \infty$.
\end{proposition}

The next Lemma asserts that  the expectations of $ \Lambda_{M}^{D,  V_M^\omega}(t)$, $ \Lambda_{M}^{D, V^\omega}(t)$ and $ \Lambda_{M}^{N,  V^\omega}(t)$ share the limit with the expectations of $ \Lambda_{M}^{N,  V_M^\omega}(t)$, as $M \to \infty$.

\begin{lemma}
\label{lem:commonlimit}
Let $\cK^{\langle \infty \rangle}$ be an USNF with the GLP and let the assumptions \textup{\textbf{(B)}},  \textup{\textbf{(Q1)}} and \textup{\textbf{(W1)}}-\textup{\textbf{(W3)}} hold.  Fix $t>0$. \begin{itemize}
\item[(a)] We have
\begin{equation*}
\sum_{M=1}^{\infty} \mathbb{E}_{\mathbb{Q}} \left(\Lambda_{M}^{D,  V^\omega}(t) - \Lambda_{M}^{N,  V^\omega} (t) \right)^2 < \infty;
\end{equation*}
in particular,
\begin{equation*}
\lim_{M \to \infty} \mathbb{E}_{\mathbb{Q}} \left(\Lambda_{M}^{D,  V^\omega} (t) - \Lambda_{M}^{N,  V^\omega} (t) \right)^2 = 0 .
\end{equation*}
\item[(b)] We have
\begin{equation*}
\lim_{M \to \infty} \mathbb{E}_{\mathbb{Q}} \left(\Lambda_{M}^{D,  V_M^\omega}(t) - \Lambda_{M}^{N,  V_M^\omega} (t) \right) = 0 \quad
\text{and} \quad
\lim_{M \to \infty} \mathbb{E}_{\mathbb{Q}} \left( \Lambda_{M}^{D,  V^\omega} (t) - \Lambda_{M}^{D,  V_M^\omega} (t)\right) = 0 .
\end{equation*}
\end{itemize}
\end{lemma}

 From Proposition \ref{thm:convergence_with_star} and Lemma \ref{lem:commonlimit}, we get the following Corollary.

\begin{corollary}\label{coro:conv}
Let $\cK^{\langle \infty \rangle}$ be an USNF with the GLP and let the assumptions \textup{\textbf{(B)}},  \textup{\textbf{(Q1)}} and \textup{\textbf{(W1)}}-\textup{\textbf{(W3)}} hold.  For every $t>0$, $ \mathbb{E}_{\mathbb{Q}} \Lambda_{M}^{D,  V^\omega}(t)$, $ \mathbb{E}_{\mathbb{Q}} \Lambda_{M}^{D,  V_M^\omega}(t)$ and $ \mathbb{E}_{\mathbb{Q}} \Lambda_{M}^{N,  V^\omega}(t)$ are also convergent as $M \to \infty$ to $\Lambda(t)$ identified in  Proposition \ref{thm:convergence_with_star}.
\end{corollary}

Now as we know that the expectations are convergent,
  our last technical step is  to show that the series of variances of the random variables $\Lambda_{M}^{D,  V^\omega}(t)$ and $\Lambda_{M}^{N,  V^\omega}(t)$ are convergent.

\begin{lemma} \label{lem:variances}
Let $\cK^{\langle \infty \rangle}$ be an USNF with the GLP and let the assumptions \textup{\textbf{(B)}},  \textup{\textbf{(Q1)}} and \textup{\textbf{(W1)}}-\textup{\textbf{(W3)}} hold.  For every $t>0$ we have
\begin{equation}
\label{eq:dvarlemma}
\sum_{M=1}^{\infty} \mathbb{E}_{\mathbb{Q}} \left[ \Lambda_{M}^{D,  V^\omega}(t) - \mathbb{E}_{\mathbb{Q}}\Lambda_{M}^{D,  V^\omega}(t)\right]^2 < \infty
\end{equation}
and
\begin{equation}
\label{eq:nvarlemma}
\sum_{M=1}^{\infty} \mathbb{E}_{\mathbb{Q}} \left[ \Lambda_{M}^{N,  V^\omega}(t) - \mathbb{E}_{\mathbb{Q}}\Lambda_{M}^{N,  V^\omega}(t)\right]^2 < \infty
\end{equation}
\end{lemma}

Once
 Proposition \ref{thm:convergence_with_star} and Lemma \ref{lem:variances} are established, to get the statement of Theorem \ref{thm:IDS}  we can follow verbatim the argument in the proof of \cite[Theorem 3.2]{bib:KaPP2}.

\section{The Lifschitz tail} \label{sec:LT}
The existence of the Lifchitz tail in the Anderson model on fractals
can be proven under additional assumptions on the  single-site potential.
  These assumptions are  as follows:\
\begin{itemize} 
\item[\bf (W4)](\emph{finite range  and uniform integrability}) there exists $M_0>0$ such that $W(x,v)=0$ when $d_{M_0}(x,v)>1$ (i.e. when $x,v$ do not belong to the same $M_0$-complex), and
\[
\sup_{x \in \cK^{\langle \infty \rangle}} \int_{\mathcal{C}_{M_0}(v)}W(x,v)\mu({\rm d}x) < \infty;
\]\
\item[ \bf (W5)] (\emph{separation from zero}) there exist $ m_1\in\mathbb Z_{-}:=\left\{-1,-2,-3,\ldots\right\}$ and $A_0>0$ such that $W(x,v)>A_0$ when $d_{m_1}(x,v)\leq 1$.
\end{itemize}
The behaviour of the IDS at 0 reflects the behaviour of the common cumulative  distribution function of the random variables at the vertices, i.e.\ the function
\begin{equation}\label{eq:f-q}
F_\xi(t) :=\mathbb Q(\xi_v\leq t).
\end{equation}
Recall that  we always assume that  the random variables $\xi_v$ are  nonnegative,  nondegenerate and integrable, see Definition \ref{def:alloy} and Assumption {\bf (Q1)}.
Now we require somewhat more regularity. What we need is that the function $F_\xi$ fulfills  the following assumption:\
\begin{itemize}
\item[\bf (Q2)] $F_\xi (\lambda)>0$ for $\lambda >0,$ and there exists $\lambda_0>0$ such that $F_\xi$ is continuous on $[0,\lambda_0]$ (with possible left discontinuity at 0, corresponding to an atom at 0).
\end{itemize}

\smallskip

We define
\begin{equation}\label{eq:d_0_alloytype}
  D_0 := \frac{\widetilde C_1}{4 A_0 C_0 r_0},
 \end{equation}
where
\begin{align} \label{eq:c-tilde}
\widetilde{C}_1 = C_1(\mu_2^1)^{\alpha/d_w}
\end{align}
($C_1$ is the constant from \eqref{eq:phi_at_zero} and $\mu_2^1$ was defined in Section \ref{sec:refl_ev}), $C_0$ comes from \eqref{eq:number-points}, $A_0$  from {\bf (W5)}, and $r_0$ is the maximal rank in \eqref{eq:def-r0}.

Following \cite[Section 3]{bib:KaPP3} we introduce the functions:
\begin{equation}\label{eq:g_i_j}
g(x): = \log \frac{1}{F_\xi(\frac{D_0}{x})}, \qquad j(x): = x^{d+\alpha}g(x^{\alpha}), \quad x>0,
\end{equation}
with the constant $D_0>0$  determined by \eqref{eq:d_0_alloytype}.
The functions:
\begin{equation}\label{eq:x_i_h}
x_t := j^{-1}(t), \qquad h(t):=g(x_{t}^{\alpha}), \quad \mbox{ for }   t \geq t_0:=j\left(\left(\frac{D_0}{\lambda_0}\right)^{1/\alpha}\right),
\end{equation}
 are well-defined and  they satisfy the relations
\begin{equation}\label{eq:t_cond}
t = x_{t}^{ d+\alpha}h(t) = x_t^{ d+\alpha} \log \frac{1}{F_\xi\left(\frac{D_0}{ x_{t}^{\alpha}}\right)}.
\end{equation}
Moreover, we have $x_t\to \infty$ as $t\to \infty.$
Our goal is the following theorem, very similar in spirit to its $\mathbb R^d-$counterpart \cite[Theorem 1.1]{bib:KaPP3}.

\begin{theorem}\label{th:Lifshitz-IDS}
Let $\cK^{\langle \infty \rangle}$ be an USNF with the GLP and let the assumptions \textup{\textbf{(B)}}, \textup{\textbf{(Q1)}}-\textup{\textbf{(Q2)}} and \textup{\textbf{(W1)}}-\textup{\textbf{(W5)}} hold.
Then there exist positive constants $K,\widetilde K, R, \widetilde R$
such that
\begin{align}\label{eq:main-statement}
-K\leq \liminf_{\lambda\searrow 0}\frac{\lambda^{ \frac{d}{\alpha}}\log \mathcal N(\lambda)}{g(R/\lambda)} \quad \text{and} \quad
\limsup_{\lambda\searrow 0}\frac{\lambda^{ \frac{d}{\alpha}}\log \mathcal N(\lambda)}{g(\widetilde R/\lambda)}\;\leq \; -\widetilde K,
\end{align}
where $g(\lambda)$ is defined by (\ref{eq:g_i_j}) above.
\end{theorem}
Observe that when $\lim_{\kappa\to 0}F_\xi(\kappa) >0$ (an atom at zero is present), then the statement (\ref{eq:main-statement}) becomes
\[
-K'\leq  \liminf_{\lambda\searrow 0}\lambda^{ \frac{d}{\alpha}}\log \mathcal N(\lambda) \leq \limsup_{\lambda\searrow 0}\lambda^{ \frac{d}{\alpha}}\log
\mathcal N(\lambda)\leq -\widetilde K'
\]
with certain positive $K',$ $\widetilde K'$, which agrees with the statement we have obtained for the Poisson potential, cf.\ \cite{bib:KaPP2,bib:HB-KK-MO-KPP1}.
We refer to \cite[Section 6]{bib:KaPP3} for further discussion of the dependence of the correction rate $g$ on the properties of $F_\xi$ close to zero.

As usual in problems of this type, we will not approach this theorem directly. Instead, we will investigate the behaviour of the Laplace transform of the IDS, $\Lambda,$ at infinity, and
then  use a Tauberian-type argument to derive bounds on the IDS.

\subsection{The upper bound for the Laplace transform of the IDS}
We start with the upper bound, which permits to correctly identify  rate function.

\begin{theorem}\label{th:main-alloy-type-fractals}
 Let $\cK^{\langle \infty \rangle}$ be an USNF with the GLP and let the assumptions \textup{\textbf{(B)}}, \textup{\textbf{(Q1)}}-\textup{\textbf{(Q2)}}, \textup{\textbf{(W1)}}-\textup{\textbf{(W3)}}, and  \textup{\textbf{(W5)}} hold.  Let $h$ be given by (\ref{eq:x_i_h}).
 Then there exists $K_1>0$ such that
 \begin{equation}\label{eq:upper-poiss-1}
\limsup_{t\to\infty} \frac{ \log  \Lambda(t)}{t^{ \frac{d}{d+\alpha}}\cdot h(t)^{ \frac{\alpha}{d+\alpha}}}\leq -K_1.
\end{equation}
In particular, when $F_\xi(0) > 0$, then
\begin{equation}\label{eq:upper-poiss-2}
\limsup_{t\to\infty} \frac{\log \Lambda(t)}{t^{ \frac{d}{d+\alpha}}}\leq -K_1 \cdot \left(\log\frac{1}{F_\xi(0)}\right)^{ \frac{\alpha}{d+\alpha}}.
\end{equation}
\end{theorem}

As the starting point for the proof of this theorem we collect some auxiliary estimates and lemmas.
From Proposition \ref{thm:convergence_with_star} we have that for sufficiently large  $M\in\mathbb Z_+$ and all $t>0$ it holds that
\[\Lambda(t)\leq \mathbb{E}_{\mathbb{Q}} \Lambda_{M}^{N,  V_M^\omega}(t)  =   \mathbb{E}_{\mathbb{Q}}  \left[\frac{1}{\mu(\mathcal K^{\langle M\rangle})}\int_{\mathcal K^{\langle M\rangle}} p_{M}(t,x,x)\mathbf E_{M,t}^{x,x}\left[{\rm e}^{-\int_0^t V_M^\omega(X_s^M){\rm d}s}\right]\mu({\rm d}x)\right],\]
and the right-hand side integral is the trace of the operator of the reflected subordinate process on $\mathcal K^{\langle M\rangle}$ at  time $t,$ evolving in presence of the potential $V_M^\omega(x),$
i.e. for every $t>0$ and $M\in\mathbb Z_+$  large enough
\[\Lambda(t)\leq \frac{1}{\mu(\mathcal K^{\langle M\rangle})}   \mathbb{E}_{\mathbb{Q}}  \mbox{Tr}\,T_t^{N,M, V_M^\omega}.\]
It can be further estimated as (cf.\ \cite[p.\ 3921]{bib:KaPP})
\begin{equation}\label{eq:to}
\frac{1}{\mu(\mathcal K^{\langle M\rangle})}   \mathbb{E}_{\mathbb{Q}}
\left[{\rm e}^{-(t-1)\lambda_1^{N,M, V_M^\omega}}\right]\sup_{x\in \mathcal K^{\langle M\rangle}}p_M(1,x,x),  \quad t >1,
\end{equation}
where
$\lambda_1^{N,M, V_M^{\omega}}$ is the ground state eigenvalue of the Schr\"odinger operator $H_{M}^{N, V_M^\omega} =\phi(-\mathcal L_M)+V_M^{\omega}$, see  Section \ref{sec:RSS}. As we have
\begin{align}\label{eq:diag_est}
A^*:=\sup_{M\geq 0} \frac{1}{\mu(\mathcal K^{\langle M\rangle})}{\sup_{x\in\mathcal K^{\langle M\rangle}}p_M(1,x,x)}<\infty,
\end{align}
 see Lemma \ref{lem:regularity_of_p} (b),  we are left with estimating from below the eigenvalue in the exponent in (\ref{eq:to}).

Fix $M\in \mathbb Z_+$  (note that automatically $M > m_1$, where $m_1$ is the negative integer appearing in \textbf{(W5)}).
We replace the random variables $\xi_v$ with truncated random variables $\widetilde\xi_v,$ where
\begin{equation}\label{eq:q-diminished-alloy-type}
\widetilde{\xi}_v:= \xi_v\wedge\frac{D_0}{L^{ M\alpha}}, \; v\in \cV_0^{\langle M\rangle}.
\end{equation}
 Observe that by \textbf{(W5)} we have
\begin{equation}\label{eq:perodized-pot-damped}
V_M^\omega(x) \geq \sum_{v \in \cV_0^{\langle M\rangle}}  \widetilde{\xi}_v(\omega) W(x,v) \geq A_0 \sum_{v \in \cV_0^{\langle M\rangle}}  \widetilde{\xi}_v(\omega) \1_{\cC_{m_1}(v)}(x) =: \widetilde V_M^\omega(x),\quad x \in \mathcal{K}^{\langle \infty \rangle}.
\end{equation}
As $ \lambda_1^{N,M, \widetilde V_M^\omega}\leq \lambda_1^{N,M, V_M^\omega},$  it is enough to find a lower bound for $ \lambda_1^{N,M, \widetilde V_M^\omega}.$ To this end we will use Temple's inequality.

 \begin{proposition}[{Temple's inequality, \cite[Theorem XIII.5]{bib:RS}}] \label{prop:temple} Suppose $(H,\mathcal D(H))$ is a self-adjoint operator  on a Hilbert space with inner product $\langle \cdot,\cdot \rangle$ such that $\lambda_1:=\inf\sigma(H)$ is an isolated eigenvalue and let $\mu\leq \inf(\sigma(H)\setminus \{\lambda_1\}).$  Then for any $\psi\in\mathcal D(H)$ which satisfies
	\begin{equation}\label{eq:temple-condition}
	\langle \psi, H\psi\rangle <\mu\quad\mbox{ and } \quad  \|\psi\|=1
	\end{equation}
	the following estimate
	\begin{equation}\label{eq:temple-ineq}
	\lambda_1\geq \langle \psi, H\psi\rangle - \frac{\langle H\psi, H\psi\rangle-\langle \psi,H\psi\rangle^2}{\mu-\langle \psi, H\psi\rangle}.
	\end{equation}
	holds.
\end{proposition}
We  derive the following lemma.

\begin{lemma}\label{lem:l-33}  Let $\cK^{\langle \infty \rangle}$ be an USNF with the GLP and let the assumptions \textup{\textbf{(B)}}, \textup{\textbf{(Q1)}}-\textup{\textbf{(Q2)}}, \textup{\textbf{(W1)}}, and \textup{\textbf{(W5)}} hold.
 Then there exists $ M_2 \in \mathbb Z_+$ such that for any $M \geq M_2$ we have
\begin{equation}\label{lem:l-33-teza-alloy-type}
 \lambda_1^{N,M,V_M^\omega} \geq \lambda_1^{N,M,\widetilde V_M^\omega} \geq \frac{1}{L^{Md}}\left[ \int_{\mathcal{K}^{\left\langle M\right\rangle}} \widetilde{V}^\omega_M(x)\mu({ \rm d}x) - \frac{2\int_{\mathcal{K}^{\left\langle M\right\rangle}} (\widetilde{V}^{\omega}_M(x))^2\mu({\rm d}x)}{\widetilde{C}_{1}L^{ -M\alpha}} \right],
\end{equation}
with the constant $\widetilde{C}_1$ introduced in (\ref{eq:c-tilde}).
\end{lemma}

\begin{proof}  Let $ M \in \mathbb Z_+.$ Since ${V}^{\omega}_M \geq \widetilde{V}^{\omega}_M,$ then we have $ \lambda_1^{N,M,V_M^\omega} \geq \lambda_1^{N,M,\widetilde V_M^\omega}.$ Temple's inequality will be applied to the operator $ H:= H_{M}^{N, \widetilde V_M^\omega} =\phi(-\mathcal L_M)+\widetilde V_M^{\omega}$ acting on $L^2(\mathcal{K}^{\langle M\rangle}),$ and $\mu = \lambda_2^{M}$
(nonrandom).
The spectrum of $ H:= H_{M}^{N, \widetilde V_M^\omega}$  is discrete and we have that
$$
\mu=\lambda_2^M \leq  \lambda_2^{N,M,\widetilde V_M^\omega}  =\inf \left(  \sigma\big(H_{M}^{N, \widetilde V_M^\omega}\big) \backslash \left\{\lambda_1^{N,M,\widetilde V_M^\omega}\right\}\right).
$$
To use the Temple inequality, we choose
 $\psi = \psi_1^M \equiv \frac{1}{L^{Md/2}}$ to be the normalized principal eigenfunction of the operator $ -\mathcal L_M$ corresponding to the eigenvalue $ \mu_1^M = 0.$  Consequently, $\phi(-\mathcal L_M)\psi = \lambda_1^M \psi = 0$, which implies
$$
\left\langle\psi, H_{M}^{N, \widetilde V_M^\omega} \psi\right\rangle = \left\langle\psi,\phi( -\mathcal L_M)\psi\right\rangle+ \left\langle\psi,\widetilde{V}_{M}^{\omega}\psi\right\rangle=
\left\langle\psi,\widetilde{V}_{M}^{\omega}\psi\right\rangle=
\frac{1}{L^{Md}}\int_{\mathcal{K}^{\left\langle M\right\rangle}} \widetilde{V}^\omega_M(x)\mu({\rm d}x).
$$

From the definition of $\widetilde{V}_M^\omega$  in \eqref{eq:perodized-pot-damped}, \eqref{eq:def-r0}, \eqref{eq:number-points} and \eqref{eq:q-diminished-alloy-type}  we get
\begin{eqnarray*}
\int_{\mathcal{K}^{\left\langle M\right\rangle}} \widetilde{V}_M^\omega(x)\mu({\rm d}x) & = & A_0 \sum_{v \in \cV_0^{\langle M\rangle}}  \widetilde{\xi}_v(\omega) \mu(\mathcal{K}^{\left\langle M\right\rangle} \cap \cC_{m_1}(v)) \\
&\leq&
   A_0 r_0 L^{m_1d}  \sum_{v \in \cV_0^{\langle M\rangle}}  \widetilde{\xi}_v(\omega)
\leq  A_0 r_0 C_0  L^{Md}\frac{D_0}{L^{M\alpha}}.
\end{eqnarray*}
Consequently, using the definition of $D_0,$ \eqref{eq:d_0_alloytype},
$$
\left\langle\psi, H_{M}^{N, \widetilde V_M^\omega}\psi\right\rangle
\leq
 \frac{A_0C_0 D_0 r_0}{L^{M\alpha}}  = \,\frac{\widetilde{C}_1 }{ 4L^{M\alpha}}.
$$
 On the other hand,  from  \eqref{eq:2.8}  and \eqref{eq:phi_at_zero} there is $M_2\in \mathbb Z_+$ such that for $M\geq M_2$
$$
\frac{\widetilde{C}_1}{L^{ M\alpha} }  <  \phi(\mu_{2}^{1}\cdot L^{-Md_w}) = \lambda_2^{M}=\mu,
$$
which means that the condition (\ref{eq:temple-condition}) is satisfied.

Further,
\begin{eqnarray}
\label{eq:temple_den_alloy_type}
\mu - \left\langle\psi, H_{M}^{N, \widetilde V_M^\omega}\psi\right\rangle &=&
\lambda_2^M-\left\langle\psi, H_{M}^{N, \widetilde V_M^\omega}\psi\right\rangle \nonumber\\
&\geq &
\frac{\widetilde{C}_1}{L^{ M\alpha}}-\frac{\widetilde{C}_1 }{ 4L^{M\alpha}}> \frac{\widetilde{C}_1}{2L^{ M\alpha}}.
\end{eqnarray}
Inserting (\ref{eq:temple_den_alloy_type}) into the formula in Temple's inequality (\ref{eq:temple-ineq}) we obtain:
\begin{eqnarray*}
 \lambda_1^{N,M,\widetilde V_M^\omega}  & \geq & \langle \psi,  H_{M}^{N, \widetilde V_M^\omega} \psi\rangle - \frac{2 \langle  H_{M}^{N, \widetilde V_M^\omega}\psi,  H_{M}^{N, \widetilde V_M^\omega} \psi\rangle-2\langle \psi, H_{M}^{N, \widetilde V_M^\omega} \psi\rangle^2}{\widetilde{C}_1L^{ -M\alpha}}\\
& \geq & \frac{1}{L^{Md}}\left[ \int_{\mathcal{K}^{\left\langle M\right\rangle}} \widetilde{V}^{\omega}_M(x)\mu({\rm d}x) - \frac{2\int_{\mathcal{K}^{\left\langle M\right\rangle}} (\widetilde{V}^{\omega}_M(x))^2\mu({\rm d}x)}{ \widetilde{C}_1L^{-M\alpha}} \right].
\end{eqnarray*}

\end{proof}

This estimate will be of use  for those configurations for which the number of sites with \linebreak $\xi_v(\omega) > \frac{D_0}{L^{M\alpha}}$ is large enough, i.e.\ on the set

$$
\mathcal{A}_{M,\delta}  =  \left\{ \omega: \#\left\{ v \in \cV_0^{\langle M\rangle} : \xi_v(\omega) > \frac{D_0}{L^{M\alpha}}\right\} \geq \delta \cdot k_0^{\langle M\rangle}\right\},
$$
where $\delta \in (0,1)$ is fixed.
Recall that $k_0^{\langle M\rangle}$ is the cardinality of $\cV_0^{\langle M\rangle}$; the estimate for $k_0^{\langle M\rangle}$ is given in \eqref{eq:number-points}.

We are ready to prove another lemma.

\begin{lemma}\label{lem:34-alloy-type}
 Let $\cK^{\langle \infty \rangle}$ be an USNF with the GLP and let the assumptions \textup{\textbf{(B)}}, \textup{\textbf{(Q1)}}-\textup{\textbf{(Q2)}}, \textup{\textbf{(W1)}}, and \textup{\textbf{(W5)}} hold. Let $\delta > 0$ be fixed (later it will be chosen not depending on $M$).  Then for any $M \geq M_2$ ($M_2$ comes from Lemma \ref{lem:l-33}) and $\omega \in \mathcal{A}_{M,\delta}$ we have
$$
\lambda_{1}^{N,M}(\omega)  \geq \frac{B_0 \delta}{L^{M\alpha}}.
$$
with $B_0:= (\widetilde C_1 L^{m_1d})/(8C_0 r_0)$.
\end{lemma}
\begin{proof}
 Let $M \geq M_2$. Recall that
\begin{equation}\label{eq:v_low}
\widetilde{V}_M^\omega(x)= A_0 \sum_{v \in \cV_0^{\langle M\rangle}}  \widetilde{\xi}_v(\omega) \1_{\cC_{m_1}(v)}(x), \quad x\in \mathcal K^{\langle M\rangle}.
\end{equation}
Since $m_1 < 0$ the sets $\mathcal C_{m_1}(v)$ are disjoint, so
$$
(\widetilde{V}_M^\omega(x))^2 = A_0^2 \sum_{v \in \cV_0^{\langle M\rangle}}  \widetilde{\xi}_v(\omega)^2  \1_{\cC_{m_1}(v)}(x), \quad x\in \mathcal K^{\langle M\rangle},
$$
and further 
\begin{equation}\label{eq:v_*_2}
\int_{\mathcal{K}^{\left\langle M\right\rangle}} (\widetilde{V}_M^\omega(x))^2\mu({ \rm d}x)  = A_0^2 \sum_{v \in \cV_0^{\langle M\rangle}}  \widetilde{\xi}_v(\omega)^2 \mu(\mathcal{K}^{\left\langle M\right\rangle} \cap \cC_{m_1}(v)) \leq A_0^2 r_0 L^{m_1d} \sum_{v \in \cV_0^{\langle M\rangle}}  \widetilde{\xi}_v(\omega)^2.
\end{equation}
Similarly, 
\begin{equation}\label{eq:int_v_low}
\int_{\mathcal{K}^{\left\langle M\right\rangle}} \widetilde{V}_M^\omega(x)\mu({ \rm d}x) = A_0 \sum_{v \in \cV_0^{\langle M\rangle}}  \widetilde{\xi}_v(\omega) \mu(\cC_{m_1}(v))  \geq A_0 L^{m_1d} \sum_{v \in \cV_0^{\langle M\rangle}}  \widetilde{\xi}_v(\omega).
\end{equation}
Observe also that
\begin{equation}\label{eq:nowe}
\sum_{v\in \cV_0^{\langle M\rangle}}\widetilde{\xi}^2_v(\omega)\leq\frac{D_0}{L^{ M\alpha}}
\sum_{v\in \cV_0^{\langle M\rangle}} \widetilde{\xi}_v(\omega),
\end{equation}
 so, by \eqref{eq:number-points} and the definition of the set $\mathcal A_{M,\delta},$ for $\omega \in \mathcal A_{M,\delta}$ we have
\begin{equation}\label{eq:nnowe}
\sum_{v\in \cV_0^{\langle M\rangle}}\xi_v(\omega)\geq \frac{D_0}{L^{ M\alpha}}\cdot\delta\cdot k_0^{\langle M\rangle}\geq
\frac{D_0}{L^{ M\alpha}}\cdot\delta\cdot L^{Md}.
\end{equation}
Inserting estimates (\ref{eq:v_*_2}) and  (\ref{eq:int_v_low}) into (\ref{lem:l-33-teza-alloy-type}), then using the definition of $D_0,$ the estimates \eqref{eq:nowe}, \eqref{eq:nnowe} and rearranging we get:
\begin{eqnarray*}
\lambda_{1}^{M,V^\omega_{M}} & \geq & \frac{1}{L^{Md}}\left[ A_0L^{m_1d}\sum_{v \in \cV_0^{\langle M\rangle}}  \widetilde{\xi}_v(\omega) - \frac{ 2 A_0^2 r_0 L^{m_1d}  \cdot \sum_{v \in \cV_0^{\langle M\rangle}}  \widetilde{\xi}_v(\omega)^2}{ \widetilde{C}_1L^{-M\alpha}} \right]\\
& \geq & \frac{1}{L^{Md}}\sum_{v \in V_0^{\langle M\rangle}}  \widetilde{\xi}_v(\omega) \left[ A_0L^{m_1d} - \frac{ 2 A_0^2 D_0 r_0 L^{m_1d}}{\widetilde{C}_1} \right]\\
& \geq &  \frac{\delta D_0}{L^{M\alpha}} \cdot A_0L^{m_1d} \cdot \left(1- \frac{1}{2C_0} \right) \\
& \geq &  \frac{\widetilde C_1 L^{m_1d} \delta}{8C_0 r_0 L^{M\alpha}}.
\end{eqnarray*}
and the proof is complete.
\end{proof}

Before we proceed with the proof of Theorem \ref{th:main-alloy-type-fractals}, we recall a Bernstein-type estimate for the binomial distribution (see e.g.  \cite[Lemma 3.5]{bib:KaPP3}) which will be used in the course of the proof.
\begin{proposition}\label{prop:lemma-35}
Let $(\Omega, \mathcal{F}, \mathbb{P})$ be a given probability space and let $S_n: \Omega \rightarrow \mathbb{R}$ be a random variable with the binomial distribution $B(n,p), n \geq 1, p \in (0,1).$ Then, for any $p,\gamma \in (0,1)$ such that $\gamma > p$,
\begin{equation}\label{eq:lemma-35}
\mathbb{P}\left(S_n \geq \gamma n \right) \leq \left(\left(\frac{1-p}{1-\gamma}\right)^{1-\gamma}
\left(\frac{p}{\gamma}\right)^\gamma\right)^n.
\end{equation}
\end{proposition}

\begin{proof}[ Conclusion of the proof of Theorem \ref{th:main-alloy-type-fractals}]
 Due to \eqref{eq:diag_est} we can  continue the estimate \eqref{eq:to} as
\begin{equation}\label{eq:trace-estimate-rewrited}
\Lambda(t) \leq  A^*  \mathbb{E}_{\mathbb{Q}}\left[{\rm e}^{-(t-1)\lambda_{1}^{N,M}(\omega)}; \mathcal{A}_{M,\delta} \right] + A^* \mathbb{Q}\left[\mathcal{A}^c_{M,\delta}\right], \quad t > 1.
\end{equation}
To estimate the integral over $\mathcal{A}_{M,\delta}$ we use Lemma \ref{lem:34-alloy-type}, which gives
 \[ \mathbb{E}_{\mathbb{Q}}\left[{\rm e}^{-(t-1)\lambda_1^{N,M}(\omega)};\mathcal A_{M,\delta}\right]
 \leq {\rm e}^{-(t-1) B_0\delta L^{ -M\alpha}}, \quad M \geq M_2\]
 and the second part can be estimated by Proposition \ref{prop:lemma-35}.
Indeed, let us note that
\begin{eqnarray*}
\mathcal{A}^c_{M,\delta} &  =  & \left\{ \omega: \#\left\{ v \in \cV_0^{\langle M\rangle} : \xi_v(\omega) > \frac{D_0}{L^{ M\alpha}}\right\} < \delta \cdot k_0^{\langle M\rangle}\right\}\\
& = & \left\{ \omega: \#\left\{ v \in \cV_0^{\langle M\rangle} : \xi_v(\omega) \leq \frac{D_0}{L^{M\alpha}}\right\} \geq (1- \delta) \cdot k_0^{\langle M\rangle}\right\}
\end{eqnarray*}
(recall that $k_0^{\langle M\rangle}$ is the cardinality of $\cV_0^{\langle M\rangle}$).
We want to use Proposition  \ref{prop:lemma-35} with  $n =  k_0^{\langle M\rangle},\; p=p_M = F_\xi(\frac{D_0}{L^{ M\alpha}})$  and $\gamma = 1-\delta.$ We  need to make sure that $\gamma > p_M \Leftrightarrow \delta < 1 - p_M.$ But since $\lim_{M \rightarrow \infty} p_M = F_\xi(0)  \in [0,1),$ this is not a problem. Moreover, as $\lim_{\delta \rightarrow 0}\frac{1}{1-\delta}\left(\frac{1}{\delta}\right)^\frac{\delta}{1-\delta} = 1$
 and $M \mapsto p_M$ is nonincreasing, we can find $M_3 \geq M_2$ and a universal number  $\delta_0>0$  such that for $M \geq M_3$ we have both  $1 - p_{M} > \delta_0$
and
\[
\left(\left(\frac{1}{\delta_0}\right)^{\frac{\delta_0}{1-\delta_0}}\frac{1}{1-\delta_0}\right)^{1-\delta_0} p_M^{(1-\delta_0)/2} \leq 1.
\]
 Therefore, from \eqref{eq:lemma-35} and the estimate (\ref{eq:number-points})  we obtain, for $M\geq M_3$:
\begin{eqnarray*}
\mathbb{Q}\left[\mathcal{A}^{c}_{M,\delta_0}\right] & \leq &    \left(\left(\frac{1-p_M}{\delta_0}\right)^{\delta_0}\left(\frac{p_M}{1-\delta_0}\right)^{1-\delta_0}\right)^{k_0^{\langle M\rangle}} \\
                                                    & \leq & \left(\left(\frac{1}{\delta_0}\right)^{\frac{\delta_0}{1-\delta_0}}\frac{\sqrt{p_M}}{1-\delta_0}\right)^{(1-\delta_0)k_0^{\langle M\rangle}} p_M^{(1-\delta_0)k_0^{\langle M\rangle}/2} \\
                                                    & \leq &  {p_M}^{(1-\delta_0) k_0^{\langle M\rangle}/2}\\
                                                    & \leq & \exp\left(-\frac{1-\delta_0}{2} \cdot L^{Md} \log \frac{1}{p_M}\right).
\end{eqnarray*} 
Let now
$$
c_1 =  B_0\delta_0, \qquad c_2 = (1-\delta_0)/2.
$$
Consequently, we obtain that there exists  $t_1 \geq t_0 \vee 1$ (recall that $t_0$ was defined in \eqref{eq:x_i_h})  such that for every $t \geq t_1$ and $M \geq M_3$ we have, with a suitable constant $ c_3>0,$
\begin{eqnarray*}
\Lambda(t) & \leq & {A^*}\left(\exp\left(-\frac{c_1(t-1)}{L^{ M\alpha}}\right) +  \exp\left(-c_2 \cdot L^{Md} \log \frac{1}{F_\xi(\frac{D_0}{L^{ M\alpha}})}\right)\right)\\
& \leq &  {A^*}\left(\exp\left(-\frac{ c_3t}{(L^{M-1})^{ \alpha}}\right) +  \exp\left(- c_2 \cdot L^{Md} \log \frac{1}{F_\xi(\frac{D_0}{L^{ M\alpha}})}\right)\right).
\end{eqnarray*}
We make $M$ depend on $t$ in such a way that $M \rightarrow \infty $ when $t \rightarrow \infty. $ Let  $x_t$ be defined by (\ref{eq:x_i_h}). Since $x_t\to\infty$ when $t\to\infty,$ there is a unique $M=M(t)$ for which
\begin{equation}\label{eq:M on t}
 L^{M-1} \leq x_t < L^M, \qquad  t \geq t_1.
\end{equation}
 Indeed, we define $M(t):=\left\lfloor \log_L x_t \right\rfloor +1$, $t \geq t_1$.
Then there is $ t_2 \geq t_1 $ such that for $  t \geq t_2$ we have $M \geq M_3.$ Thanks to (\ref{eq:t_cond}),
$$
(L^{M-1})^{\alpha} \leq x_t^{\alpha}=\left(\frac{t}{\log \frac{1}{F_\xi(D_0/x_t^{ \alpha})}}\right)^{ \frac{\alpha}{d+\alpha}}
$$
and
$$
\frac{t}{(L^{M-1})^{ \alpha}}  \geq  t^{ \frac{d}{d+\alpha}}\left( \log \frac{1}{F_\xi(D_0/x_t^{ \alpha})} \right)^{ \frac{\alpha}{d+\alpha}}, \qquad  t \geq t_2.
$$
Due to the fact that the function $ x \mapsto x^d\log \frac{1}{F_\xi(D_0/ x^{\alpha})}$ is increasing and  $x_t  < L^M$,  we also get
$$
L^{Md}\log \frac{1}{F_\xi(D_0/L^{ M\alpha})} \geq x_t^d \log \frac{1}{F_\xi(D_0/x_t^{ \alpha})}=t^{ \frac{d}{d+\alpha}}\left( \log \frac{1}{F_\xi(D_0/x_t^{ \alpha})} \right)^{ \frac{\alpha}{d+\alpha}},
$$
which implies, with $c_4=c_2\wedge c_3,$:
$$
\Lambda(t) \leq {2A^*}
\exp\left(-c_4 t^{\frac{d}{d+\alpha}}\left(\log \frac{1}{F_\xi(D_0/x_t^{ \alpha})} \right)^{ \frac{\alpha}{d+\alpha}}\right)={2A^*}
\exp\left(-c_4 t^{\frac{d}{d+\alpha}}\left(h(t) \right)^{ \frac{\alpha}{d+\alpha}}\right),\, \quad t \geq t_2.
$$
To conclude, we take the logarithm, rearrange, and pass to the limit $t\to\infty.$
Formula \eqref{eq:upper-poiss-1}  follows.  When an atom at 0 is present, then $\lim_{t\to\infty}h(t)= \log\frac{1}{F_\xi(0)}$ and \eqref{eq:upper-poiss-2} follows as well.

\end{proof}

\subsection{The lower bound for the Laplace transform of the IDS}

Now that we know the rate function in the asymptotics for the IDS, we can complement Theorem \ref{th:main-alloy-type-fractals} with a matching lower bound.

Our main tool in this section will be the Feynman--Kac semigroup $\big\{T_t^{D,\Delta, V}: t\geq 0\big\}$ of the process $X$ \emph{killed upon exiting a given complex} $\Delta \subset \cK^{\langle \infty \rangle}$ which is defined for a (non-random) potential $V \in \mathbb K_{loc}^X(\mathcal K^{\langle\infty\rangle})$:
\begin{equation} \label{eq:killed_F-k}
T_t^{D,\Delta,V} f(x) = \mathbf{E}^x \left[{\rm e}^{-\int_0^t V(X_s) {\rm d}s}f(X_t); t<\tau_{\Delta} \right], \quad f \in L^2(\Delta, \mu),  \ t>0;
\end{equation}
here $\tau_{\Delta}$ is the first exit time of the process from the complex $\Delta,$ i.e.  $\tau_{\Delta} := \inf \{t > 0: X_t \notin \Delta\}.$ Clearly, this semigroup has similar properties to $\big\{T_t^{D,M,V^\omega}: t\geq 0\big\}$ defined in Section \ref{sec:RSS}. The difference is that here we consider an arbitrary complex $\Delta$ instead of $\cK^{\langle M\rangle}$, and that the potential $V$ is deterministic. Denote: $H_{\Delta}^{D, V}  := - A^{D,\Delta,V}$, where $A^{D,\Delta,V}$ is the $L^2$-generator of the semigroup $\big\{T_t^{D,\Delta,V}: t\geq 0\big\}$, and let $\lambda_1^V(\Delta) := \inf\spec(H_{\Delta}^{D, V})$ be the corresponding ground state eigenvalue. If $V \equiv 0$, then we simply write $\lambda_1(\Delta)$. Moreover, if $X=Z$ in \eqref{eq:killed_F-k}, that is the underlying process is the Brownian motion, so we use a different symbol $\mu_1(\Delta)$ for the corresponding ground state eigenvalue.

Before we proceed we note that the bound (\ref{eq:sub_est}) in our particular case gives:
for every $t_0 > 0$ there exists $ c=c(t_0)$ such that
\begin{equation}\label{eq:ps0}
p(t,x,x) \leq c  t^\frac{-d}{\alpha},  \qquad t \geq t_0, \  x \in \cK^{\langle \infty \rangle}.
\end{equation}
We also need a version of  \cite[Lemma 4.1]{bib:KaPP3}, adapted to the present setting. We just replace the space $\mathbb{R}^d$ with $\mathcal{K}^{\langle \infty \rangle}$,  and repeat the proof verbatim. The statement goes as follows.

\begin{lemma}\label{prop:lemma-41}
Let $ 0 \leq V \in \mathbb{K}^X_{loc}\cap L^1(\mathcal{K}^{\langle \infty \rangle},\mu)$. Then for any complex $\Delta \subset \mathcal{K}^{\langle \infty \rangle}$ we have
$$
\lambda_1^V(\Delta) \leq \lambda_1(\Delta)  + {\rm e}\cdot\sup_{x\in\mathcal K^{\langle \infty\rangle}} p(s,x,x) ||V||_1, \mbox{ with } s:= \frac{1}{ \lambda_1(\Delta)}.
$$
\end{lemma}

In this section we use another representation od the IDS --  it  arises as the limit of finite-volume expressions with Dirichlet boundary conditions (cf.\ Corollary \ref{coro:conv}). More precisely, the Laplace transform $\Lambda(t)$ of the IDS can be recovered from the formula
\begin{equation}\label{eq:L_t_inty_lower}
\Lambda(t)=\lim_{M\to\infty}  \mathbb{E}_{\mathbb{Q}} \Lambda_{M}^{D,  V^\omega}(t),
\end{equation}
where
$$
 \Lambda_{M}^{D,  V^\omega}(t) = \frac{1}{L^{Md}}\int_{\mathcal K^{\langle M\rangle}}p(t,x,x)\mathbf{E}^{x,x}_t \left[ {\rm e}^{-\int_0^t V^{\omega}(X_s){\rm d}s};  t< \tau_{\mathcal K^{\langle M\rangle}}\right]\mu({\rm d}x).
$$
\begin{theorem}\label{th:main-2}
  Let $\cK^{\langle \infty \rangle}$ be an USNF with the GLP and let the assumptions \textup{\textbf{(B)}}, \textup{\textbf{(Q1)}}-\textup{\textbf{(Q2)}} and \textup{\textbf{(W1)}}-\textup{\textbf{(W4)}} hold. Let $h$ be given by (\ref{eq:x_i_h}).  Then there exists $K_1'>0$ such that
\begin{equation}\label{eq:lower-poiss-1}
\liminf_{t\to\infty} \frac{\log  \Lambda(t)}{t^{\frac{d}{ d+\alpha}} (h(t))^{\frac{\alpha}{d+\alpha}}}\geq -K_1',
\end{equation}
Again, when $F_\xi(0) > 0$, then
\begin{equation}\label{eq:lower-poiss-2}
\liminf_{t\to\infty} \frac{ \log  \Lambda(t)}{t^{\frac{d}{ d+\alpha}}}\geq -K_1' \left(\log\frac{1}{{F_\xi(0)}}\right)^{\frac{ \alpha}{d+\alpha}}.
\end{equation}
\end{theorem}

\begin{proof} Fix $M\in\mathbb Z_+,$ $M\geq M_0,$  where the integer $M_0$ comes from \textbf{(W4)}. 
Given (\ref{eq:L_t_inty_lower}) we have also
$$
\Lambda(t) = \lim_{n \rightarrow \infty} \mathbb{E}_{\mathbb{Q}} \Lambda_{M+n}^{D, V^\omega}(t).
$$
For fixed $n \geq 1,$ the set $\mathcal{K}^{\langle M + n \rangle}$  consists of $L^{nd}=N^{n}$ $M$-complexes   meeting only through their vertices. Denote them $ {\Delta}^{(1)},\ldots,{\Delta}^{(N^n)}$. As $ {\Delta}^{(i)} \subset \mathcal{K}^{\langle M + n \rangle}$ we get:
\begin{eqnarray*}
 \mathbb{E}_{\mathbb{Q}} \Lambda_{M+n}^{D,  V^\omega}(t)  & = & \frac{1}{L^{(M + n)d}}\int_{\mathcal K^{\langle M + n\rangle}}p(t,x,x) \mathbb{E}_{\mathbb{Q}} \mathbf{E}^{x,x}_t \left[ {\rm e}^{-\int_0^t V^{\omega}(X_s){\rm d}s}\mathbf{1}_{\{\tau_{\mathcal K^{\langle M + n\rangle}} > t\}}\right]\mu({\rm d}x)\\
& = & \frac{1}{L^{(M + n)d}}\sum_{i=1}^{N^n}\int_{ {\Delta}^{(i)}}p(t,x,x) \mathbb{E}_{\mathbb{Q}} \mathbf{E}^{x,x}_t \left[ {\rm e}^{-\int_0^t V^{\omega}(X_s){\rm d}s}\mathbf{1}_{\{\tau_{\mathcal K^{\langle M + n\rangle}} > t\}}\right]\mu({\rm d}x)\\
& \geq & \frac{1}{L^{(M + n)d}}\sum_{i=1}^{N^n}\int_{ {\Delta}^{(i)}}p(t,x,x) \mathbb{E}_{\mathbb{Q}} \mathbf{E}^{x,x}_t \left[ {\rm e}^{-\int_0^t V^{\omega}(X_s){\rm d}s}\mathbf{1}_{\{\tau_{ {\Delta}^{(i)}} > t\}}\right]\mu({\rm d}x)\\
& \geq & \inf_{i} \mathbb{E}_{\mathbb{Q}} [ \Lambda_{ {\Delta}^{(i)}}^{D,V^\omega}(t)],
\end{eqnarray*}
where by $\Lambda_{ {\Delta}^{(i)}}^{D,V^\omega}(t)$ we have denoted the expression
\[\frac{1}{\mu( {\Delta}^{(i)})}\int_{ {\Delta}^{(i)}}p(t,x,x) \mathbb{E}_{\mathbb{Q}} \mathbf{E}^{x,x}_t \left[ {\rm e}^{-\int_0^t V^{\omega}(X_s){\rm d}s}\mathbf{1}_{\{\tau_{ {\Delta}^{(i)}} > t\}}\right]\mu({\rm d}x).\]
Let $\kappa>0$ be fixed, to be chosen later on. Fix some $i_0\in\{1,..., N^n\}$ and let
$$
\mathcal{M}_{i_0}=\{\omega: \forall_{v \in \cV_0({ \mathcal O}_{i_0}^{M_0})} \; \xi_v (\omega)\leq \kappa \},
$$
where ${ \mathcal O}_{i_0}^{M_0}
 $denotes the $1-$vicinity of $ {\Delta}^{(i_0)}$ in the metric $d_{M_0}$ (i.e.\ those points $x$ that belong to  $M_0-$complexes  with at least one vertex in $ {\Delta}^{(i_0)}$), and for a set $A\subset\mathcal K^{\langle\infty\rangle},$ write $\cV_0(A)=  \cV_0^{\langle\infty\rangle}\cap A.$ In particular,
\begin{equation}\label{eq:calc_m_i-0}
 \mathbb{E}_{\mathbb{Q}} [ \Lambda_{ {\Delta}^{(i_0)}}^{D,V^\omega}(t)]\geq \mathbb{E}_{\mathbb{Q}} [ \Lambda_{ {\Delta}^{(i_0)}}^{D,V^\omega}(t) \mathbf{1}_{\mathcal{M}_{i_0}}].
\end{equation}
Fix a trajectory $ X_t$ originating at $x \in  {\Delta}^{(i_0)}$ and not leaving the set $ {\Delta}^{(i_0)}$ up to time $t$. Because of the assumption { \bf (W4)}, for $\omega \in \mathcal{M}_{i_0}$  we have that
$$
V^\omega(X_s)=\sum_{v \in \cV_0({ \mathcal O}_{i_0}^{M_0})}\xi_v(\omega)W(X_s,v) \leq \kappa \sum_{v \in \cV_0({\mathcal O}_{i_0}^{M_0})}W(X_s,v), \qquad  s \leq t.
$$
 Denote:
$$
V_\kappa(x) = \kappa \sum_{v \in \cV_0({\mathcal O}_{i_0}^{M_0})}W(x,v), \quad x \in \cK^{\langle\infty\rangle}.
$$
It follows that for such a trajectory
$$
{\rm e}^{-\int_0^t V^\omega(X_s){\rm d}s} \geq {\rm e}^{-  \int_0^t V_{\kappa}(X_s){\rm d}s }
$$
therefore on the set $\{\tau_{ {\Delta}^{(i_0)}} > t\}$ it holds:
\begin{eqnarray}\label{eq:tal}
\mathbb{E_Q}\left[{\rm e}^{-\int_0^t V(X_s, \omega){\rm d}s}\mathbf{1}_{\mathcal{M}_{i_0}} \right] &\geq& \mathbb{E_Q}\left[{\rm e}^{-\int_0^t V_{\kappa}(X_s){\rm d}s} \cdot\mathbf{1}_{\mathcal{M}_{i_0}} \right]
= {\rm e}^{- \int_0^t V_{\kappa}(X_s){\rm d}s} \mathbb{Q}(\mathcal{M}_{i_0}).
\end{eqnarray}
Consequently, from  (\ref{eq:tal}) and   (\ref{eq:calc_m_i-0}):
\begin{eqnarray}\label{eq:l2}
 \mathbb{E}_{\mathbb{Q}}  \Lambda_{ {\Delta}^{(i_0)}}^{D,V^\omega}(t)] \geq
\frac{\mathbb{Q}(\mathcal{M}_{i_0})}{\mu( {\Delta}^{(i_0)})}\int_{ {\Delta}^{(i_0)}}p(t,x,x) \mathbf{E}^{x,x}_t \left[ {\rm e}^{-\int_0^t V_{\kappa}(X_s){\rm d}s}\mathbf{1}_{\{\tau_{ {\Delta}^{(i_0)}} > t\}}\right]\mu({\rm d}x).
\end{eqnarray}
As the number of points in $ \cV_0({\mathcal O}_{i_0}^{M_0})$ does not exceed
$ \widetilde C_0L^{Md}$ with some $\widetilde C_0>0$ (cf.\ (\ref{eq:number-points})),  we have
\begin{eqnarray*}
\mathbb{Q}[\mathcal{M}_{i_0}] = \mathbb{Q}[\forall_{v \in \cV_0({\mathcal O}_{i_0}^{M_0})}\; \xi_v \leq \kappa]\geq F_\xi(\kappa)^{\widetilde C_0L^{Md}}
= {\rm e}^{-\widetilde C_0L^{Md} \log \frac{1}{F_\xi(\kappa)}}
\end{eqnarray*}
and
we obtain
\begin{align}\label{eq:final_eq}
 \mathbb{E}_{\mathbb{Q}} [ \Lambda_{ {\Delta}^{(i_0)}}^{D,V^\omega}(t)] \geq
\frac{1}{\mu( {\Delta}^{(i_0)})}\int_{ {\Delta}^{(i_0)}}p(t,x,x) \mathbf{E}^{x,x}_t \left[ {\rm e}^{-\int_0^t V_{\kappa}(X_s){\rm d}s}\mathbf{1}_{\{\tau_{ {\Delta}^{(i_0)}} > t\}}\right]\mu({\rm d}x) \cdot
{\rm e}^{- \widetilde C_0L^{Md} \log \frac{1}{F_\xi(\kappa)}}.
\end{align}
Denote
$$
I_{i_0} := \int_{ {\Delta}^{(i_0)}}p(t,x,x) \mathbf{E}^{x,x}_t \left[ {\rm e}^{-\int_0^t  V_{\kappa}(X_s){\rm d}s}\mathbf{1}_{\{\tau_{ {\Delta}^{(i_0)}} > t\}}\right]\mu({\rm d}x).
$$
This integral is the trace of the operator $T_t^{D,\Delta^{(i_0)},V_\kappa}$, see \eqref{eq:killed_F-k}.
Consequently, it is not smaller than the principal eigenvalue ${\rm e}^{- \lambda_1^{V_\kappa}(\Delta^{(i_0)})t}$ of that operator. Using Lemma (\ref{prop:lemma-41}),  we get that for $s = \frac{1}{\lambda_1( {\Delta}^{(i_0)})}$ it holds
$$
I_{i_0} \geq {\rm e}^{-t\lambda_1^{V_\kappa}( {\Delta}^{(i_0)})} \geq {\rm e}^{-t(\lambda_1({\Delta}^{(i_0)}) + e \sup_{x\in\mathcal K^{\langle\infty\rangle}}p(s,x,x)||V_\kappa||_1)}.
$$
Now using { \bf (W4)} we have that
$$
||V_\kappa||_1 = \kappa \int_{\mathcal{K}^{\langle \infty \rangle}}\sum_{v \in \cV_0({ \mathcal O}_{i_0}^{M_0})}W(x,v)\mu({\rm d}x)
               = \kappa \sum_{v \in \cV_0({ \mathcal O}_{i_0}^{M_0})} \int_{\mathcal{C}_{M_0}(v)}W(x,v)\mu({\rm d}x)  \leq   c_1 L^{Md} \cdot \kappa.
$$
Moreover,  from (\ref{eq:ps0}), with $s = 1/\lambda_1( {\Delta}^{(i_0)})$ we get
$$
p(s,x,x) \leq c_2\lambda_1( {\Delta}^{(i_0)})^{ \frac{d}{\alpha}},
$$
with certain $c_2>0$ not depending on $M.$  Furthermore, by \cite[Theorem 3.4]{bib:Chen-Song},
\[
 \lambda_1({\Delta}^{(i_0)}) \leq \phi(\mu_1({\Delta}^{(i_0)})),
\]
where $\mu_1({\Delta}^{(i_0)})$ is the ground state eigenvalue corresponding to the killed Brownian motion in ${\Delta}^{(i_0)}$. 
Due to translation invariance of the killed Brownian motion,
 scaling and \eqref{eq:phi_at_zero}, there exists $M_1 \geq M_0$ such that for $M \geq M_1$
 \begin{eqnarray*}
\phi(\mu_1( {\Delta}^{(i_0)})) = \phi(\mu_1(\mathcal K^{\langle M\rangle}))=\phi(L^{-Md_w}\mu_1(\mathcal K^{\langle 0\rangle}))
 \leq  C_2 L^{-M \alpha}\phi(\mu_1(\mathcal K^{\langle 0\rangle}))
 = c_3 L^{ -M\alpha}.
 \end{eqnarray*}
with a constant $c_3>0$ not depending on $M.$
Consequently,
$$
I_{i_0} \geq \exp\left(-t(c_3L^{-M\alpha} + c_4\kappa )\right).
$$
Choosing $\kappa = \frac{D_0}{L^{ M \alpha}}$ (with  $D_0$ coming from \eqref{eq:d_0_alloytype}) and returning to (\ref{eq:final_eq}), we see that it holds
\begin{eqnarray*}
 \mathbb{E}_{\mathbb{Q}} \left[ \Lambda_{ {\Delta}^{(i_0)}}^{D,V^\omega}(t)\right]  & \geq & \frac{1}{L^{Md}}\exp\left(- c_5  \frac{t}{ L^{M\alpha}}-  \widetilde C_0  L^{Md}\log \frac{1}{F_\xi(\frac{D_0}{L^{M \alpha}})}\right) \\
& \geq & \frac{1}{L^{Md}}\exp\left(- c_6  \left(\frac{t}{L^{(M+1)\alpha}}+L^{Md}\left(\frac{L^M}{L^{M+1}}\right)^{\alpha}\log \frac{1}{F_\xi(\frac{D_0}{L^{M\alpha}})}\right)\right)\\
& \geq & \frac{1}{L^{Md}}\exp\left(-\frac{ c_6}{L^{(M+1)\alpha}}\left(t+L^{ M(d+\alpha)}\log \frac{1}{F_\xi(\frac{D_0}{L^{M\alpha}})}\right)\right).
\end{eqnarray*}

Similarly to the proof of the upper bound, let us use (\ref{eq:x_i_h}) and take the unique $M = M(t)$ for which
\begin{equation}\label{eq:M on t-lower}
L^M \leq x_t < L^{M+1}, \qquad t \geq t_0.
\end{equation}
 As $M \rightarrow \infty $ when $t \rightarrow \infty,$ there is $t_1 \geq t_0 $ such that for $ t \geq t_1$ we have $M \geq M_1.$  Given (\ref{eq:x_i_h}) and  \eqref{eq:t_cond},
$$
L^{ M(d+\alpha)}\log \frac{1}{F_\xi(\frac{D_0}{ L^{M\alpha}})} \leq t.
$$
It follows
$$
 \mathbb{E}_{\mathbb{Q}} \left[ \Lambda_{ {\Delta}^{(i_0)}}^{D,V^\omega}(t)\right]  \geq \frac{1}{L^{Md}}\exp\left(-\frac{ 2c_6t}{L^{(M+1)\alpha}}\right).
$$
It is essential that the constant $ c_6$ does not depend on $i_0.$
Thanks to (\ref{eq:t_cond}) it holds:
$$
\frac{t}{(L^{M+1})^{\alpha}} \leq \frac{t}{x_t^{ \alpha}}=t^{ \frac{d}{d+\alpha}} \left(\log \frac{1}{F_\xi(D_0/x_t^{ \alpha})}\right)^{\frac{\alpha}{d+\alpha}}=t^{ \frac{d}{d+\alpha}}\left(h(t)\right)^{ \frac{\alpha}{d+\alpha}}
$$
and
we obtain, for sufficiently large $t$ (recall $M=M(t)$ goes to $\infty$ with $t$)
\begin{eqnarray*}
\log[\Lambda(t)]
&\geq & -\log L^{Md}- 2c_6  t^{ \frac{d}{d+\alpha}} (h(t))^{ \frac{\alpha}{d+\alpha}}.
\end{eqnarray*}
We now divide both sides by $t^{ \frac{d}{d+\alpha}}\left(h(t)\right)^{ \frac{\alpha}{d+\alpha}}$
and pass to the limit $t\to\infty.$
To get rid of the unwanted term, we need to verify that
\[\lim_{t\to\infty} \frac{\log L^{Md}}{t^{\frac{d}{d+\alpha}}\left(h(t)\right)^{\frac{\alpha}{d+\alpha}}} =0.\]
Indeed, it is so: we have $L^{Md}\sim x_t^d,$ and
$t=x_t^{ d+\alpha}h(t).$ Therefore $(d+\alpha)\log x_t = \log t - \log h(t),$ meaning that $\log L^{Md}$ behaves, up to a constant,
as $\log t -\log h(t).$ As $h(t)$ has a positive (possibly infinite) limit when $t\to\infty,$ the desired statement, and consequently the Theorem, follow.
\end{proof}

\subsection{The bounds for the IDS}
To transform the estimates (\ref{eq:upper-poiss-1}) and (\ref{eq:lower-poiss-1}) into respective statements for the IDS itself from the statement of Theorem \ref{th:Lifshitz-IDS}, we use Tauberian theorems of mixed type proven in \cite[Theorem 5.1]{bib:KaPP3} an we conclude with (\ref{eq:main-statement}), in the same manner as it was done in \cite{bib:KaPP3}. \hfill$\Box$

\appendix

\section{Technical lemmas on subordinate processes}

We need the following technical results.

\begin{lemma}
\label{lem:supseries}
Let $\cK^{\langle \infty \rangle}$ be an USNF with the GLP and let the assumption \textup{\textbf{(B)}} hold. For every $t>0$ and $a>1$ we have
\begin{equation*}
\sum_{M=1}^{\infty} \sup_{x \in \cK^{\langle \infty\rangle}} \mathbf{P}^x \left[\sup_{s\leq t} |X_s - x| > a^M \right] <\infty.
\end{equation*}
\end{lemma}

\begin{proof}
The proof follows the lines of that of \cite[Lemma 2.3]{bib:KaPP2}. The only difference is that now we use the Euclidean distance instead of the geodesic one, and apply the sub-Gaussian estimates from \cite[Lemma 5.6, Remark 3.7]{bib:FHK}. At the end, we also use the summation property \eqref{eq:sub} as before.
\end{proof}

Let us recall that $\Delta_{M,i}$ for $1 \leq i \leq N$ denote the $M$-complexes in $\cK^{\langle M+1 \rangle}$, see Definition \ref{def:Mcomplex} (9).

\begin{lemma}
\label{lem:ceemteandtrace} Let $\cK^{\langle \infty \rangle}$ be an USNF with the GLP and let the assumption \textup{\textbf{(B)}} hold. We have the following.
\begin{itemize}
\item[(a)] For
\begin{equation}
\label{eq:ceemte}
C(M,t):= \sup_{x,y \in \cK^{\langle M \rangle}} \sum_{\substack{y^{\prime} \in \pi_{M}^{-1}(y) \\ y' \notin \cK^{\left\langle M+1 \right\rangle}}} p(t,x,y'), \quad t>0, \quad M \in \mathbb{Z}
\end{equation}
we have
\begin{equation*}
\sum_{M=1}^{\infty} C(M,t) < \infty, \quad t>0;
\end{equation*}
in particular, for every $t>0$, $C(M,t) \to 0$ as $M \to \infty$.

\item[(b)]For every $t>0$,
\begin{equation}
\label{eq:traceest}
\sum_{M=1}^{\infty} \frac{1}{\mu \left(\cK^{\langle M \rangle}\right)} \int_{\cK^{\langle M \rangle}} \left| p(t,x,x) - p_M(t,x,x) \right| \mu({\rm d}x) < \infty;
\end{equation}
in particular
\begin{equation}
\label{eq:tracelimit}
\frac{1}{\mu \left(\cK^{\langle M \rangle}\right)} \int_{\cK^{\langle M \rangle}} \left| p(t,x,x) - p_M(t,x,x) \right| \mu({\rm d}x) \to 0, \quad as \ M \to \infty.
\end{equation}
\end{itemize}
\end{lemma}

\begin{proof}

By Tonelli's theorem and \cite[formula (3.3) and Lemma 3.5]{bib:MO}, we get
\begin{align*}
C(M,t) & = \sup_{x,y \in \cK^{\langle M \rangle}} \int_0^{\infty} \sum_{\substack{y^{\prime} \in \pi_{M}^{-1}(y) \\ y' \notin \cK^{\left\langle M+1 \right\rangle}}} g(t,x,y') \eta_t({\rm d}u) \\
       & \leq c_1 \int_0^{\infty} L^{-d M} \left( \frac{L^M}{t^{1/d_w}} \vee 1\right)^{d - \frac{d_w}{d_J-1}} \exp\left(-c_2 \left( \frac{L^M}{t^{1/d_w}} \vee 1\right)^{\frac{d_w}{d_J-1}}\right) \eta_t({\rm d}u) \\
			 & \leq c_1 \int_0^{L^{d_{w}M}} L^{-d M} \left( \frac{L^M}{u^{1/d_w}}\right)^{d - \frac{d_w}{d_J-1}} \exp{\left(-c_2 \left( \frac{L^M}{u^{1/d_w}}\right)^{\frac{d_w}{d_J-1}}\right)} \eta_t({\rm d}u) \\
       & \ \ \ \ \ \ \ \ \ \ \ \ \ \ + c_3 L^{-d M} \eta_{t} \left(L^{d_{w}M}, \infty\right).
\end{align*}
Further steps of the proof follow exactly the reasoning in the proof of \cite[Lemma 2.5 (b)]{bib:KaPP2}.

The proof of (b) is similar to that of \cite[Lemma 2.7]{bib:KaPP2}, i.e.\ we start with
\begin{align*}
& \frac{1}{\mu \left(\cK^{\langle M \rangle}\right)} \int_{\cK^{\langle M \rangle}} \left| p(t,x,x) - p_M(t,x,x) \right| \mu({\rm d}x) \\
& \leq \frac{N}{\mu \left(\cK^{\langle M \rangle}\right)} \int_{\cK^{\langle M \rangle}} \sup_{\substack{x^{\prime} \in \pi_{M}^{-1}(x) \\ x' \in \cK^{\langle M+1 \rangle} \backslash \cK^{\langle M \rangle}}} p(t,x,x') \mu({\rm d}x) + \frac{1}{\mu \left(\cK^{\langle M \rangle}\right)} \int_{\cK^{\langle M \rangle}} \sum_{\substack{x^{\prime} \in \pi_{M}^{-1}(x) \\x' \notin \cK^{\langle M+1 \rangle}}} p(t,x,x') \mu({\rm d}x)\\
& =: \cA_M(t) + \cB_M(t)
\end{align*}
and first observe that $\cB_M(t)$ is a term of convergent series by part (a). To prove that the same is true for $\cA_M(t)$, we only need to modify the estimate in the proof of the quoted lemma. The difference is that the domain of the integration in $\cA_M(t)$ has to be divided into two different sets $E_M^1$ and $E_M^2$. In the present general case, we cannot use geodesic balls. We will apply the graph distance here.

Let us denote $V_M^{\langle M \rangle} = \{v_1, ..., v_k\}$ and
\begin{equation}
E_M^1 = \bigcap_{i=1}^{k} \left\{ y \in \cK^{\langle M \rangle} : d_{\lfloor M/2\rfloor} (y, v_i) > 1\right\}, \quad E_M^2 = \cK^{\langle M \rangle} \backslash E_M^1.
\end{equation}
Recall that $\sup_{(x,y) \in \cK^{\langle \infty \rangle} \times \cK^{\langle \infty \rangle}} p(t,x,y) < \infty$ by Lemma \ref{lem:regularity_of_p} (a).
Since $E_M^2$ consists of $k$ $\lfloor M/2 \rfloor$-complexes (each one attached to one of the vertices from $V_M^{\langle M \rangle}$), we have that
\begin{equation}
\frac{\mu\left(E_M^2\right)}{\mu \left(\cK^{\langle M \rangle}\right)} = \frac{k \cdot N^{\lfloor M/2 \rfloor}}{N^M}.
\end{equation}
In consequence, the part of $\cA_M(t)$ including the integral over $E_M^2$ is a term of a convergent series.

When $x \in E_M^1$ and $x^{\prime} \in \pi_{M}^{-1}(x) \backslash \cK^{\langle M \rangle}$, then $d_{\lfloor M/2 \rfloor} (x, x') >2$, so from \cite[Lemma A.2]{bib:KOP} we have $|x-x'|> c_4 L^{M/2}$. Using the subordination formula for the density $p(t,x,y)$ and the subgaussian estimates from \cite{bib:Kum} for the density $g(u,x,y)$, we then get

\begin{equation*}
\frac{1}{\mu \left(\cK^{\langle M \rangle}\right)} \int_{E_M^1} \sup_{\substack{x^{\prime} \in \pi_{M}^{-1}(x) \\ x' \in \cK^{\langle M+1 \rangle} \backslash \cK^{\langle M \rangle}}} p(t,x,x') \mu({\rm d}x) \leq c_5 \int_0^{\infty} u^{-d_s / 2} {\rm e}^{-c_6 \left(\frac{L^{M/2}}{u^{1/d_w}}\right)^{\frac{d_w}{d_J-1}}} \eta_t({\rm d}u).
\end{equation*}
Collecting these estimates, we obtain
\begin{equation*}
\cA_M(t) \leq c_7 \left(N^{-M/2}+ \int_0^{\infty} u^{-d_s / 2} {\rm e}^{-c_6 \left(\frac{L^{M/2}}{u^{1/d_w}}\right)^{\frac{d_w}{d_J-1}}} \eta_t({\rm d}u)\right),
\end{equation*}
with the constants $c_6, c_7$ independent of $M$, and from this point the proof can be continued in the same manner as that of \cite[Lemma 2.7]{bib:KaPP2}.
\end{proof}

\section{Proofs of the statements from Section  \ref{sec:existence}}

\begin{proof}[Proof of Proposition \ref{thm:convergence_with_star}]
Fix $t>0$. For $\mu$-almost all $x \in \cK^{\langle M+1\rangle},$ by Lemmas \ref{lem:bridges}(a),  \ref{lem:monoto}, and the inclusion $\pi_{M+1}^{-1}(x) \subset \pi_M^{-1}\left(\pi_M(x)\right)$, we may write
\begin{align*}
&p_{M+1}(t,x,x) \mathbf{E}_{M+1,t}^{x,x} [\mathbb{E}_{\mathbb{Q}} {\rm e}^{-\int_0^t V^{\omega}_{M+1}(X^{M+1}_s){\rm d}s} ]\\
& = \sum_{x' \in \pi_{M+1}^{-1} (x)} p(t,x,x') \mathbf{E}_{t}^{x,x'} \left[\mathbb{E}_{\mathbb{Q}} {\rm e}^{-\int_0^t V^{\omega}_{M+1}(\pi_{M+1}(X_s)){\rm d}s}  \right] \\
& \leq \sum_{x' \in \pi_{M}^{-1} \left(\pi_M(x)\right)} p(t,x,x') \mathbf{E}_{t}^{x,x'} \left[\mathbb{E}_{\mathbb{Q}} {\rm e}^{-\int_0^t V^{\omega}_{M+1}(\pi_{M+1}(X_s)){\rm d}s}  \right] \\
& \leq \sum_{x' \in \pi_{M}^{-1} \left(\pi_M(x)\right)} p(t,x,x') \mathbf{E}_{t}^{x,x'} \left[\mathbb{E}_{\mathbb{Q}} {\rm e}^{-\int_0^t V^{\omega}_{M}(\pi_{M}(X_s)){\rm d}s}  \right].
\end{align*}
Using this estimate, the definition of $\Lambda_{M^*}^{N}(t,\omega)$ and Lemma \ref{lem:bridges}(b), we may now follow the argument in the second part of the proof of \cite[Theorem 3.1]{bib:KaPP2}, getting that
$$
\mathbb{E}_{\mathbb{Q}}[\Lambda_{M+1}^{N,  V_{M+1}^\omega}(t)] \leq \mathbb{E}_{\mathbb{Q}}[ \Lambda_{M}^{N,  V_M^\omega}(t)], \quad M \in \Z.
$$
Since $\mathbb{E}_{\mathbb{Q}}\Lambda_{M}^{N,  V_M^\omega}(t) \geq 0$, it converges to a finite limit $\Lambda(t)$ as $M \to \infty$. This completes the proof.
\end{proof}

Proofs of Lemmas \ref{lem:commonlimit} and \ref{lem:variances} given below are of technical nature. They follow the steps and ideas from the proofs of Proposition 3.1 and Lemmas 3.1-3.2 in \cite{bib:KaPP2}. The main difference here is that we now work with  different type of random potentials and the state space is now a general USNF (let us emphasize that the Sierpi\'nski triangle which was studied in the quoted paper is one of the simplest planar nested fractals with high regularity). This causes some extra geometric issues, which are solved by using the graph distance (the geodesic metric may not be defined at all) and the comparison principle from \cite[Lemma A.2]{bib:KOP}. Therefore, we will follow only the main steps of the proofs which are affected by these changes, focusing on the most critical differences.

\begin{proof}[Proof of Lemma \ref{lem:commonlimit}]
We first show (a). By following the argument at the beginning of \cite[Proposition 3.1]{bib:KaPP2} we obtain that for each fixed $t>0$ there exists a constant $c=c(t)$ such that
$$
\mathbb{E}_{\mathbb{Q}} \left(\Lambda_{M}^{D,  V^\omega}(t) - \Lambda_{M}^{N,  V^\omega} (t)  \right)^2 \leq c \left( R_{1,M}(t) + R_{2,M}(t)\right), \quad M \in \mathbb{Z}_{+},
$$
where
\begin{align*}
R_{1,M}(t) &= \frac{1}{\mu \left(\cK^{\langle M \rangle}\right)} \int_{\cK^{\langle M \rangle}}  p(t,x,x) \mathbf{E}_{t}^{x,x} \left[t\geq \tau_{\cK^{\langle M\rangle}} \right] \mu({\rm d}x)\\
R_{2,M}(t) &= \frac{1}{\mu \left(\cK^{\langle M \rangle}\right)} \int_{\cK^{\langle M \rangle}} \sum_{x' \in \pi_M^{-1}(x), x' \neq x} p(t,x,x') \mu({\rm d}x) \\
           &= \frac{1}{\mu \left(\cK^{\langle M \rangle}\right)} \int_{\cK^{\langle M \rangle}} (p_M(t,x,x) - p(t,x,x)) \mu({\rm d}x).
\end{align*}
Note that the above bound does not depend on $\omega$. By Lemma \ref{lem:ceemteandtrace} (b), $R_{2,M}(t)$ is the term of a convergent series, so we only need to estimate $R_{1,M}(t)$.

Denote the vertices from $\cV_M^{\langle M \rangle}$ as $v_i$, $1\leq i\leq k$, and let $\Delta_{\left\lfloor M/2 \right\rfloor,v_i} \subset \cK^{\langle M \rangle}$, be the $\left\lfloor M/2 \right\rfloor$-complex attached to $v_i$.
If the process starts from $x \in \cD_M:=\cK^{\langle M \rangle} \backslash \bigcup_{i=1}^{k} \Delta_{\left\lfloor M/2 \right\rfloor,v_i}$, then, using \cite[Lemma A.2]{bib:KOP}, we have
\begin{equation*}
\left\{t\geq \tau_{\cK^{\langle M \rangle}}\right\} \subseteq \left\{ \sup_{0<s\leq t} d_{\left\lfloor M /2 \right\rfloor}(x,X_s) >2\right\} \subseteq \left\{ \sup_{0<s\leq t} |x-X_s| > c_{1} L^{M/2}\right\},
\end{equation*}
with a constant $c_1$, independent of $M$. We also have
\begin{equation*}
\left\{\sup_{0<s\leq t} |x-X_s| > c_{1} L^{M/2} \right\} \subseteq \left\{\sup_{0<s\leq t/2} |x-X_s| > c_{1} L^{M/2} \right\} \cup \left\{\sup_{t/2<s\leq t} |x-X_s| > c_{1} L^{M/2} \right\}.
\end{equation*}
Recall that $\mu \left( \cK^{\langle M \rangle} \backslash \cD_M \right) = k N^{\left\lfloor M/2 \right\rfloor}$. Then, by using the upper bound in Lemma \ref{lem:regularity_of_p}(a), the formula \eqref{eq:bridge} and the symmetry of the bridge measure, we get
\begin{align*}
\int_{\cK^{\langle M \rangle}}  p(t,x,x) & \mathbf{P}^{t}_{x,x} \left[t\geq \tau_{\cK^{\langle M\rangle}} \right] \mu({\rm d}x)\\
& \leq c_{2} \mu\left( \cK^{\langle M \rangle} \backslash \cD_M \right) + \int_{\cD_M} p(t,x,x) \mathbf{P}^{x,x}_{t}\left[t\geq \tau_{\cK^{\langle M \rangle}}\right] \mu({\rm d}x)\\
& \leq c_3 \left(N^{M/2} + N^M \sup_{x \in \cK^{\langle M \rangle}} \mathbf{P}^{x}\left[ \sup_{0<s\leq t/2} |x-X_s|>c_{1} L^{M/2}\right]\right).
\end{align*}
Therefore
\begin{equation*}
R_{1,M}(t) \leq c_3 \left(N^{-M/2} + \sup_{x \in \cK^{\langle M \rangle}} \mathbf{P}^{x}\left[ \sup_{0<s\leq t/2} |x-X_s|>c_{1} L^{M/2}\right]\right),
\end{equation*}
which is, by Lemma \ref{lem:supseries}, a term of a convergent series. The proof of (a) is completed.

The proof of the first convegence in (b) follows the lines of that of (a) and it is omitted. We only show the second convergence.

First observe that
$$
\left|\mathbb{E}_{\mathbb{Q}} \left( \Lambda_{M}^{D,  V^\omega} (t) - \Lambda_{M}^{D,  V_M^\omega} (t)\right) \right|
 \leq \frac{1}{\mu\left(\cK^{\langle M\rangle}\right)} \int_{\cK^{\langle M \rangle}} p(t,x,x)\mathbf{E}^{x,x}_{t} \left[ \left|F(t,M)\right| ; t<\tau_{\cK^{\langle M \rangle}}\right] \mu({\rm d}x),
$$
where
\begin{align*}
F(t,M) = \mathbb{E}_{\mathbb{Q}} \bigg( \exp \bigg( -\int_0^t \sum_{y \in \cV_{0}^{\langle \infty \rangle}} & \xi_y(\omega) W\left(X_s, y \right)  \ {\rm d}s \bigg) \\ & - \exp \bigg( -\int_0^t \sum_{y \in \cV_{0}^{\langle M\rangle}}\xi_y(\omega) \sum_{y' \in \pi_{M}^{-1}(y)} W\left(X_s, y' \right) \ {\rm d}s \bigg) \bigg) .
\end{align*}
By using exactly the same notation and argument as in the proof of part (a), we may write
\begin{align*}
\frac{1}{\mu\left(\cK^{\langle M\rangle}\right)} & \int_{\cK^{\langle M \rangle}} p(t,x,x) \mathbf{E}^{x,x}_{t} \left[ \left|F(t,M)\right| ; t<\tau_{\cK^{\langle M \rangle}}\right] \mu({\rm d}x) \\
& \leq \frac{c_4 \mu \left(\cK^{\langle M\rangle} \backslash \cD_M \right)}{\mu\left(\cK^{\langle M\rangle}\right)}
+ \frac{1}{\mu\left(\cK^{\langle M\rangle}\right)} \int_{\cD_M} p(t,x,x)\mathbf{E}^{x,x}_{t} \left[ \left|F(t,M)\right| ; t<\tau_{\cK^{\langle M \rangle}}\right] \mu({\rm d}x) \\
& = c_4 \frac{k N^{\left\lfloor M/2\right\rfloor}} {N^M} + \frac{1}{\mu\left(\cK^{\langle M\rangle}\right)} \int_{\cD_M} p(t,x,x)\mathbf{E}^{x,x}_{t} \left[ \left|F(t,M)\right| ; t<\tau_{\cK^{\langle M \rangle}}\right] \mu({\rm d}x).
\end{align*}
We see that the first term converges to 0 as $M \to \infty$. It is sufficient to show that the second term goes to zero as well; denote it by $I(t,M)$. We have
\begin{align*}
I(t,M) & \leq \frac{1}{\mu\left(\cK^{\langle M\rangle}\right)} \int_{\cD_M} p(t,x,x)\mathbf{E}^{x,x}_{t} \left[ \left|F(t,M)\right| ; t<\tau_{\cD_M }\right] \mu({\rm d}x) \\
& \ \ \ \ \ \ \ \ \ \ \ \ \ \  + \frac{1}{\mu\left(\cK^{\langle M\rangle}\right)} \int_{\cD_M} p(t,x,x)\mathbf{P}^{x,x}_{t} \left[ t\geq \tau_{\cD_M } \right] \mu({\rm d}x) \\
& =: I_1(t,M) + I_2(t,M).
\end{align*}
To estimate $I_1(t,M)$ we will use the inequality $|{\rm e}^{-x}-{\rm e}^{-y}| \leq |x-y|, x,y>0$, the fact that $W$ and $\xi_y$ are nonnegative.
\begin{align*}
&\left|F(t,M)\right| \\
& \leq \mathbb{E}_{\mathbb{Q}} \left| \exp \left( -\int_0^t \sum_{y \in \cV_{0}^{\langle \infty \rangle}} \xi_y(\omega) W\left(X_s, y \right)  \ {\rm d}s \right) - \exp \left( -\int_0^t \sum_{y \in \cV_{0}^{\langle M\rangle}}\xi_y(\omega) \sum_{y' \in \pi_{M}^{-1}(y)} W\left(X_s, y' \right) \ {\rm d}s \right)\right| \\
& \leq \mathbb{E}_{\mathbb{Q}} \left| \int_0^t \sum_{y \in \cV_{0}^{\langle \infty \rangle}} \xi_y(\omega) W\left(X_s, y \right)  \ {\rm d}s - \int_0^t \sum_{y \in \cV_{0}^{\langle M\rangle}}\xi_y(\omega) \sum_{y' \in \pi_{M}^{-1}(y)} W\left(X_s, y' \right) \ {\rm d}s \right|.
\end{align*}
Now, observe that for every $x \in \cK^{\langle \infty \rangle}$ we have
$$
\sum_{y \in \cV_{0}^{\langle \infty \rangle}} \xi_y(\omega) W\left(x, y \right)
= \sum_{y \in \cV_{0}^{\langle M \rangle}} \xi_y(\omega) W\left(x, y \right) + \sum_{y \in \cV_{0}^{\langle \infty \rangle} \setminus \cV_{0}^{\langle M \rangle}} \xi_y(\omega) W\left(x, y \right)
$$
and
$$
\sum_{y \in \cV_{0}^{\langle M\rangle}} \sum_{y' \in \pi_{M}^{-1}(y)} \xi_{y'}(\omega) W\left(x, y' \right)
= \sum_{y \in \cV_{0}^{\langle M\rangle}} \xi_{y}(\omega) W\left(x, y \right) + \sum_{y \in \cV_{0}^{\langle M\rangle}} \sum_{y' \in \pi_{M}^{-1}(y) \backslash \{y\}} \xi_{y'}(\omega) W\left(x, y' \right).
$$
By this observation, the Fubini theorem and the fact that all  the lattice random variables together with the profile function $W$ are nonnegative, we get that the above expectation can be estimated above by
\begin{align*}
\mathbb{E}_{\mathbb{Q}}\int_0^t \left(\sum_{y \in \cV_{0}^{\langle \infty \rangle} \setminus \cV_{0}^{\langle M \rangle}} \xi_y(\omega) W\left(X_s, y \right)
           \right. & \left.+ \sum_{y \in \cV_{0}^{\langle M\rangle}} \sum_{y' \in \pi_{M}^{-1}(y) \backslash \{y\}} \xi_{y'}(\omega) W\left(X_s, y' \right)\right){\rm d}s \\
& \leq
2 \mathbb{E_Q}\xi \int_{0}^{t} \sum_{y \in \cV_{0}^{\langle \infty \rangle} \backslash  \cV_{0}^{\langle M \rangle}} W\left(X_s, y \right) {\rm d}s.
\end{align*}
As $X_s \in \cD_M$ for $0\leq s \leq t$ and $y \in \cV_{0}^{\langle \infty \rangle} \backslash  \cV_{0}^{\langle M \rangle}$, we have $d_{\left\lfloor M/2 \right\rfloor}(X_s, y)>2$ for $0\leq s\leq t$. This gives
\begin{align*}
I_1(t,M) &
\leq \frac{2\kappa}{\mu\left(\cK^{\langle M\rangle}\right)} \int_{\cD_M} p(t,x,x)
\mathbf{E}^{x,x}_{t} \left[ \int_0^t \sum_{y \in \cV_{0}^{\langle \infty \rangle} \backslash  \cV_{0}^{\langle M \rangle}} W(X_s(\omega), y) {\rm d}s ; t<\tau_{\cD_M }\right] \mu({\rm d}x)\\
& \leq c_5 \sup_{z \in \cK^{\langle \infty \rangle}} \sum_{y \in \cV_0^{\langle \infty \rangle} \backslash B_{\lfloor M/2 \rfloor}(z,1)}  W(z,y),
\end{align*}
with a constant $c_5>0$ independent of $M$. By \textbf{(W3)} this is a term of a convergent series.

To show that $I_2(t,M)$ is a term of a convergent series, we just follow the steps from the proof of (a) for an appropriate set $\cD^{'}_M \subset \cD_M$. The proof of (b) is finished.
\end{proof}

\begin{proof}[Proof of Lemma \ref{lem:variances}] The proof of this lemma follows the main steps from the proof of \cite[Lemma 3.2]{bib:KaPP2}, but it substantially differs in some details which are critical for the argument. We will focus on explanation of these differences.

First note that due to Lemma \ref{lem:commonlimit}(a), similarly as in the proof of \cite[Lemma 3.2]{bib:KaPP2}, we only need to establish \eqref{eq:dvarlemma}.
We consider a family of measures $(\nu_M)_{M \in \Z}$ given by
\begin{equation}
\nu_M := \left( \frac{1}{\mu\left(\cK^{\langle M\rangle}\right)} \int_{\cK^{\langle M\rangle}} p(t,x,x) \mathbf{P}_{t}^{x,x} \mu(dx) \right)^{\otimes 2} \otimes \mathbb{Q}^{\otimes 3}, \quad M \in \mathbb{Z}_{+},
\end{equation}
defined with the product space $\widetilde{\Omega} = D([0,t],\cK^{\langle \infty \rangle})^2 \times \Omega^3$, and nondecreasing sequence of positive integers $\left(a_M\right)_{M \in \mathbb{Z}_{+}}$ such that $a_M \leq M$, $M \in \mathbb{Z}_{+}$. Its values will be chosen later in the proof.

For $m \in \mathbb{Z}_{+}$ we set
\begin{equation}
V^{\omega,m}(x):= \sum_{y \in \cV_{0}^{\langle \infty \rangle} \cap B_m(x,1)} \xi_{y}(\omega) W(x,y)
\end{equation}
and
\begin{equation}
\widetilde{V}^{\omega,m}(x):= \sum_{y \in \cV_{0}^{\langle \infty \rangle} \setminus B_m(x,1)} \xi_{y}(\omega) W(x,y) ,
\end{equation}
where the ball $B_m(x,1)$ is taken in the $m$-graph metric, i.e.\ for $x \in \cK^{\langle \infty \rangle} \backslash \mathcal V^{\langle \infty \rangle}_m$ it is equal to $\Delta_m(x)$ -- the only $m$-complex containing $x$; for $x \in \mathcal V^{\langle \infty \rangle}_m$ it is a sum of $m$-complexes attached to $x$ (there are $\textrm{rank}(x) \in \{1,2,3\}$ of them).
We also denote for $M \in \mathbb{Z}_{+}$
\begin{equation*}
F_M(w,\omega):= {\rm e}^{-\int_0^t V^{\omega,a_M}(X_s(w)) {\rm d}s} \quad \text{and} \quad
\widetilde{F}_M(w,\omega):= {\rm e}^{-\int_0^t \widetilde{V}^{\omega,a_M}(X_s(w)) {\rm d}s}.
\end{equation*}
With this notation we have
\begin{align*}
\mathbb{E}_{\mathbb{Q}} \left[ \Lambda_M^D - \mathbb{E}_{\mathbb{Q}} \Lambda_M^D  \right]^2
& = \int_{\widetilde{\Omega}} \prod_{i=1}^{2} \left(F_M(w_i, \omega_0) \widetilde{F}_M(w_i, \omega_0) - F_M(w_i, \omega_i) \widetilde{F}_M(w_i, \omega_i) \right) \\
& \ \ \ \ \ \ \ \ \ \ \ \ \ \ \times \mathbf{1}_{\{t<\tau_{\cK^{\langle M \rangle}}(w_i)\}} \cdot {\rm d}\nu_M (w_1,w_2,\omega_0,\omega_1,\omega_2) \\
& =: \int_{\widetilde{\Omega}} \cX(w_1,w_2,\omega_0,\omega_1,\omega_2) {\rm d}\nu_M (w_1,w_2,\omega_0,\omega_1,\omega_2).
\end{align*}
We now consider the partition of $\widetilde{\Omega}$ into three disjoint sets:
\begin{align*}
D_0^M & := \left\{(w_1,w_2) \in D([0,t],\cK^{\langle \infty\rangle})^2 : \textrm{for every } s_1,s_2 \in [0,t] \  d_{a_M} \left(X_{s_1}(w_1),X_{s_2}(w_2)\right) >2  \right\} \times \Omega^3,\\
D_1^M & := \bigg\{(w_1,w_2) \in D([0,t],\cK^{\langle \infty\rangle})^2 : d_{a_M+3}\left(X_{0}(w_1),X_{0}(w_2)\right) >2 \textrm{ and there exist } s_1,s_2 \in (0,t] \\
      & \ \ \ \ \ \textrm{ such that } d_{a_M}\left(X_{s_1}(w_1),X_{s_2}(w_2)\right) \leq 2  \bigg\} \times \Omega^3,\\
D_2^M & := \widetilde{\Omega} \backslash \left(D_0^M \cup D_1^M\right).
\end{align*}
We will integrate $\cX$ over each of these sets separately.
Let us point out that all these sets are now defined with the $m$-graph metric $d_m(x,y)$. We also use this opportunity to correct the definition of the sets
$D_0^M$ and $D_1^M$ on p.\ 1272 in \cite{bib:KaPP2}: they also should be defined with $s_1,s_2 \in [0,t]$ as above instead of single $s \in [0,t]$. Also, $a_M$ should be used in the definition of $D_0^M$ instead of $c_M$.

By following the argument in the proof of the quoted lemma, we get
\begin{multline*}
\prod_{i=1}^{2} \left(F_M(w_i,\omega_0) \widetilde{F}_M(w_i, \omega_0) - F_M(w_i,\omega_i) \widetilde{F}_M(w_i, \omega_i) \right)\\
\leq \prod_{i=1}^{2} \left(F_M(w_i,\omega_0) - F_M(w_i,\omega_i) \right) \\
+ 2 - \left(\widetilde{F}_M(w_1, \omega_0) \widetilde{F}_M(w_2, \omega_2) + \widetilde{F}_M(w_2, \omega_0) \widetilde{F}_M(w_1, \omega_1) \right).
\end{multline*}
For a given $M \in \mathbb{Z}_{+}$ and a path $X_s(w)$, the functional $F_M(w,\cdot)$ depends only on those fractal lattice points from $V_{0}^{\langle \infty \rangle}$ which are in the set $X_{[0,t]}^{a_M}(w):= \bigcup_{s\in [0,t]} B_{a_M}(X_s(w),1)$. From the definition of $D_0^M$ we see that on this set we have $X_{[0,t]}^{a_M}(w_1) \cap X_{[0,t]}^{a_M}(w_2) = \emptyset$ and therefore, the random variables $F_M(w_1,\omega_0) - F_M(w_1,\omega_1)$ and $F_M(w_2,\omega_0) - F_M(w_2,\omega_2)$ are $\mathbb{Q}^{\otimes 3}$-independent. In consequence,
\begin{equation*}
\int_{D_0^M} \prod_{i=1}^{2} \left(F_M(w_i,\omega_0) - F_M(w_i,\omega_i) \right) \mathbf{1}_{\{t<\tau_{\cK^{\langle M \rangle}}(w_i)\}} {\rm d}\nu_M (w_1,w_2,\omega_0,\omega_1,\omega_2) =0
\end{equation*}
and, by following the argument in the proof of the cited lemma, including Jensen's inequality and the assumption that all lattice random variables are nonnegative and integrable, we obtain
\begin{equation}
\label{eq:intd0m}
\int_{D_0^M} \cX \ {\rm d}\nu_M \leq c t \, \mathbb{E_Q}\xi \sup_{x \in \cK^{\langle \infty \rangle}} \sum_{y \notin B_{a_M}(x,1)} W(x,y),
\end{equation}
for a constant $c>0$, independent of $M$.

On the set $D_1^M$ we have
\begin{equation}
\label{eq:d1mcond}
\sup_{s \in (0,t]} d_{a_M} \left(X_0(w_i), X_s(w_i) \right) >2, \quad \textrm{for }i=1 \textrm{ or } i=2.
\end{equation}
This can be seen as follows (see Figure \ref{fig:pentvar} for illustration). From the definition of $D_1^M$ we have that $d_{a_M+3} \left(X_0(w_1), X_0(w_2) \right) >2$, what means that $X_0(w_1)$ and $X_0(w_2)$ are in separate $(a_M+3)$-complexes (light grey complexes in the figure). If, on the contrary to \eqref{eq:d1mcond}, for both $i \in \{1,2\}$ and all $s \in (0,t]$ were $d_{a_M} \left(X_0(w_i), X_s(w_i) \right) \leq 2$, then for both $i=1,2$ the entire path $X_{s}(w_i)$ would be bounded inside $B_{a_M}(X_0(w_i),2)$ (dark grey complexes in the figure). That would mean that these two paths of the process up to time $t$  are not closer to each other than in two different $a_M$-complexes attached to the vertices of a common $(a_M+3)$-complex (white in the figure). In fact, they would be in separate $(a_M+1)$-complexes inside two different $(a_M+2)$-complexes. This would mean that $d_{a_M} \left(X_{s_1}(w_1), X_{s_2}(w_2) \right) \geq 6$ for all $s_1, s_2 \in (0,t]$, because the path  realizing  the graph distance must pass through another $a_M$-complex inside $(a_M+1)$-complex containing $X_{s_1}(w_1)$, then through another $(a_M+1)$-complex inside $(a_M+2)$-complex containing $X_{s_1}(w_1)$, so at least two more $a_M$-complexes. By symmetry, it must go next through at least three $a_M$-complexes inside $(a_M+2)$-complex containing $X_{s_2}(w_2)$.

\begin{figure}[ht]
\centering
\includegraphics[scale=0.1]{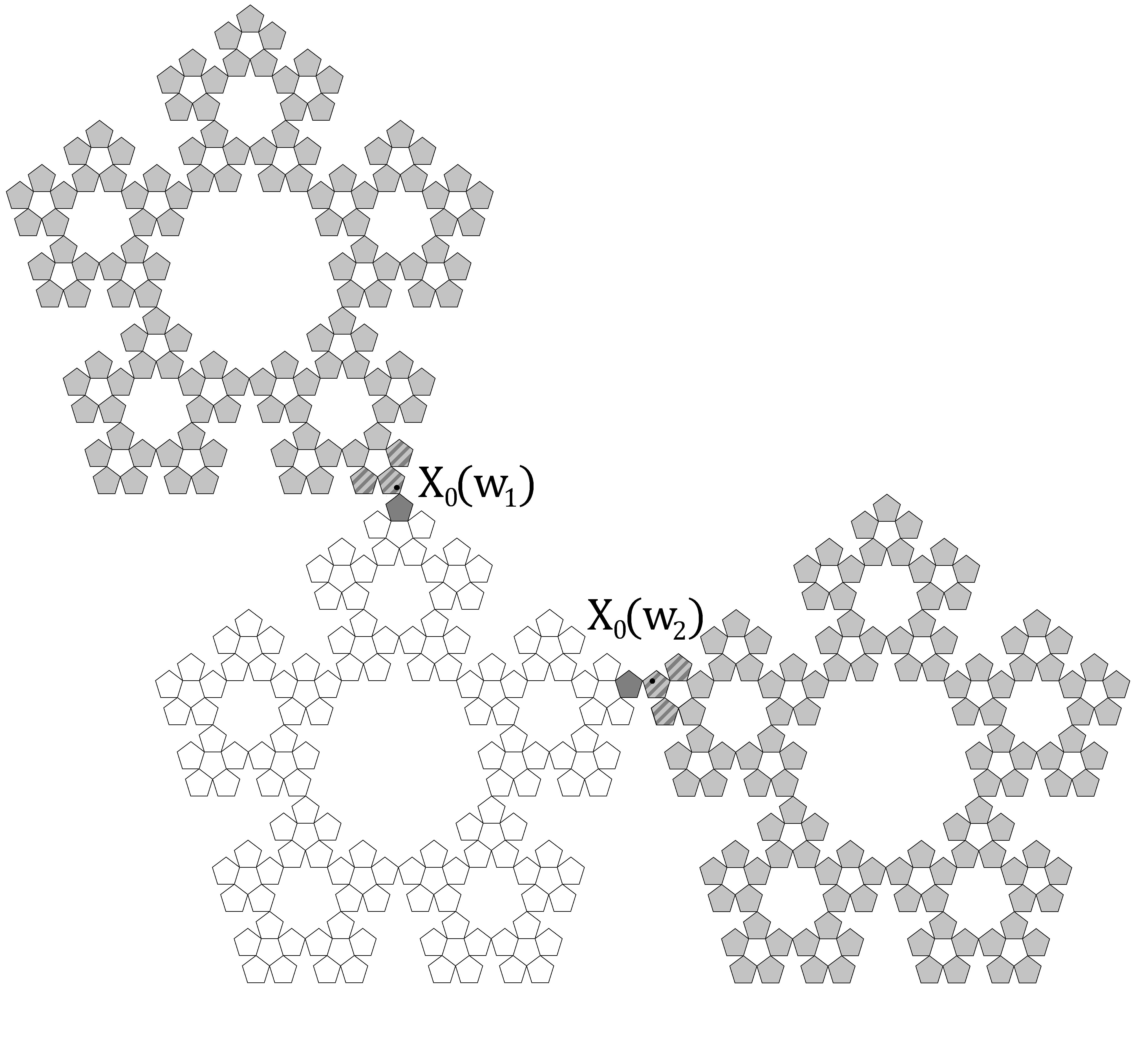}
\caption{$X_0(w_i)$ are in separable $(a_M+3)$-complexes (light grey). If \eqref{eq:d1mcond} does not hold, then the entire paths $X_s(w_i)$, $s \in [0,t]$, must be restricted to the respective dark grey $a_M$-complexes, contradicting the definition of $D_1^M$.}
\label{fig:pentvar}
\end{figure}

By \cite[Lemma A.2]{bib:KOP} this implies that
\begin{equation*}
\sup_{s \in (0,t]} |X_0(w_i), X_s(w_i)| \geq c_1 L^{a_M }, \quad \textrm{for }i=1 \textrm{ or } i=2,
\end{equation*}
with a constant $c_1$ independent of $M$,
and since $0\leq F_M(w_i,\omega_k) \leq 1$, $i=1,2$, $k=0,1,2$, the integral over $D_1^M$ can be estimated by
\begin{equation}
\frac{1}{\mu\left(\cK^{\langle M\rangle}\right)} \int_{\cK^{\langle M\rangle}} p(t,x,x) \mathbf{P}_{t}^{x,x} \left[ \sup_{s \in (0,t]} |X_0(w_i), X_s(w_i)| \geq c_1 L^{a_M } \right] \mu({\rm d}x).
\end{equation}
Using the same argument as in the same step of the proof of the cited lemma, we then get that the above expression is less than or equal to
\begin{multline}
\label{eq:intd1m}
c_2 \sup_{x \in \cK^{\langle \infty \rangle}} \mathbf{P}^x \left[ \sup_{s \in (0,t/2]} |X_0(w_i), X_s(w_i)| \geq c_1 L^{a_M } \right],
\end{multline}
with some $c_2$, independent of $M$.

Since the integrand $\cX$ is not bigger than 1, it is enough to estimate the measure of $D_2^M$.
We have
\begin{equation}
\label{eq:intd2m}
\nu_M(D_2^M)  \leq \frac{ c_3 \mu\{(x,y) \in \cK^{\langle M \rangle} \times \cK^{\langle M \rangle} : d_{a_M+3}\left(x,y\right) \leq 2 \}}{\left(\mu(\cK^{\langle M \rangle})\right)^2} \leq \frac{ 4 c_3 N^{a_M+3}}{N^{M}},
\end{equation}
with some $c_3$ independent of $M$.

We may choose $a_M=\lfloor M/4 \rfloor$. Then \eqref{eq:intd0m} is a term of a convergent series by \textbf{(W3)} and \eqref{eq:intd1m} is a term of convergent series by Lemma \ref{lem:supseries}.
\end{proof}

\end{document}